\documentclass[a4paper]{article}
\usepackage{amsmath,amsthm,amsfonts,amssymb,amscd,bm,makeidx,indentfirst,tocloft}

\newcommand{\e}{\ensuremath{\epsilon}}
\newcommand{\ve}{\ensuremath{\varepsilon}}
\newcommand{\rto}{\ensuremath{\rightarrow}}
\newcommand{\lem}{\ensuremath{\lesssim}}

\newtheorem{theorem}{Theorem}[section]
\newtheorem{definition}[theorem]{Definition}
\newtheorem{lemma}[theorem]{Lemma}

\newtheorem{remark}[theorem]{Remark}

\numberwithin{equation}{section}

\title{\Large Global Existence and Nonlinear Diffusion of Classical Solutions to Non-Isentropic Euler Equations with Damping in Bounded Domain}
\author{\normalsize Fuzhou Wu\thanks{E-mail: michael8723@gmail.com; fuzhou.wu@yahoo.com} \\
\small\it  Mathematical Sciences Center, Tsinghua University\\
\small\it  Beijing 100084, China
}
\date{\normalsize }

\begin{document}
\maketitle
\setlength\parindent{2em}
\setlength\parskip{5pt}

\begin{abstract}
\normalsize{
We considered classical solutions to the initial boundary value problem for non-isentropic compressible Euler equations with damping in multi-dimensions. We obtained global a priori estimates and global existence results of classical solutions to both non-isentropic Euler equations with damping and their nonlinear diffusion equations under small data assumption.
We proved the pressure and velocity decay exponentially to constants, while the entropy and density can not approach constants. Finally, we proved the pressure and velocity of the non-isentropic Euler equations with damping converge exponentially to those of their nonlinear diffusion equations when the time goes to infinity.
}
\end{abstract}

\tableofcontents

\section{Introduction}
In this paper, we consider classical solutions to IBVP for non-isentropic compressible Euler equations with damping in three dimensions:
\begin{equation}\label{Sect1_NonIsentropic_EulerEq_Original}
\left\{\begin{array}{lll}
\varrho_t + u\cdot\nabla \varrho + \varrho\nabla\cdot u =0, \\[6pt]
\varrho u_t + \varrho u\cdot\nabla u + \nabla p + a\varrho u =0, \\[6pt]
S_t + u\cdot\nabla S =0, \\[6pt]
(\varrho,u,S)(x,0)=(\varrho_0(x), u_0(x), S_0(x)), \\[6pt]
u\cdot n|_{\partial\Omega} =0,\ \forall t\geq 0,
\end{array}\right.
\end{equation}
where $\varrho, u, S, p$ denotes the density, velocity, entropy and pressure of ideal gases, respectively.
The friction coefficient $a>0$, $\Omega\subset\mathbb{R}^3$ is a bounded domain with smooth boundary $\partial\Omega$. The physical model of the equations $(\ref{Sect1_NonIsentropic_EulerEq_Original})$ is the non-isentropic flow of the ideal gases in porous media, for which the pressure law reads
\begin{equation}\label{Sect1_Pressure}
p = A\varrho^{\gamma}e^S,
\end{equation}
where $A>0$, $\gamma=\frac{C_p}{C_V}>1$ are constants.

As long as $(\varrho,p,v,S)$ in $(\ref{Sect1_NonIsentropic_EulerEq_Original})$ remain classical, IBVP $(\ref{Sect1_NonIsentropic_EulerEq_Original})$ are equivalent to the following IBVP, where the first two equations can be symmetrized.
\begin{equation}\label{Sect1_NonIsentropic_EulerEq}
\left\{\begin{array}{lll}
p_t + u\cdot\nabla p + \gamma p\nabla\cdot u = 0, \\[6pt]
u_t + u\cdot\nabla u + \frac{1}{\varrho}\nabla p + a u =0, \\[6pt]
S_t + u\cdot\nabla S =0, \\[6pt]
(p,u,S)(x,0)=(p_0(x), u_0(x), S_0(x)), \\[6pt]
u\cdot n|_{\partial\Omega} =0,\ \forall t\geq 0,
\end{array}\right.
\end{equation}
where $\varrho = \varrho(p,S)
:= \frac{1}{\sqrt[\gamma]{A}}p^{\frac{1}{\gamma}}\exp\{-\frac{S}{\gamma}\}$.

There is a huge literature about the compressible Euler equations with damping, we introduce these results as follows:

\par
As to the isothermal compressible Euler equations with damping:
\begin{equation}\label{Sect1_Isothermal_EulerEq}
\left\{\begin{array}{ll}
\varrho_t + \nabla\cdot(\varrho u)=0, \\[6pt]
\varrho u_t + \varrho u\cdot\nabla u+ \bar{\sigma}^2\nabla \varrho + a\varrho u = 0,
\end{array}\right.
\end{equation}
where $\bar{\sigma}^2 = \mathcal{R}\theta_{\ast}$ is constant. 
The equations $(\ref{Sect1_Isothermal_EulerEq})$ describe the isothermal flow of ideal gases in porous media.
Zhao (see \cite{Zhao_2010}) proved the global existence of classical solutions to IBVP for $(\ref{Sect1_Isothermal_EulerEq})$ with small data. For BV solutions, see \cite{Dafermos_1995,Luskin_Temple_1982}. For entropy weak solutions, see \cite{Huang_Pan_2006,Zhao_2010}.

\par
As to the isentropic compressible Euler equations with damping:
\begin{equation}\label{Sect1_Isentropic_EulerEq}
\left\{\begin{array}{ll}
\varrho_t + \nabla\cdot(\varrho u)=0, \\[6pt]
\varrho u_t + \varrho u\cdot\nabla u +\nabla p + a \varrho u = 0,
\end{array}\right.
\end{equation}
with $p(\rho)=A\rho^{\gamma}$,
Sideris, Thomases and Wang (see \cite{Sideris_Thomases_Wang_2003}) proved the global existence of classical solutions to 3D Cauchy problem for $(\ref{Sect1_Isentropic_EulerEq})$ under small data assumption. They also proved the singularity formation of classical solutions for a class of large data. Pan and Zhao (see \cite{Pan_Zhao_2009}) proved the global existence and exponential decay of classical solutions to 3D IBVP for $(\ref{Sect1_Isentropic_EulerEq})$ under small data assumption, verified the Darcy law when the total mass of the diffusion equations equals the total mass of IBVP $(\ref{Sect1_Isentropic_EulerEq})$. Due to the boundary conditions $\partial_t^{\ell}u\cdot n|_{\partial\Omega}=0$ but $\mathcal{D}^{\alpha}u\cdot n|_{\partial\Omega}$ may not be zero, the a priori estimates for IBVP (see \cite{Pan_Zhao_2009}) are more complicated than those for Cauchy problem (see \cite{Sideris_Thomases_Wang_2003}).

All the variables in the isothermal case $(\ref{Sect1_Isothermal_EulerEq})$ and the isentropic case $(\ref{Sect1_Isentropic_EulerEq})$ have diffusion property, which approach constants when the time goes to infinity. While the entropy and density of the non-isentropic Euler equations with damping $(\ref{Sect1_NonIsentropic_EulerEq_Original})$ are transported in Eulerian coordinates, which bring main difficulties for the non-isentropic Euler equations with damping.

In (\cite{Hsiao_Luo_1996},\cite{Hsiao_Pan_1999},\cite{Hsiao_Pan_2000},\cite{Hsiao_Serre_1996},
\cite{Marcati_Pan_2001},\cite{Pan__2001},\cite{Pan_Darcy_2006},\cite{Zheng_1996}), the authors applied characteristics analysis together with energy estimate method to study the 1D non-isentropic p-system with damping in Lagrangian coordinates $\{(y,t)\}$:
\begin{equation}\label{Sect1_P_System}
\left\{\begin{array}{ll}
\mathcal{V}_t - u_y =0,\\[6pt]
u_t + p(\mathcal{V},S)_y = - a u,\\[6pt]
S_t =0,
\end{array}\right.
\end{equation}
where $\mathcal{V}=\frac{1}{\varrho},\ p(\mathcal{V},S) = A\mathcal{V}^{-\gamma}e^S$.
While in Lagrangian coordinates, the entropy $S(y,t)\equiv S_0(y)$, whose transportation is implicit in this coordinates. All vertical lines in Lagrangian coordinates are particle paths of 1D non-isentropic p-system with damping and its diffusion system, so that the phenomena in 1D Lagrangian coordinates are much simpler.

As to the non-isentropic Euler equations with damping in multi-dimensional Eulerian coordinates, the only results in the present are the global existence and decay properties of classical solutions to $(\ref{Sect1_NonIsentropic_EulerEq})$ in $\mathbb{R}^3$ (see \cite{Wu_Tan_Huang_2013}) and periodic domain $\mathbb{T}\subset\mathbb{R}^3$ (see \cite{Zhang_Wu_2014}).
The spectral method and Duhamel's principle are applied in \cite{Wu_Tan_Huang_2013} to prove $p-\bar{p},u, S_t$ algebraically decay and $S-\bar{S}$ is uniformly bounded. Due to the convenience of periodic boundary condition, similar energy estimate method was applied in \cite{Zhang_Wu_2014}, where $p-\bar{p},u,S_t$ decay exponentially and $S-\bar{S}$ is uniformly bounded. While the initial boundary value problem is more difficult, due to the boundary conditions $\partial_t^{\ell}u\cdot n|_{\partial\Omega}=0$ but $\mathcal{D}^{\alpha}u\cdot n|_{\partial\Omega}$ may not be zero.

The aims of this paper are as follows: (1) to study the long time behavior of classical solutions to the non-isentropic Euler equations with damping, such as global existence, exponential decay, equilibrium states, singularity formation. (2) to study the long time behavior of classical solutions to the nonlinear diffusion equations. (3) to study the relationship between the solutions of the above two systems when the time is large.

In this paper, we assume no vacuum initially, i.e., $\inf\limits_{x\in\Omega}\varrho_0>0$ or $\inf\limits_{x\in\Omega}p_0>0$, otherwise the degeneracy aroused by the vacuum brings about new difficulties, such as local existence and behavior of vacuum boundary. Then $\inf\limits_{\Omega\times[0,T]}\varrho(x,t)>0$ and $\inf\limits_{\Omega\times[0,T]}p(x,t)>0$ as long as the solution remains classical in the time interval $[0,T]$.

We introduce the following constants:
\begin{equation}\label{Sect1_Constants_EulerEq}
\begin{array}{ll}
\bar{p}= \left( \frac{1}{|\Omega|}\int\limits_{\Omega} p_0^{\frac{1}{\gamma}}\,\mathrm{d}x \right)^{\gamma},\quad
\bar{S} = \frac{1}{\Omega}\int\limits_{\Omega} S_0 \,\mathrm{d}x, \quad
\bar{\varrho} = \frac{1}{\Omega}\int\limits_{\Omega} \varrho_0 \,\mathrm{d}x,
\end{array}
\end{equation}
where $\varrho_0 = \frac{1}{\sqrt[\gamma]{A}}p_0^{\frac{1}{\gamma}}\exp\{-\frac{S_0}{\gamma}\}$. Thus, we can express the concept of small data for IBVP $(\ref{Sect1_NonIsentropic_EulerEq})$, i.e., the smallness of $\|(p_0-\bar{p},u_0-0,S_0-\bar{S},\varrho_0-\bar{\varrho})\|_{H^3(\Omega)}$.

We proved that if the initial data $(p_0,u_0,S_0)\in H^3(\Omega)$ are sufficiently small perturbations of their mean values $(\frac{1}{|\Omega|}\int\limits_{\Omega} p_0 \,\mathrm{d}x,0,\bar{S})$ or $(\bar{p},0,\bar{S})$, then IBVP $(\ref{Sect1_NonIsentropic_EulerEq})$ admits a unique global classical solution
$(p,u,S)\in \underset{0\leq \ell\leq 3}{\cap}C^{\ell}([0,+\infty),H^{3-\ell}(\Omega))$,
moreover, $\varrho = \varrho(p,S)\in \underset{0\leq \ell\leq 3}{\cap}C^{\ell}([0,+\infty),H^{3-\ell}(\Omega))$.
$(p,u)$ converge exponentially to $(\bar{p},0)$ rather than $(\frac{1}{|\Omega|}\int\limits_{\Omega} p_0 \,\mathrm{d}x,0)$ as $t\rto +\infty$, $(\varrho,S)$ are uniformly bounded all the time.
Moreover, $\sum\limits_{0\leq\ell\leq 3}(\|\partial_t^{\ell}(p-\bar{p})\|_{H^{3-\ell}(\Omega)}
+\|\partial_t^{\ell}u\|_{H^{3-\ell}(\Omega)})$,
$\sum\limits_{1\leq\ell\leq 3}\|\partial_t^{\ell}\varrho\|_{H^{3-\ell}(\Omega)}$ and
$\sum\limits_{1\leq\ell\leq 3}\|\partial_t^{\ell}S\|_{H^{3-\ell}(\Omega)}$ decay exponentially,
$\|\varrho-\bar{\varrho}\|_{H^{3}(\Omega)}$ and $\|S-\bar{S}\|_{H^{3}(\Omega)}$ are uniformly bounded.

Since $\bar{p} \leq \frac{1}{|\Omega|}\int\limits_{\Omega} p_0\,\mathrm{d}x$,
$\|p_0-\bar{p}\|_{H^\ell(\Omega)}\geq \|p_0- \frac{1}{|\Omega|}\int\limits_{\Omega} p_0\,\mathrm{d}s\|_{H^{\ell}(\Omega)},\ell\geq 0$.
Thus for $p_0$, the smallness of $p_0-\bar{p}$ implies the smallness of $p_0- \frac{1}{|\Omega|}\int\limits_{\Omega} p_0\,\mathrm{d}x$. Therefore, even $(p_0,u_0,S_0)$ are small perturbations of $(\frac{1}{|\Omega|}\int\limits_{\Omega} p_0\,\mathrm{d}x,0,\bar{S})$, the pressure $p$ still converges to $\bar{p}$ as $t\rto +\infty$.

\vspace{0.3cm}
In order to describe the equilibrium states of the global classical solutions, we introduce the following notations:
\begin{equation}\label{Sect1_Variables_Infinity}
\begin{array}{ll}
(p_{\infty}(x),u_{\infty}(x),S_{\infty}(x),\varrho_{\infty}(x))
= \lim\limits_{t\rto\infty} (p(x,t),u(x,t),S(x,t),\varrho(x,t)).
\end{array}
\end{equation}

We define $S_{+} := \max\limits_{x\in\Omega}\{S_0(x)\},\ S_{-} := \min\limits_{x\in\Omega}\{S_0(x)\}$. Due to the characteristic boundary $u\cdot n|_{\partial\Omega}=0$, each particle path in $\Omega\times \{t\geq 0\}$
extends to $\Omega\times \{t=+\infty\}$ rather than terminating on $\partial\Omega\times \{t\geq 0\}$, and
$S$ is invariant along every particle path, so $\max\limits_{x\in\Omega}\{S_{\infty}(x)\}=S_{+},\ \min\limits_{x\in\Omega}\{S_{\infty}(x)\}=S_{-}$. This is a physical explanation of the transportation of the entropy, but we proved mathematically that $S_{\infty}\neq const$ and $\varrho_{\infty}\neq const$, if $S_{+}\neq S_{-}$.
Moreover, $(p,u,S,\varrho)$ converge exponentially to their equilibrium states
$(\bar{p},0,S_{\infty}(x),\varrho_{\infty}(x))$ in $|\cdot|_{\infty}$ norm.

However, the damping effect on the velocity makes the equations $(\ref{Sect1_NonIsentropic_EulerEq_Original})$ or $(\ref{Sect1_NonIsentropic_EulerEq})$ weakly dissipative, such that it
can not prevent the formation of singularities without small data assumption. We proved that for a class of large initial data whose support $Supp(p_0-\bar{p},u_0,S_0-\bar{S})$ is away from the boundary $\partial\Omega$, the singularities must form in the interior of ideal gases. These singularities will have formed before $Supp(p-\bar{p},u,S-\bar{S})$ reaches the boundary. Our argument is based on the analysis of the moment $M_{\varrho}(t) = \int\limits_{\Omega} \varrho u\cdot x\,\mathrm{d}x$ and finite propagation speed of the classical solutions, this method can be extended easily to Cauchy problem. However, the finite size of bounded domain $\Omega$ can not replace the finite propagation speed of the solutions in our proof.

Toward a better understanding of the large time behavior and nonlinear diffusion property of classical solutions to non-isentropic Euler equations with damping $(\ref{Sect1_NonIsentropic_EulerEq})$, we study the following nonlinear diffusion equations which are obtained by applying Darcy's law to $(\ref{Sect1_NonIsentropic_EulerEq})_2$,
\begin{equation}\label{Sect1_Diffusion_Eq}
\left\{\begin{array}{lll}
p_t + u\cdot\nabla p + \gamma p\nabla\cdot u = 0, \\[6pt]
\frac{1}{\varrho}\nabla p + a u =0, \\[6pt]
S_t + u\cdot\nabla S =0, \\[6pt]
(p,S)(x,0)=(\hat{p}_0(x), \hat{S}_0(x)), \\[6pt]
u\cdot n|_{\partial\Omega} =0,\ \forall t\geq 0,
\end{array}\right.
\end{equation}
where $\varrho = \varrho(p,S) := \frac{1}{\sqrt[\gamma]{A}}p^{\frac{1}{\gamma}}\exp\{-\frac{S}{\gamma}\}$,
$(\hat{p}_0(x), \hat{S}_0(x))$ may be different from $(p_0(x), S_0(x))$.
Here, $(\ref{Sect1_Diffusion_Eq})_2$ is not an evolution equation of $u$, thus $u$ itself does not need the initial data.

The physical model of the equations $(\ref{Sect1_Diffusion_Eq})$ is the sufficiently slow motion of the ideal gases in porous media, Darcy's law gives the relationship between the momentum of ideal gases and the gradient of their pressure. The system $(\ref{Sect1_Diffusion_Eq})$ is essentially a parabolic-hyperbolic system with respect to $p$ and $S$ after eliminating $u$:
\begin{equation}\label{Sect1_Parabolic_Hyperbolic}
\left\{\begin{array}{lll}
p_t = \frac{\gamma p}{a\varrho}\triangle p - \frac{\gamma p}{a\varrho^2}\nabla\varrho\cdot\nabla p
+\frac{1}{a\varrho}|\nabla p|^2, \\[6pt]
S_t - \frac{1}{a\varrho}\nabla p\cdot\nabla S =0, \\[6pt]
(p,S)(x,0)=(\hat{p}_0(x), \hat{S}_0(x)), \\[6pt]
\frac{\partial p}{\partial n}|_{\partial\Omega} =0,\ \forall t\geq 0,
\end{array}\right.
\end{equation}
where $\varrho = \varrho(p,S)$.

We introduce the following constants:
\begin{equation}\label{Sect1_Constants_DiffusionEq}
\begin{array}{ll}
\hat{\bar{p}}= \left( \frac{1}{|\Omega|}\int\limits_{\Omega} \hat{p}_0^{\frac{1}{\gamma}}\,\mathrm{d}x \right)^{\gamma},\quad
\hat{\bar{S}} = \frac{1}{\Omega}\int\limits_{\Omega} \hat{S}_0 \,\mathrm{d}x, \quad
\hat{\bar{\varrho}} = \frac{1}{\Omega}\int\limits_{\Omega} \hat{\varrho}_0 \,\mathrm{d}x,
\end{array}
\end{equation}
where $\hat{\varrho}_0 = \frac{1}{\sqrt[\gamma]{A}}\hat{p}_0^{\frac{1}{\gamma}}\exp\{-\frac{\hat{S}_0}{\gamma}\}$.
Thus, we can express the concept of small data for IBVP $(\ref{Sect1_Diffusion_Eq})$, i.e., the smallness of  $\|\hat{p}_0-\hat{\bar{p}}\|_{H^4(\Omega)}+\|(\hat{S}_0-\hat{\bar{S}},\hat{\varrho}_0-\hat{\bar{\varrho}})\|_{H^3(\Omega)}$.

We proved that if the initial data $(\hat{p}_0,\hat{S}_0)\in H^4(\Omega)\times H^3(\Omega)$ are sufficiently small perturbations of their mean values
$(\frac{1}{|\Omega|}\int\limits_{\Omega} \hat{p}_0 \,\mathrm{d}x,\hat{\bar{S}})$ or $(\hat{\bar{p}},\hat{\bar{S}})$, then IBVP $(\ref{Sect1_Diffusion_Eq})$ and $(\ref{Sect1_Parabolic_Hyperbolic})$ admit a unique global classical solution $(\hat{p},\hat{S})$ satisfying
\begin{equation*}
\begin{array}{ll}
(\hat{p},\hat{S})\in \underset{0\leq \ell\leq 3}{\cap}C^{\ell}([0,+\infty),H^{4-\ell}(\Omega)\times H^{3-\ell}(\Omega)),\
\triangle \hat{p} \in C(\Omega\times[0,+\infty)),
\end{array}
\end{equation*}
moreover,
\begin{equation*}
\left\{\begin{array}{ll}
\hat{\varrho} = \varrho(\hat{p},\hat{S})
\in \underset{0\leq \ell\leq 3}{\cap}C^{\ell}([0,+\infty),H^{3-\ell}(\Omega)),\\[6pt]
\hat{u}= - \frac{1}{a\hat{\varrho}}\nabla\hat{p}\in
\underset{0\leq \ell\leq 3}{\cap}C^{\ell}([0,+\infty),H^{3-\ell}(\Omega)),\
\nabla\cdot \hat{u} \in C(\Omega\times[0,+\infty)).
\end{array}\right.
\end{equation*}
Then $(\hat{p},\hat{u})$ converge exponentially to $(\hat{\bar{p}},0)$ rather than $(\frac{1}{|\Omega|}\int\limits_{\Omega} \hat{p}_0 \,\mathrm{d}x,0)$ as $t\rto +\infty$, $(\hat{\varrho},\hat{S})$ are uniformly bounded all the time.
Moreover, $\sum\limits_{0\leq\ell\leq 3}(\|\partial_t^{\ell}(\hat{p}-\hat{\bar{p}})\|_{H^{4-\ell}(\Omega)}
+\|\partial_t^{\ell}\hat{u}\|_{H^{3-\ell}(\Omega)})$,
$\sum\limits_{1\leq\ell\leq 3}\|\partial_t^{\ell}\hat{\varrho}\|_{H^{3-\ell}(\Omega)}$ and
$\sum\limits_{1\leq\ell\leq 3}\|\partial_t^{\ell}\hat{S}\|_{H^{3-\ell}(\Omega)}$ decay exponentially,
$\|\hat{\varrho}-\hat{\bar{\varrho}}\|_{H^{3}(\Omega)}$ and $\|\hat{S}-\hat{\bar{S}}\|_{H^{3}(\Omega)}$ are uniformly bounded.

We define $\hat{S}_{+} :=\sup\limits_{x\in\Omega}\hat{S}_0(x),\ \hat{S}_{-} :=\inf\limits_{x\in\Omega}\hat{S}_0(x)$ and denote $(\hat{S}_{\infty},\hat{\varrho}_{\infty})=\lim\limits_{t\rto \infty}(\hat{S},\hat{\varrho})$. Along the particle paths determined by $\hat{u}$, the entropy $\hat{S}$ remains invariant. We also proved mathematically that $\hat{S}_{\infty}\neq const,\ \hat{\varrho}_{\infty}\neq const$, if $\hat{S}_{+}\neq\hat{S}_{-}$. Moreover, $(\hat{p},\hat{u},\hat{S},\hat{\varrho})$ converge exponentially to their equilibrium states $(\hat{\bar{p}},0,\hat{S}_{\infty}(x),\hat{\varrho}_{\infty}(x))$ in $|\cdot|_{\infty}$ norm.

Furthermore, we proved that if $\int\limits_{\Omega} p_0^{\gamma}\,\mathrm{d}x
= \int\limits_{\Omega} \hat{p}_0^{\gamma}\,\mathrm{d}x$, then $\bar{p}=\hat{\bar{p}}$ and
$(p,u)$ of IBVP $(\ref{Sect1_NonIsentropic_EulerEq})$ converge exponentially to $(\hat{p},\hat{u})$ of IBVP $(\ref{Sect1_Diffusion_Eq})$, namely, as $t\rto +\infty$,
$$ \|p-\hat{p}\|_{H^3(\Omega)} + \|u-\hat{u}\|_{H^3(\Omega)} \leq C_1\exp\{-C_2 t\}.$$

In Lagrangian coordinates $\{(y,t)\}$, if $S_0(y)=\hat{S}_0(y)$, then $S_{\infty}(y)\equiv S(y,t)\equiv S_0(y)=\hat{S}_0(y)\equiv \hat{S}(y,t)\equiv \hat{S}_{\infty}(y)$.
While in Eulerian coordinates, $\hat{S}_{\infty}(x)\neq S_{\infty}(x), \hat{\varrho}_{\infty}(x)\neq \varrho_{\infty}(x)$ in general, due to the transportation of $\hat{S},\hat{\varrho},S,\varrho$. For a given $S_0(x)$, whether there exists $\hat{S}_0(x)$ such that $\hat{S}_{\infty}(x)=S_{\infty}(x)$ is still open. If such a $\hat{S}_0(x)$ exists, $(p,u,S,\varrho)$ of IBVP $(\ref{Sect1_NonIsentropic_EulerEq})$ converge to $(\hat{p},\hat{u},\hat{S},\hat{\varrho})$ of IBVP $(\ref{Sect1_Diffusion_Eq})$ in Eulerian coordinates, as $t\rto +\infty$.

The rest of this paper is organized as follows: In Section 2, we reformulate the equations $(\ref{Sect1_NonIsentropic_EulerEq}),(\ref{Sect1_Diffusion_Eq})$ into appropriate forms and state the main results. In Section 3, we prove global a priori estimates for the non-isentropic Euler equations with damping $(\ref{Sect1_NonIsentropic_EulerEq})$. In Section 4, we prove the global existence of classical solutions to $(\ref{Sect1_NonIsentropic_EulerEq})$ and singularity formation for large data. In Section 5, we prove global a priori estimates for the diffusion equations $(\ref{Sect1_Diffusion_Eq})$. In Section 6, we prove the global existence of classical solutions to $(\ref{Sect1_Diffusion_Eq})$ and the nonlinear diffusion property of $(\ref{Sect1_NonIsentropic_EulerEq})$.

\section{Preliminaries and Precise Statements of Main Results}
In this section, we will reformulate the equations $(\ref{Sect1_NonIsentropic_EulerEq}),(\ref{Sect1_Diffusion_Eq})$ into appropriate forms, define some energy quantities and state precisely the main results of this paper.

The following lemma mainly gives the relationship between $p_{\infty}(x)$ and the initial data $(p_0(x),S_0(x),\varrho_0(x))$.
\begin{lemma}\label{Sect2_P_Infty_Lemma}
\begin{equation}\label{Sect2_P_Infty_toProve}
\begin{array}{ll}
p_{\infty} = \bar{p} = \left( \frac{1}{|\Omega|}\int\limits_{\Omega} p_0^{\frac{1}{\gamma}}\,\mathrm{d}x \right)^{\gamma},\quad \bar{p}\in[\inf\limits_{x\in\Omega}p(t),\sup\limits_{x\in\Omega}p(t)].
\end{array}
\end{equation}
\end{lemma}

\begin{proof}
By $(\ref{Sect1_NonIsentropic_EulerEq_Original})_1$ and $(\ref{Sect1_NonIsentropic_EulerEq_Original})_3$, we have
\begin{equation}\label{Sect2_P_Infty_1}
\begin{array}{ll}
(\varrho\exp\{\frac{S}{\gamma}\})_t + u\cdot\nabla(\varrho\exp\{\frac{S}{\gamma}\})
+ \varrho\exp\{\frac{S}{\gamma}\}\nabla\cdot u =0, \\[6pt]

\frac{\mathrm{d}}{\mathrm{d}t}\int\limits_{\Omega} \varrho\exp\{\frac{S}{\gamma}\} \,\mathrm{d}x
= - \int\limits_{\Omega} \nabla\cdot(\varrho u\exp\{\frac{S}{\gamma}\})\,\mathrm{d}x \\[6pt]\hspace{2.6cm}

= - \int\limits_{\partial\Omega} \varrho\exp\{\frac{S}{\gamma}\} u\cdot n\,\mathrm{d}S_x = 0, \\[6pt]

\int\limits_{\Omega} p^{\frac{1}{\gamma}} \,\mathrm{d}x
= \int\limits_{\Omega} p_0^{\frac{1}{\gamma}} \,\mathrm{d}x.
\end{array}
\end{equation}

In the equilibrium state, $\partial_t\varrho_{\infty}=u_{\infty}=\partial_t S_{\infty}=0$, plug which into the equations $(\ref{Sect1_NonIsentropic_EulerEq_Original})$, we have $\nabla p_{\infty}=0$, namely, $p_{\infty}$ is a constant. Then
\begin{equation}\label{Sect2_P_Infty_2}
\begin{array}{ll}
p_{\infty}^{\frac{1}{\gamma}}|\Omega| =
\int\limits_{\Omega} p_0^{\frac{1}{\gamma}} \,\mathrm{d}x, \\[6pt]

p_{\infty} =\bar{p} = \left( \frac{1}{|\Omega|}\int\limits_{\Omega} p_0^{\frac{1}{\gamma}}\,\mathrm{d}x \right)^{\gamma}.
\end{array}
\end{equation}

If $\bar{p}>\sup\limits_{x\in\Omega}p$ or $\bar{p}<\inf\limits_{x\in\Omega}p$, it contradicts with
$\int\limits_{\Omega} p^{\frac{1}{\gamma}} \,\mathrm{d}x
=\int\limits_{\Omega} \bar{p}^{\frac{1}{\gamma}} \,\mathrm{d}x$.
Thus, $\bar{p}\in[\inf\limits_{x\in\Omega}p,\sup\limits_{x\in\Omega}p]$.
\end{proof}

\begin{remark}
$\cite{Zhang_Wu_2014}$ pointed out the pressure $p$ in the periodic domain $\mathbb{T}$ converges \\[6pt]
to $\bar{p}=\frac{1}{|\mathbb{T}^3|}\int\limits_{\mathbb{T}^3}p_0\,\mathrm{d}x$. However,
$\bar{p}=\left(\frac{1}{|\mathbb{T}^3|}\int\limits_{\mathbb{T}^3}p_0^{\frac{1}{\gamma}}\,\mathrm{d}x \right)^{\gamma}$ is correct.
\end{remark}

For the non-isentropic Euler equations with damping $(\ref{Sect1_NonIsentropic_EulerEq})$ together with their initial data $(p_0,u_0,S_0,\varrho_0)$ and constants $\bar{p},\bar{S},\bar{\varrho}$, we introduce the constants:
\begin{equation*}
\begin{array}{ll}
k_1=\sqrt{\frac{1}{\gamma\bar{\varrho}\bar{p}}},\quad k_2 = \sqrt{\frac{\gamma\bar{p}}{\bar{\varrho}}},
\end{array}
\end{equation*}
define the variables:
\begin{equation*}
\begin{array}{ll}
\xi=p-\bar{p},\quad \phi=S-\bar{S},\quad v=\frac{1}{k_1}u,\quad \omega=\nabla\times v,\quad
(v_{\infty},\omega_{\infty})=\lim\limits_{t\rto \infty}(v,\omega).
\end{array}
\end{equation*}

In order to establish the global existence of IBVP $(\ref{Sect1_NonIsentropic_EulerEq})$, we reformulate the equations $(\ref{Sect1_NonIsentropic_EulerEq})$ into the following form:
\begin{equation}\label{Sect2_Final_Eq}
\left\{\begin{array}{ll}
\xi_t + k_2\nabla\cdot v = -\gamma k_1\xi\nabla\cdot v - k_1 v\cdot\nabla \xi, \\[6pt]
v_t + k_2 \nabla\xi + a v = - k_1 v\cdot\nabla v
+ \frac{1}{k_1} (\frac{1}{\bar{\varrho}}- \frac{1}{\varrho})\nabla\xi, \\[6pt]
\phi_t = - k_1 v\cdot\nabla\phi, \\[6pt]
(\xi, v, \phi)(x,0)=(p_0(x)-\bar{p}, \frac{1}{k_1} u_0(x), S_0(x)-\bar{S}), \\[6pt]
v\cdot n|_{\partial\Omega} =0,
\end{array}\right.
\end{equation}
where $\varrho =\varrho(\xi,\phi) := \frac{1}{\sqrt[\gamma]{A}}(\xi+\bar{p})^{\frac{1}{\gamma}}\exp\{-\frac{\phi+\bar{S}}{\gamma}\}$.

In order to prove the global existence of classical solutions to IBVP $(\ref{Sect2_Final_Eq})$ via the energy method, we define the following energy quantities:
\begin{definition}\label{Sect2_Def_Energy} Define
\begin{equation}\label{Sect2_Energy_Define}
\begin{array}{ll}
E[\xi](t):= \sum\limits_{\ell=0}^{3}\|\partial_t^{\ell} \xi(t)\|_{L^2(\Omega)}^2,\
E_1[\xi](t):=\sum\limits_{\ell=0}^{3}\|\partial_t^{\ell}\xi\|_{L^2(\Omega)}^2
- \int\limits_{\Omega}\frac{\xi}{p} \xi_{ttt}^2\,\mathrm{d}x, \\[8pt]

E[v](t):= \sum\limits_{\ell=0}^{3}\|\partial_t^{\ell} v(t)\|_{L^2(\Omega)}^2,\
E_1[v](t):=\sum\limits_{\ell=0}^{3}|\partial_t^{\ell}v|^2
+\int\limits_{\Omega}(\frac{\varrho}{\bar{\varrho}}- 1)|v_{ttt}|^2 \,\mathrm{d}x, \\[8pt]

\mathcal{E}[\xi](t):=
\sum\limits_{0\leq \ell+|\alpha|\leq 3} \|\partial_t^{\ell} \mathcal{D}^{\alpha}\xi(t)\|_{L^2(\Omega)}^2,\
\mathcal{E}[v](t):=
\sum\limits_{0\leq \ell+|\alpha|\leq 3} \|\partial_t^{\ell} \mathcal{D}^{\alpha} v(t)\|_{L^2(\Omega)}^2, \\[8pt]

E[\phi](t):= \sum\limits_{\ell=0}^{3}\|\partial_t^{\ell} \phi(t)\|_{L^2(\Omega)}^2,\
\mathcal{E}[\phi](t):=
\sum\limits_{0\leq \ell+|\alpha|\leq 3} \|\partial_t^{\ell} \mathcal{D}^{\alpha}\phi(t)\|_{L^2(\Omega)}^2,\\[8pt]

\mathcal{E}_1[\omega](t)
:= \sum\limits_{0\leq \ell+|\alpha|\leq 2} \|\partial_t^{\ell} \mathcal{D}^{\alpha} \omega(t)\|_{L^2(\Omega)}^2, \\[8pt]

\mathcal{E}[\xi,v](t):=\mathcal{E}[\xi](t) + \mathcal{E}[v](t),\
\mathcal{E}[\xi,v,\phi](t):=\mathcal{E}[\xi,v](t) + \mathcal{E}[\phi](t).
\end{array}
\end{equation}
\end{definition}

In order to have classical solutions, even locally in time, the initial data are required to be compatible with the boundary condition, namely, $\partial_t^{\ell} u(x,0)\cdot n|_{\partial\Omega}=0, 0\leq \ell\leq 3$, where $\partial_t^{\ell} u(x,0)$ can be solved by the equations $(\ref{Sect1_NonIsentropic_EulerEq})$ in terms of initial data $(p_0,u_0,S_0)$.

The following theorem states the global existence and large time behavior of classical solutions to IBVP $(\ref{Sect2_Final_Eq})$ and $(\ref{Sect1_NonIsentropic_EulerEq})$:
\begin{theorem}\label{Sect2_Global_Existence_Thm}
Assume $(p_0,u_0,S_0)\in H^3(\Omega)$, $\inf\limits_{x\in\Omega}p_0(x)>0$, $\partial_t^{\ell} u(x,0)\cdot n|_{\partial\Omega}=0$, $0\leq \ell\leq 3$.
There exists a sufficiently small number $\delta_1>0$, such that if $\|(p_0-\bar{p},\frac{1}{k_1}u_0,S_0-\bar{S})\|_{H^3(\Omega)}\leq \delta_1$, then IBVP $(\ref{Sect2_Final_Eq})$ admits a unique global classical solution
$(\xi,v,\phi)\in \underset{0\leq \ell\leq 3}{\cap}C^{\ell}([0,+\infty),H^{3-\ell}(\Omega)),$
moreover, $\varrho = \varrho(\xi,\phi)\in \underset{0\leq \ell\leq 3}{\cap}C^{\ell}([0,+\infty),H^{3-\ell}(\Omega))$.
Thus, IBVP $(\ref{Sect1_NonIsentropic_EulerEq})$ admits a unique global classical solution $(p=\bar{p}+\xi, u=k_1 v, S=\bar{S}+\phi)$.
$\sum\limits_{0\leq\ell\leq 3}(\|\partial_t^{\ell}(p-\bar{p})\|_{H^{3-\ell}(\Omega)}
+\|\partial_t^{\ell}u\|_{H^{3-\ell}(\Omega)})$,
$\sum\limits_{0\leq\ell\leq 2}\|\partial_t^{\ell}\omega\|_{H^{2-\ell}(\Omega)}$,
$\sum\limits_{1\leq\ell\leq 3}\|\partial_t^{\ell}\varrho\|_{H^{3-\ell}(\Omega)}$ and
$\sum\limits_{1\leq\ell\leq 3}\|\partial_t^{\ell}S\|_{H^{3-\ell}(\Omega)}$ decay exponentially,
$\|\varrho-\bar{\varrho}\|_{H^{3}(\Omega)}$ and $\|S-\bar{S}\|_{H^{3}(\Omega)}$ are uniformly bounded.

Furthermore, $S_{\infty}(x)\in [S_{-},S_{+}]$ exists and is unique, $p_{\infty}=\bar{p}$, $u_{\infty}=v_{\infty}=\omega_{\infty}=0$, $\varrho_{\infty}(x) = \frac{1}{\sqrt[\gamma]{A}}\bar{p}^{\frac{1}{\gamma}}\exp\{-\frac{S_{\infty}(x)}{\gamma}\}$.
If $S_{+}\neq S_{-}$, then $S_{\infty}\neq \bar{S}$, $\varrho_{\infty}\neq \bar{\varrho}$. As $t\rto +\infty$,
$(p,u,S,\varrho)$ converge to $(\bar{p},0,S_{\infty},\varrho_{\infty})$ exponentially in $|\cdot|_{\infty}$ norm.
\end{theorem}

However, the damping effect on the velocity is weakly dissipative, which can not prevent the singularity formation without small data assumption. The following theorem states that for a class of large initial data whose support $Supp(p_0-\bar{p},u_0,S_0-\bar{S})$ is away from the boundary $\partial\Omega$, the singularities form in the interior of ideal gases.
\begin{theorem}\label{Sect2_Blowup_Thm}
Assume $0\in\Omega$, $(p_0,u_0,S_0)\in H^3(\Omega)$, $\inf\limits_{x\in\Omega}p_0(x)>0$,
$h=dist\{\partial\Omega, Supp(p_0-\bar{p},u_0,S_0-\bar{S})\}>0$, $(p,u,S)\in C^1(\Omega\times[0,\tau))$ is the classical solution to IBVP $(\ref{Sect1_NonIsentropic_EulerEq})$ where $\tau>0$ is the lifespan of $(p,u,S)$. Denote
\begin{equation}\label{Sect2_blowup_Def_Quantities}
\begin{array}{ll}
M_{\varrho}(t) = \int\limits_{\Omega} \varrho u\cdot x \,\mathrm{d}x, \quad
B_0 = |Diam(\Omega)|^2\int\limits_{\Omega} \varrho_0 \,\mathrm{d}x, \\[6pt]
B_1 = \frac{3A e^{S_{-}}}{|\Omega|^{\gamma-1}}
\left(\int\limits_{\Omega} \varrho_0 \,\mathrm{d}x\right)^{\gamma} - 3\int\limits_{\Omega} \bar{p} \,\mathrm{d}x,
\quad
r=\sqrt{|B_1-\frac{a^2B_0}{4}|}.
\end{array}
\end{equation}
\noindent
For any fixed $T$ satisfying $0<T<\min\{\frac{h}{k_2},\frac{\pi}{2}\frac{B_0}{r}\}$, if
\begin{equation}\label{Sect2_Blowup_Condition}
\begin{array}{ll}
M_{\varrho}(0)> \max\{\frac{aB_0}{1-\exp\{-aT\}},\ \frac{aB_0}{2} + r\cot(\frac{rT}{B_0}), \\[6pt]\hspace{2.2cm}
\frac{aB_0}{2} -r + \cfrac{2r}{1-\exp\{-\frac{2rT}{B_0}\}},\ \frac{aB_0}{2} +r \},
\end{array}
\end{equation}
then $\tau\leq T$.
\end{theorem}

For the diffusion equations $(\ref{Sect1_Diffusion_Eq})$ with their initial data $(\hat{p}_0,\hat{u}_0,\hat{S}_0,\hat{\varrho}_0)$ and constants $(\hat{\bar{p}},\hat{\bar{S}},\hat{\bar{\varrho}})$, we introduce the following constants and variables:
\begin{equation*}
\begin{array}{ll}
\hat{k}_1=\sqrt{\frac{1}{\gamma\hat{\bar{\varrho}}\hat{\bar{p}}}},\quad
\hat{k}_2 = \sqrt{\frac{\gamma\hat{\bar{p}}}{\hat{\bar{\varrho}}}},\quad
\hat{\xi}=\hat{p}-\hat{\bar{p}},\quad \hat{\phi}=\hat{S}-\hat{\bar{S}},\\[6pt]
\hat{v}=\frac{1}{k_1}\hat{u},\quad \hat{\omega}=\nabla\times \hat{v},\quad
(\hat{v}_{\infty},\hat{\omega}_{\infty})=\lim\limits_{t\rto \infty}(\hat{v},\hat{\omega}).
\end{array}
\end{equation*}
For simplicity, we omit the symbol $\ \hat{}\ $ over all variables and constants in the equations, initial data and global a priori estimates, if there is no ambiguity, otherwise we will add the symbol $\ \hat{}\ $.

In order to establish the global existence of IBVP $(\ref{Sect1_Diffusion_Eq})$, we reformulate the equations $(\ref{Sect1_Diffusion_Eq})$ into the following form:
\begin{equation}\label{Sect2_Final_Diffusion}
\left\{\begin{array}{ll}
\xi_t + k_2\nabla\cdot v = -\gamma k_1\xi\nabla\cdot v - k_1 v\cdot\nabla \xi, \\[6pt]
k_2 \nabla\xi + a v = \frac{1}{k_1} (\frac{1}{\bar{\varrho}}- \frac{1}{\varrho})\nabla\xi, \\[6pt]
\phi_t = - k_1 v\cdot\nabla\phi, \\[6pt]
(\xi, \phi)(x,0)=(p_0(x)-\bar{p}, S_0(x)-\bar{S}), \\[6pt]
v\cdot n|_{\partial\Omega} =0,
\end{array}\right.
\end{equation}
where $\varrho =\varrho(\xi,\phi) := \frac{1}{\sqrt[\gamma]{A}}(\xi+\bar{p})^{\frac{1}{\gamma}}\exp\{-\frac{\phi+\bar{S}}{\gamma}\}$.

The system $(\ref{Sect2_Final_Diffusion})$ is still a parabolic-hyperbolic system with respect to $\xi$ and $\phi$, which has the following form after eliminating $v$:
\begin{equation}\label{Sect2_Parabolic_Hyperbolic}
\left\{\begin{array}{lll}
\xi_t = \frac{\gamma p}{a\varrho}\triangle\xi + \frac{p}{a\varrho}\nabla\xi\cdot\nabla\phi, \\[6pt]
\phi_t = \frac{1}{a\varrho}\nabla\xi\cdot\nabla\phi, \\[6pt]
(\xi,\phi)(x,0)=(p_0(x)-\bar{p}, S_0(x)-\bar{S}), \\[6pt]
\frac{\partial\xi}{\partial n}|_{\partial\Omega} =0,\ \forall t\geq 0,
\end{array}\right.
\end{equation}
where $\varrho = \varrho(\xi,\phi)$.

Next, we derive the evolution equations of $v$ from $(\ref{Sect2_Final_Diffusion})$, which is useful for proving a priori estimate for $\mathcal{E}_1[\omega](t)$.
Apply $\partial_{i}$ to $(\ref{Sect2_Final_Diffusion})_1$, we get
\begin{equation}\label{Sect2_Velocity_Solve 1}
\begin{array}{ll}
(\partial_i\xi)_t + k_1 \sum\limits_{\nu=1}^{3}v_{\nu}\partial_{\nu}(\partial_i\xi)
+ k_1 \sum\limits_{\nu=1}^{3}(\partial_i v_{\nu})\partial_{\nu}\xi
+ k_1\gamma(\partial_i \xi)\sum\limits_{\nu=1}^{3}\partial_{\nu} v_{\nu} \\[6pt]\hspace{1cm}
+ k_1\gamma p\sum\limits_{\nu=1}^{3}\partial_i\partial_{\nu} v_{\nu} =0.
\end{array}
\end{equation}

Plug $\partial_i \xi = - a k_1\varrho v_i$ into $(\ref{Sect2_Velocity_Solve 1})$, we have
\begin{equation}\label{Sect2_Velocity_Solve 2}
\begin{array}{ll}
(\varrho v_i)_t + k_1 \sum\limits_{\nu=1}^{3}v_{\nu}\cdot\partial_{\nu}(\varrho v_i)
+ k_1 \sum\limits_{\nu=1}^{3}(\partial_i v_{\nu})(\varrho v_{\nu})
+ k_1\gamma\varrho v_i\sum\limits_{\nu=1}^{3}\partial_{\nu} v_{\nu} \\[6pt]\hspace{1cm}
- \frac{\gamma p}{a}\sum\limits_{\nu=1}^{3}\partial_i\partial_{\nu} v_{\nu} =0.
\end{array}
\end{equation}

Plug $\varrho_t = -k_1 \sum\limits_{\nu=1}^{3}v_{\nu}\partial_{\nu}\varrho - k_1 \varrho \sum\limits_{\nu=1}^{3}\partial_{\nu}v_{\nu}$ into $(\ref{Sect2_Velocity_Solve 2})$, we obtained the development equation of $v$:
\begin{equation}\label{Sect2_Velocity_Solve 3}
\begin{array}{ll}
v_t = k_1(1-\gamma)v(\nabla\cdot v) - k_1 v\cdot\nabla v
- \frac{k_1}{2} \nabla(|v|^2) + \frac{\gamma p}{a\varrho}\nabla(\nabla\cdot v).
\end{array}
\end{equation}

Add $(\ref{Sect2_Velocity_Solve 3})$ to $(\ref{Sect2_Final_Diffusion})_2$, we obtained the following equations where $v$ is compatible with $\xi$ and $\phi$, but $(\ref{Sect2_Final_Diffusion_Vt})_2$ is not independent of $(\ref{Sect2_Final_Diffusion_Vt})_1$ and $(\ref{Sect2_Final_Diffusion_Vt})_3$.
\begin{equation}\label{Sect2_Final_Diffusion_Vt}
\left\{\begin{array}{ll}
\xi_t + k_2\nabla\cdot v = -\gamma k_1\xi\nabla\cdot v - k_1 v\cdot\nabla \xi, \\[6pt]

v_t + k_2 \nabla\xi + a v = \frac{1}{k_1} (\frac{1}{\bar{\varrho}}- \frac{1}{\varrho})\nabla\xi
+ k_1(1-\gamma)v(\nabla\cdot v) - k_1 v\cdot\nabla v \\[6pt]\hspace{2.8cm}
- \frac{k_1}{2} \nabla(|v|^2) + \frac{\gamma p}{a\varrho}\nabla(\nabla\cdot v), \\[6pt]

\phi_t = - k_1 v\cdot\nabla\phi,
\end{array}\right.
\end{equation}
where $v = -\frac{1}{a k_1\varrho}\nabla\xi,\ \varrho = \varrho(\xi,\phi)$.

In order to prove the global existence of classical solutions to IBVP $(\ref{Sect2_Final_Diffusion})$ and $(\ref{Sect1_Diffusion_Eq})$ via the energy method, we define the following energy quantities besides
the energy quantities which have been defined in $(\ref{Sect2_Energy_Define})$:
\begin{definition}\label{Sect2_Def_Energy_Diffusion} Define
\begin{equation}\label{Sect2_Energy_Define_Diffusion}
\begin{array}{ll}
\mathcal{F}[\xi](t):= \sum\limits_{0\leq \ell\leq 3,\ell + |\alpha|\leq 4}
\|\partial_t^{\ell} \mathcal{D}^{\alpha}\xi(t)\|_{L^2(\Omega)}^2, \\[12pt]

\mathcal{F}[v](t):= \mathcal{E}[v](t) + \sum\limits_{0\leq \ell\leq 2, \ell + |\alpha|\leq 3}
\|\partial_t^{\ell} \mathcal{D}^{\alpha}(\nabla\cdot v(t))\|_{L^2(\Omega)}^2,\\[12pt]

\mathcal{F}[\xi,v](t):=\mathcal{F}[\xi](t) + \mathcal{F}[v](t),\

\mathcal{F}[\xi,v,\phi](t):=\mathcal{F}[\xi,v](t) + \mathcal{E}[\phi](t).
\end{array}
\end{equation}
\end{definition}

In addition, $\mathcal{F}[v](t)$ contains more information about $\xi$ than $\mathcal{F}[\xi](t)$ itself. All the definitions of energy quantities in $(\ref{Sect2_Energy_Define})$, $(\ref{Sect2_Energy_Define_Diffusion})$ are independent of the equations and initial data, thus the definitions $(\ref{Sect2_Energy_Define})$ can be used for IBVP $(\ref{Sect2_Final_Diffusion})$.

In order to have classical solutions to IBVP $(\ref{Sect2_Final_Diffusion})$, we need to improve the regularity of the initial data, namely $(p_0,S_0)\in H^4(\Omega)\times H^3(\Omega)$. Also, the initial data are required to be compatible with the boundary condition, namely, $\partial_t^{\ell} \nabla p(x,0)\cdot n|_{\partial\Omega}=0,0\leq \ell\leq 3$, where $\partial_t^{\ell} \nabla p(x,0)$ are solved by the equations $(\ref{Sect1_Diffusion_Eq})$ in terms of initial data $(p_0,S_0)$.

The following theorem states that the global existence and large time behavior of classical solutions to IBVP $(\ref{Sect2_Final_Diffusion})$ and $(\ref{Sect1_Diffusion_Eq})$:
\begin{theorem}\label{Sect2_GlobalExistence_Thm_DiffusionEq}
Assume $(p_0,S_0)\in H^4(\Omega)\times H^3(\Omega)$, $\inf\limits_{x\in\Omega}p_0(x)>0$ and $\partial_t^{\ell} \nabla p(x,0)\cdot n|_{\partial\Omega}=0$, $0\leq \ell\leq 3$.
There exists a sufficiently small number $\delta_2>0$, such that if $\|p_0-\bar{p}\|_{H^4(\Omega)}+\|S_0-\bar{S}\|_{H^3(\Omega)}\leq \delta_2$, then IBVP $(\ref{Sect2_Final_Diffusion})$ admits a unique global classical solution $(\xi,\phi)$ satisfying
\begin{equation*}
\begin{array}{ll}
(\xi,\phi)\in \underset{0\leq \ell\leq 3}{\cap}C^{\ell}([0,+\infty),H^{4-\ell}(\Omega)\times H^{3-\ell}(\Omega)),\
\triangle\xi \in C(\Omega\times[0,+\infty)),
\end{array}
\end{equation*}
moreover,
\begin{equation*}
\left\{\begin{array}{ll}
\varrho = \varrho(\xi,\phi)
\in \underset{0\leq \ell\leq 3}{\cap}C^{\ell}([0,+\infty),H^{3-\ell}(\Omega)),
\\[6pt]
v = -\frac{1}{ak_1 \varrho}\nabla\xi
\in \underset{0\leq \ell\leq 3}{\cap}C^{\ell}([0,+\infty),H^{3-\ell}(\Omega)),\
\nabla\cdot v \in C(\Omega\times[0,+\infty)).
\end{array}\right.
\end{equation*}

Thus, IBVP $(\ref{Sect1_Diffusion_Eq})$ admits a unique global classical
solution $(p=\bar{p}+\xi, S=\bar{S}+\phi)$. Moreover,
$\sum\limits_{0\leq\ell\leq 3}(\|\partial_t^{\ell}(p-\bar{p})\|_{H^{4-\ell}(\Omega)}
+\|\partial_t^{\ell}u\|_{H^{3-\ell}(\Omega)})$,
$\sum\limits_{0\leq\ell\leq 2}\|\partial_t^{\ell}\omega\|_{H^{2-\ell}(\Omega)}$,
$\sum\limits_{1\leq\ell\leq 3}\|\partial_t^{\ell}\varrho\|_{H^{3-\ell}(\Omega)}$ and
$\sum\limits_{1\leq\ell\leq 3}\|\partial_t^{\ell}S\|_{H^{3-\ell}(\Omega)}$ decay exponentially,
$\|\varrho-\bar{\varrho}\|_{H^{3}(\Omega)}$ and $\|S-\bar{S}\|_{H^{3}(\Omega)}$ are uniformly bounded.

Furthermore, $S_{\infty}(x)\in [S_{-},S_{+}]$ exists and is unique, $p_{\infty}=\bar{p}$, $u_{\infty}=v_{\infty}=\omega_{\infty}=0$, $\varrho_{\infty}(x) = \frac{1}{\sqrt[\gamma]{A}}\bar{p}^{\frac{1}{\gamma}}\exp\{-\frac{S_{\infty}(x)}{\gamma}\}$.
If $S_{+}\neq S_{-}$, then $S_{\infty}\neq \bar{S}$, $\varrho_{\infty}\neq \bar{\varrho}$. As $t\rto +\infty$,
$(p,u,S,\varrho)$ converge to $(\bar{p},0,S_{\infty},\varrho_{\infty})$ exponentially in $|\cdot|_{\infty}$ norm.
\end{theorem}

The following theorem states that the pressure and velocity of non-isentropic Euler equations with damping have nonlinear diffusion property, they converges to the pressure and velocity of the diffusion equations.
\begin{theorem}\label{Sect2_Diffusion_Thm}
Assume $(\hat{p},\hat{u},\hat{S},\hat{\varrho})$ are variables of the diffusion equations $(\ref{Sect1_Diffusion_Eq})$ and $(p,u,S,\varrho)$ are variables of non-isentropic Euler equations with damping $(\ref{Sect1_NonIsentropic_EulerEq})$, the initial data $(p_0,u_0,S_0)$ satisfy the conditions in Theorem $\ref{Sect2_Global_Existence_Thm}$, $(\hat{p}_0,\hat{S}_0)$ satisfy the conditions in Theorem 
$\ref{Sect2_GlobalExistence_Thm_DiffusionEq}$.  
If 
\begin{equation}\label{Sect2_Diffusion_1}
\begin{array}{ll}
\int\limits_{\Omega}p_0^{\frac{1}{\gamma}}\,\mathrm{d}x
=\int\limits_{\Omega}\hat{p}_0^{\frac{1}{\gamma}}\,\mathrm{d}x,
\end{array}
\end{equation}
then
\begin{equation}\label{Sect2_Diffusion_2}
\begin{array}{ll}
\|p-\hat{p}\|_{H^3(\Omega)} + \|u-\hat{u}\|_{H^3(\Omega)} \leq C_1\exp\{-C_2 t\},
\end{array}
\end{equation}
for some positive $C_1,C_2$.
\end{theorem}

\begin{remark}\label{Sect2_Diffusion_remark}
In Eulerian coordinates, starting from $(x_0,0)$, the particle path $\chi(t;x_0)$ of non-isentropic Euler equations with damping do not coincide with $\hat{\chi}(t;x_0)$ of the diffusion equations, then $S_{\infty}(x)\neq \hat{S}_{\infty}(x)$ in general,
$\|S-\hat{S}\|_{L^2(\Omega)} + \|\varrho-\hat{\varrho}\|_{L^2(\Omega)}$ or $|S-\hat{S}|_{\infty} + |\varrho-\hat{\varrho}|_{\infty}$ may not decay. While this does not contradict with the results in 1D Lagrangian coordinates (see \cite{Hsiao_Pan_1999}).
Both $\chi(t;x_0)$ and $\hat{\chi}(t;x_0)$ in Eulerian coordinates correspond to the same line $\{(y_0,t)|t\geq 0\}$ in Lagrangian coordinates, where $y_0=\int\limits_0^{x_0}\varrho_0(x)\,\mathrm{d}x$ if $\Omega=[0,1]$. The entropy $S$ and $\hat{S}$ remains constant along vertical lines, so $S(y,t)\equiv \hat{S}(y,t)\equiv S_0(y)$ in Lagrangian coordinates, then $(p,u,S,\varrho)$ converge to $(\hat{p},\hat{u},\hat{S},\hat{\varrho})$ in Lagrangian coordinates.
\end{remark}

In the sequent sections, we will use the following notations: $X\lem Y$ denotes the estimate $X\leq CY$ for some implied constant $C>0$ which may different line by line. $(\cdot)_{k}$ denotes a vector in $\mathbb{R}^3$, for instance, $\omega_{k}=\delta_{ijk}\partial_{i}v_j$, where $\delta_{ijk}$ is totally anti-symmetric tensor such that $\delta_{123}=\delta_{231}=\delta_{312}=1, \delta_{213}=\delta_{321}=\delta_{132}=-1$, others are $0$. 'R.H.S.' is the abbreviation for 'right hand side'.

\section{Global A Priori Estimates for Non-Isentropic Euler Equations with Damping}
In this section, we derive global a priori estimates for the non-isentropic Euler equations with damping $(\ref{Sect2_Final_Eq})$.

The following lemma indicates $\varrho-\bar{\varrho}$, $\varrho_t$, $\nabla\varrho$ can be estimated by $\mathcal{E}[\xi,\phi](t)$.
\begin{lemma}\label{Sect3_Varrho_Lemma}
For any given $T\in (0,+\infty]$, if
\begin{equation*}
\sup\limits_{0\leq t\leq T} \mathcal{E}[\xi,v,\phi](t) \leq\e,
\end{equation*}
where $0<\e\ll 1$, then
\begin{equation}\label{Sect3_Varrho_Eq}
\begin{array}{ll}
\sup\limits_{0\leq t\leq T}\mathcal{E}[\varrho-\bar{\varrho}](t)\lem \e,\quad
\sup\limits_{0\leq t\leq T}|\varrho-\bar{\varrho}|_{\infty}\lem \sqrt{\e},\\[6pt]
\sup\limits_{0\leq t\leq T}|\varrho_t|_{\infty}\lem \sqrt{\e},\quad
\sup\limits_{0\leq t\leq T}|\nabla\varrho|_{\infty}\lem \sqrt{\e}.
\end{array}
\end{equation}
\end{lemma}

\begin{proof}
Since $\varrho=\frac{1}{\sqrt[\gamma]{A}}p^{\frac{1}{\gamma}}\exp\{-\frac{S}{\gamma}\}$, we have
\begin{equation*}
\begin{array}{ll}
\sup\limits_{0\leq t\leq T}\mathcal{E}[\varrho-\bar{\varrho}](t)
\lem \sup\limits_{0\leq t\leq T}(\mathcal{E}[\xi](t)+\mathcal{E}[\phi](t)) \lem \e, \\[8pt]

\sup\limits_{0\leq t\leq T}|\varrho-\bar{\varrho}|_{\infty}\lem
\sup\limits_{0\leq t\leq T}\mathcal{E}[\varrho-\bar{\varrho}](t)^{\frac{1}{2}} \lem \sqrt{\e}, \\[8pt]

\sup\limits_{0\leq t\leq T}|\varrho_t|_{\infty}
\lem \sup\limits_{0\leq t\leq T}(|\xi_t|_{\infty}+|\phi_t|_{\infty}) \lem \sqrt{\e}, \\[8pt]

\sup\limits_{0\leq t\leq T}|\nabla\varrho|_{\infty}
\lem \sup\limits_{0\leq t\leq T}(|\nabla\xi|_{\infty}+|\nabla\phi|_{\infty}) \lem \sqrt{\e}.
\end{array}
\end{equation*}
\end{proof}

The following lemma involves the Helmholtz-Hodge decomposition of vector fileds, which states that $\nabla v$ is estimated by $\omega$ and $\nabla\cdot v$. The proof of this lemma is standard (see \cite{Bourguignon_Brezis_1974,Pan_Zhao_2009}).
\begin{lemma}\label{Sect3_DivCurl_Decomposition}
Let $v \in H^s(\Omega)$ be a vector satisfying $v\cdot n|_{\partial\Omega}=0$, where $n$ is the unit outer norm of $\partial\Omega$, then $\|v\|_{s}\lem \|\omega\|_{s-1}+\|\nabla\cdot v\|_{s-1}+\|v\|_{s-1}$.
\end{lemma}

Lemma $\ref{Sect3_DivCurl_Decomposition}$ and the standard Sobolev's inequality $\|\cdot\|_{L^4(\Omega)}\lem \|\cdot\|_{H^1(\Omega)}$ are widely used to prove a priori estimates in Section 3 and Section 5.

The following lemma is an application of Lemma $\ref{Sect3_DivCurl_Decomposition}$, which states that the spatial derivatives are bounded by the temporal derivatives and the vorticity, then the total energy $\mathcal{E}[\xi,v](t)$ can be bounded by $E[\xi,v](t)$ and $\mathcal{E}_1[\omega](t)$.
\begin{lemma}\label{Sect3_Total_Energy_Lemma}
For any given $T\in (0,+\infty]$, there exists $\e_0>0$ which is independent of $(\xi_0,v_0,\phi_0)$, such that if
\begin{equation*}
\sup\limits_{0\leq t\leq T} \mathcal{E}[\xi,v,\phi](t) \leq\e,
\end{equation*}
where $\e\ll \min\{1,\e_0\}$, then for $\forall t\in [0,T]$,
\begin{equation}\label{Sect3_Total_Energy_Eq}
\mathcal{E}[\xi,v](t) \leq c_0(E[\xi,v](t) + \mathcal{E}_1[\omega](t)),
\end{equation}
for some $c_0>0$.
\end{lemma}

\begin{proof}
By $(\ref{Sect2_Final_Eq})_2$, we get
\begin{equation}\label{Sect3_Total_Energy_Prove1}
\begin{array}{ll}
\nabla\xi = - k_1\varrho v_t - ak_1\varrho v - k_1^2 \varrho v\cdot\nabla v, \\[6pt]

\|\nabla\xi\|_{L^2(\Omega)}^2 \lem \|v_t\|_{L^2(\Omega)}^2 +\|v\|_{L^2(\Omega)}^2
+\|v\cdot\nabla v\|_{L^2(\Omega)}^2.
\end{array}
\end{equation}

By $(\ref{Sect2_Final_Eq})_1$, we get
\begin{equation}\label{Sect3_Total_Energy_Prove2}
\begin{array}{ll}
\nabla\cdot v = \frac{-1}{k_1\gamma p}(\xi_t + k_1 v\cdot\nabla \xi), \\[6pt]
\|\nabla\cdot v\|_{L^2(\Omega)}^2 \lem \|\xi_t\|_{L^(\Omega)}^2 + \sqrt{\e}\mathcal{E}[\xi,v](t), \\[6pt]

\|\nabla v\|_{L^2(\Omega)}^2 \lem \|v\|_{L^2(\Omega)}^2+\|\omega\|_{L^2(\Omega)}^2
+\|\nabla\cdot v\|_{L^2(\Omega)}^2  \\[6pt]\hspace{1.7cm}
\lem \|\xi_t\|_{L^(\Omega)}^2 +\|\omega\|_{L^2(\Omega)}^2 +\|v\|_{L^2(\Omega)}^2 + \sqrt{\e}\mathcal{E}[\xi,v](t).
\end{array}
\end{equation}

Apply $\partial_t$ to $(\ref{Sect3_Total_Energy_Prove1})_1$, we get
\begin{equation}\label{Sect3_Total_Energy_Prove3}
\begin{array}{ll}
\nabla\xi_t = - k_1\xi_t v_t -k_1\varrho v_{tt} - ak_1\xi_t v -ak_1\varrho v_t
-k_1^2\partial_t(\varrho v\cdot\nabla v),\\[6pt]

\|\nabla\xi_t\|_{L^2(\Omega)}^2 \lem \|v_{tt}\|_{L^2(\Omega)}^2 +\|v_t\|_{L^2(\Omega)}^2 +\sqrt{\e}\mathcal{E}[\xi,v](t).
\end{array}
\end{equation}

Apply $\partial_t$ to $(\ref{Sect3_Total_Energy_Prove2})_1$, we get
\begin{equation}\label{Sect3_Total_Energy_Prove4}
\begin{array}{ll}
\nabla\cdot v_t = \frac{-1}{k_1\gamma p^2}[p(\xi_{tt} + k_1 v\cdot\nabla\xi_t + k_1 v_t\cdot\nabla\xi)
- \xi_t^2 - k_1\xi_t v\cdot\nabla\xi], \\[6pt]

\|\nabla\cdot v_t\|_{L^2(\Omega)}^2
 \lem \|\xi_{tt}\|_{L^2(\Omega)}^2+ \sqrt{\e}\mathcal{E}[\xi,v](t),\\[6pt]

\|\nabla v_t\|_{L^2(\Omega)}^2 \lem \|\nabla\cdot v_t\|_{L^2(\Omega)}^2
+ \|\omega_t\|_{L^2(\Omega)}^2 + \|v_t\|_{L^2(\Omega)}^2 \\[6pt]\hspace{1.8cm}
\lem \|\xi_{tt}\|_{L^2(\Omega)}^2 + \|\omega_t\|_{L^2(\Omega)}^2 + \|v_t\|_{L^2(\Omega)}^2
+ \sqrt{\e}\mathcal{E}[\xi,v](t).
\end{array}
\end{equation}

Apply $\partial_{tt}$ to $(\ref{Sect3_Total_Energy_Prove1})_1$, we get
\begin{equation}\label{Sect3_Total_Energy_Prove5}
\begin{array}{ll}
\nabla\xi_{tt} = - k_1\xi_{tt} v_t - 2k_1\xi_t v_{tt} -k_1\varrho v_{ttt}
-k_1^2(\varrho v\cdot\nabla v)_{tt}
 \\[6pt]\hspace{1.1cm}
- ak_1\xi_{tt} v - 2ak_1\xi_t v_t -ak_1\varrho v_{tt}, \\[6pt]

\|\nabla\xi_{tt}\|_{L^2(\Omega)}^2 \lem \|v_{ttt}\|_{L^2(\Omega)}^2 +\|v_{tt}\|_{L^2(\Omega)}^2 +\sqrt{\e}\mathcal{E}[\xi,v](t).
\end{array}
\end{equation}

Apply $\partial_{tt}$ to $(\ref{Sect3_Total_Energy_Prove2})_1$, we get
\begin{equation}\label{Sect3_Total_Energy_Prove6}
\begin{array}{ll}
\nabla\cdot v_{tt} = \frac{-1}{k_1\gamma p}\xi_{ttt} +\frac{1}{k_1\gamma p^2}\xi_t\xi_{tt}
 -\frac{1}{k_1\gamma}\partial_t[\frac{1}{p}(k_1 v\cdot\nabla\xi_t + k_1 v_t\cdot\nabla\xi) \\[6pt]\hspace{1.3cm}
-\frac{1}{p^2}(\xi_t^2 + k_1\xi_t v\cdot\nabla\xi)], \\[6pt]

\|\nabla\cdot v_{tt}\|_{L^2(\Omega)}^2 \lem \|\xi_{ttt}\|_{L^2(\Omega)}^2+ \sqrt{\e}\mathcal{E}[\xi,v](t),\\[6pt]

\|\nabla v_{tt}\|_{L^2(\Omega)}^2 \lem \|\nabla\cdot v_{tt}\|_{L^2(\Omega)}^2
+\|\omega_{tt}\|_{L^2(\Omega)}^2 +\|v_{tt}\|_{L^2(\Omega)}^2 \\[6pt]\hspace{1.9cm}
\lem \|\xi_{ttt}\|_{L^2(\Omega)}^2 +\|\omega_{tt}\|_{L^2(\Omega)}^2 +\|v_{tt}\|_{L^2(\Omega)}^2
+ \sqrt{\e}\mathcal{E}[\xi,v](t).
\end{array}
\end{equation}

Apply $\mathcal{D}^{\alpha}$ to $(\ref{Sect3_Total_Energy_Prove1})_1$, where $|\alpha|=1$, we get
\begin{equation}\label{Sect3_Total_Energy_Prove7}
\begin{array}{ll}
\mathcal{D}^{\alpha}\nabla\xi = - k_1(\mathcal{D}^{\alpha}\xi) v_t - k_1\varrho \mathcal{D}^{\alpha}v_t
- ak_1(\mathcal{D}^{\alpha}\xi) v - ak_1\varrho \mathcal{D}^{\alpha}v
- k_1^2 \mathcal{D}^{\alpha}(\varrho v\cdot\nabla v),  \\[6pt]

\|\mathcal{D}^{\alpha}\nabla\xi\|_{L^2(\Omega)}^2 \lem \|\mathcal{D}^{\alpha}v_t\|_{L^2(\Omega)}^2 +\|\mathcal{D}^{\alpha}v\|_{L^2(\Omega)}^2 +\sqrt{\e}\mathcal{E}[\xi,v](t) \\[6pt]\hspace{2.1cm}

\lem \|\xi_{tt}\|_{L^2(\Omega)}^2 +\|\omega_t\|_{L^2(\Omega)}^2+\|v_t\|_{L^2(\Omega)}^2
+\|\xi_t\|_{L^2(\Omega)}^2 +\|\omega\|_{L^2(\Omega)}^2 \\[6pt]\hspace{2.5cm}
+\|v\|_{L^2(\Omega)}^2 +\sqrt{\e}\mathcal{E}[\xi,v](t).
\end{array}
\end{equation}

Apply $\mathcal{D}^{\alpha}$ to $(\ref{Sect3_Total_Energy_Prove2})_1$, where $|\alpha|=1$, we get
\begin{equation}\label{Sect3_Total_Energy_Prove8}
\begin{array}{ll}
\mathcal{D}^{\alpha}\nabla\cdot v = \frac{1}{k_1\gamma p^2}(\mathcal{D}^{\alpha}\xi)(\xi_t + k_1 v\cdot\nabla\xi)
-\frac{1}{k_1\gamma p}\mathcal{D}^{\alpha}(\xi_t + k_1 v\cdot\nabla\xi), \\[6pt]

\|\mathcal{D}^{\alpha}\nabla\cdot v\|_{L^2(\Omega)}^2 \lem \|\mathcal{D}^{\alpha}\xi_t\|_{L^2(\Omega)}^2
+ \sqrt{\e}\mathcal{E}[\xi,v](t) \\[6pt]\hspace{2.4cm}
\lem \|v_t\|_{L^2(\Omega)}^2 + \|v_{tt}\|_{L^2(\Omega)}^2 + \sqrt{\e}\mathcal{E}[\xi,v](t), \\[6pt]

\|\mathcal{D}^{\alpha}\nabla v\|_{L^2(\Omega)}^2 \lem \|\nabla\cdot v\|_{H^1(\Omega)}^2
+ \|\omega\|_{H^1(\Omega)}^2 + \|v\|_{H^1(\Omega)}^2
\lem \|v_t\|_{L^2(\Omega)}^2 +\|v_{tt}\|_{L^2(\Omega)}^2 \\[6pt]\hspace{2.6cm}
+ \|\omega\|_{H^1(\Omega)}^2 + \|v\|_{L^2(\Omega)}^2 +\|\xi_t\|_{L^2(\Omega)}^2 + \sqrt{\e}\mathcal{E}[\xi,v](t).
\end{array}
\end{equation}

Apply $\mathcal{D}^{\alpha}$ to $(\ref{Sect3_Total_Energy_Prove3})_1$, where $|\alpha|=1$, we get
\begin{equation}\label{Sect3_Total_Energy_Prove9}
\begin{array}{ll}
\mathcal{D}^{\alpha}\nabla\xi_t = - k_1\mathcal{D}^{\alpha}(\xi_t v_t) -k_1(\mathcal{D}^{\alpha}\xi) v_{tt}
-k_1\varrho \mathcal{D}^{\alpha}v_{tt} -k_1^2\mathcal{D}^{\alpha}(\xi_t v\cdot\nabla v) \\[6pt]\hspace{1.5cm}

-k_1^2\mathcal{D}^{\alpha}(\varrho v_t\cdot\nabla v)-k_1^2\mathcal{D}^{\alpha}(\varrho v\cdot\nabla v_t)
- ak_1\mathcal{D}^{\alpha}(\xi_t v) \\[6pt]\hspace{1.5cm}
-ak_1(\mathcal{D}^{\alpha}\xi) v_t -ak_1\varrho\mathcal{D}^{\alpha} v_t, \\[6pt]

\|\mathcal{D}^{\alpha}\nabla\xi_t\|_{L^2(\Omega)}^2 \lem \|\mathcal{D}^{\alpha} v_{tt}\|_{L^2(\Omega)}^2
+\|\mathcal{D}^{\alpha} v_t\|_{L^2(\Omega)}^2 +\sqrt{\e}\mathcal{E}[\xi,v](t) \\[6pt]\hspace{2.3cm}

\lem \|\xi_{ttt}\|_{L^2(\Omega)}^2
+ \|\omega_{tt}\|_{L^2(\Omega)}^2 +\|v_{tt}\|_{L^2(\Omega)}^2 +\|\xi_{tt}\|_{L^2(\Omega)}^2 \\[6pt]\hspace{2.7cm} +\|\omega_t\|_{L^2(\Omega)}^2 +\|v_t\|_{L^2(\Omega)}^2 +\sqrt{\e}\mathcal{E}[\xi,v](t).
\end{array}
\end{equation}

Apply $\mathcal{D}^{\alpha}$ to $(\ref{Sect3_Total_Energy_Prove4})_1$, where $|\alpha|=1$, we get
\begin{equation}\label{Sect3_Total_Energy_Prove10}
\begin{array}{ll}
\mathcal{D}^{\alpha}\nabla\cdot v_t = \frac{-1}{k_1\gamma p}\mathcal{D}^{\alpha}\xi_{tt}
+\frac{1}{k_1\gamma p^2}(\mathcal{D}^{\alpha}\xi)\xi_{tt}
-\frac{1}{k_1\gamma}\mathcal{D}^{\alpha}[\frac{1}{p}(k_1 v\cdot\nabla\xi_t + k_1 v_t\cdot\nabla\xi) \\[6pt]\hspace{1.8cm}
- \frac{1}{p^2}(\xi_t^2 + k_1\xi_t v\cdot\nabla\xi)], \\[12pt]

\|\mathcal{D}^{\alpha}\nabla\cdot v_t\|_{L^2(\Omega)}^2 \lem \|\mathcal{D}^{\alpha}\xi_{tt}\|_{L^2(\Omega)}^2
+\sqrt{\e}\mathcal{E}[\xi,v](t) \\[6pt]\hspace{2.5cm}
\lem \|v_{ttt}\|_{L^2(\Omega)}^2 +\|v_{tt}\|_{L^2(\Omega)}^2 +\sqrt{\e}\mathcal{E}[\xi,v](t), \\[6pt]

\|\mathcal{D}^{\alpha}\nabla v_t\|_{L^2(\Omega)}^2 \lem \|\nabla\cdot v_t\|_{H^1(\Omega)}^2
+ \|\omega_t\|_{H^1(\Omega)}^2 + \|v_t\|_{H^1(\Omega)}^2 + \sqrt{\e}\mathcal{E}[\xi,v](t)
\\[6pt]\hspace{2.3cm}

\lem \|\xi_{tt}\|_{L^2(\Omega)}^2 + \|v_{ttt}\|_{L^2(\Omega)}^2 + \|v_{tt}\|_{L^2(\Omega)}^2
+ \|\omega_t\|_{H^1(\Omega)}^2 \\[6pt]\hspace{2.7cm}
+ \|v_t\|_{L^2(\Omega)}^2 + \sqrt{\e}\mathcal{E}[\xi,v](t).
\end{array}
\end{equation}

Apply $\mathcal{D}^{\alpha}$ to $(\ref{Sect3_Total_Energy_Prove1})_1$, where $|\alpha|=2$, $\alpha=\alpha_1+\alpha_2$, we get
\begin{equation}\label{Sect3_Total_Energy_Prove11}
\begin{array}{ll}
\mathcal{D}^{\alpha}\nabla\xi = - k_1\varrho\mathcal{D}^{\alpha}v_t
- k_1\sum\limits_{\alpha_1>0}(\mathcal{D}^{\alpha_1}\xi) \mathcal{D}^{\alpha_2}v_t
- ak_1\varrho\mathcal{D}^{\alpha}v \\[6pt]\hspace{1.4cm}
- ak_1\sum\limits_{\alpha_1>0}(\mathcal{D}^{\alpha_1}\xi)\mathcal{D}^{\alpha_2}v
- k_1^2 \mathcal{D}^{\alpha}(\varrho v\cdot\nabla v), \\[12pt]

\|\mathcal{D}^{\alpha}\nabla\xi\|_{L^2(\Omega)}^2 \lem \|\mathcal{D}^{\alpha}v_t\|_{L^2(\Omega)}^2 +\|\mathcal{D}^{\alpha}v\|_{L^2(\Omega)}^2 +\sqrt{\e}\mathcal{E}[\xi,v](t)
\\[6pt]\hspace{2.2cm}
\lem \|v_{ttt}\|_{L^2(\Omega)}^2 +\|v_{tt}\|_{L^2(\Omega)}^2 +\|\omega_t\|_{H^1(\Omega)}^2
+ \|\xi_{tt}\|_{L^2(\Omega)}^2 +\|v_t\|_{L^2(\Omega)}^2 \\[6pt]\hspace{2.6cm}
+ \|\omega\|_{H^1(\Omega)}^2 +\|\xi_t\|_{L^2(\Omega)}^2 +\|v\|_{L^2(\Omega)}^2+ \sqrt{\e}\mathcal{E}[\xi,v](t).
\end{array}
\end{equation}

Apply $\mathcal{D}^{\alpha}$ to $(\ref{Sect3_Total_Energy_Prove2})_1$, where $|\alpha|=2$, $\alpha=\alpha_1+\alpha_2$, we get
\begin{equation*}
\begin{array}{ll}
\mathcal{D}^{\alpha}\nabla\cdot v = \frac{-1}{k_1\gamma}\mathcal{D}^{\alpha_1}(\frac{1}{p})
\mathcal{D}^{\alpha_2}(\xi_t + k_1 v\cdot\nabla\xi), \\[6pt]
\end{array}
\end{equation*}

\begin{equation}\label{Sect3_Total_Energy_Prove12}
\begin{array}{ll}
\|\mathcal{D}^{\alpha}\nabla\cdot v\|_{L^2(\Omega)}^2 \lem \|\mathcal{D}^{\alpha}\xi_t\|_{L^2(\Omega)}^2+ \sqrt{\e}\mathcal{E}[\xi,v](t) \\[6pt]\hspace{2.4cm}
\lem \|\xi_{ttt}\|_{L^2(\Omega)}^2 +\|\omega_{tt}\|_{L^2(\Omega)}^2 +\|v_{tt}\|_{L^2(\Omega)}^2
+\|\xi_{tt}\|_{H^1(\Omega)}^2  \\[6pt]\hspace{2.8cm}
+\|\omega_t\|_{L^2(\Omega)}^2 +\|v_t\|_{L^2\Omega)}^2 + \sqrt{\e}\mathcal{E}[\xi,v](t),\\[6pt]

\|\mathcal{D}^{\alpha}\nabla v\|_{L^2(\Omega)}^2 \lem
\|\nabla\cdot v\|_{H^2(\Omega)}^2 +\|\omega\|_{H^2(\Omega)}^2 +\|v\|_{H^2(\Omega)}^2 \\[6pt]\hspace{2.2cm}

\lem \|\xi_{ttt}\|_{L^2(\Omega)}^2 +\|\omega_{tt}\|_{L^2(\Omega)}^2 +\|v_{tt}\|_{L^2(\Omega)}^2
+\|\xi_{tt}\|_{H^1(\Omega)}^2  \\[6pt]\hspace{2.6cm}
+\|\omega_t\|_{L^2(\Omega)}^2 +\|\omega\|_{H^2(\Omega)}^2 +\|v_t\|_{L^2\Omega)}^2
+\|v\|_{L^2\Omega)}^2 \\[6pt]\hspace{2.6cm}
+\|\xi_t\|_{L^2\Omega)}^2 + \sqrt{\e}\mathcal{E}[\xi,v](t).
\end{array}
\end{equation}

Thus, $\mathcal{E}[\xi,v](t) \leq C_3E[\xi,v](t) + C_3\mathcal{E}_1[\omega](t)+ C_3\sqrt{\e}\mathcal{E}[\xi,v](t)$,
where $C_3>0$.

Let $\e_0=\frac{1}{4C_3^2}$, when $\e\ll\min\{1,\e_0\}$, we have
\begin{equation}\label{Sect3_Total_Energy_Prove14}
\begin{array}{ll}
\mathcal{E}[\xi,v](t)\leq 2C_3 \{E[\xi,v](t) + \mathcal{E}_1[\omega](t)\}.
\end{array}
\end{equation}
Let $c_0=2C_3$. Thus, Lemma $\ref{Sect3_Total_Energy_Lemma}$ is proved.
\end{proof}

Next, in order to prove the exponential decay of $\mathcal{E}[\xi,v](t)$ and $\mathcal{E}_1[\omega](t)$, we need to prove a priori estimates for $\mathcal{E}_1[\omega](t)$, $E[\xi,v](t)$ and
$\int\limits_{\Omega}\sum\limits_{\ell=1}^{3}\partial_t^{\ell}\xi \partial_t^{\ell-1}\xi \,\mathrm{d}x$ separatively.

The following lemma concerns a priori estimate for $\mathcal{E}_1[\omega](t)$.
\begin{lemma}\label{Sect3_Vorticity_Lemma}
For any given $T\in (0,+\infty]$, if
\begin{equation*}
\sup\limits_{0\leq t\leq T} \mathcal{E}[\xi,v,\phi](t) \leq\e,
\end{equation*}
where $0<\e\ll 1$, then for $\forall t\in [0,T]$,
\begin{equation}\label{Sect3_Vorticity_toProve}
\begin{array}{ll}
\frac{\mathrm{d}}{\mathrm{d}t} \mathcal{E}_1[\omega](t) + 2a\mathcal{E}_1[\omega](t)\leq C\sqrt{\e}\mathcal{E}[\xi,v](t).
\end{array}
\end{equation}
\end{lemma}

\begin{proof}
By $\nabla\times(\ref{Sect2_Final_Eq})_2$, we get
\begin{equation}\label{Sect3_Vorticity_1}
\begin{array}{ll}
\omega_t + a\omega = - k_1\nabla\times(v\cdot\nabla v)
+ \frac{1}{k_1} \nabla\times[(\frac{1}{\bar{\varrho}}- \frac{1}{\varrho})\nabla\xi] \\[6pt]\hspace{1.3cm}
= - k_1(v\cdot\nabla\omega - \omega\cdot\nabla v + \omega\nabla\cdot v)
+ \frac{1}{k_1} \nabla\times[(\frac{1}{\bar{\varrho}}- \frac{1}{\varrho})\nabla\xi] \\[6pt]\hspace{1.3cm}
= - k_1(v\cdot\nabla\omega - \omega\cdot\nabla v + \omega\nabla\cdot v)
+ \frac{1}{k_1} (\frac{\partial_i\varrho}{\varrho^2}\partial_j\xi
-\frac{\partial_j\varrho}{\varrho^2}\partial_i\xi)_{k}.
\end{array}
\end{equation}

Apply $\partial_t^{\ell}\mathcal{D}^{\alpha}$ to $(\ref{Sect3_Vorticity_1})$, where $\ell+|\alpha|\leq 2$, we get
\begin{equation*}
\begin{array}{ll}
(\partial_t^{\ell}\mathcal{D}^{\alpha}\omega)_t + a\partial_t^{\ell}\mathcal{D}^{\alpha}\omega
= - k_1\partial_t^{\ell}\mathcal{D}^{\alpha}(v\cdot\nabla\omega - \omega\cdot\nabla v + \omega\nabla\cdot v) \\[6pt]\hspace{3.6cm}
+ \frac{1}{k_1} \partial_t^{\ell}\mathcal{D}^{\alpha}
(\frac{\partial_i\varrho}{\varrho^2}\partial_j\xi-\frac{\partial_j\varrho}{\varrho^2}\partial_i\xi)_{k},
\end{array}
\end{equation*}

\begin{equation}\label{Sect3_Vorticity_2}
\begin{array}{ll}
\partial_t(|\partial_t^{\ell}\mathcal{D}^{\alpha}\omega|^2)+ 2a|\partial_t^{\ell}\mathcal{D}^{\alpha}\omega|^2
= - 2k_1\partial_t^{\ell}\mathcal{D}^{\alpha}(v\cdot\nabla\omega - \omega\cdot\nabla v + \omega\nabla\cdot v)\cdot\partial_t^{\ell}\mathcal{D}^{\alpha}\omega \\[6pt]\hspace{4.6cm}
+ \frac{2}{k_1} \sum\limits_{k=1}^{3}\partial_t^{\ell}\mathcal{D}^{\alpha}
(\frac{\partial_i\varrho}{\varrho^2}\partial_j\xi-\frac{\partial_j\varrho}{\varrho^2}\partial_i\xi)_{k}
\partial_t^{\ell}\mathcal{D}^{\alpha}\omega_{k}.
\end{array}
\end{equation}

After integrating in $\Omega$, we get
\begin{equation}\label{Sect3_Vorticity_3}
\begin{array}{ll}
\frac{\mathrm{d}}{\mathrm{d}t}\int\limits_{\Omega}|\partial_t^{\ell}\mathcal{D}^{\alpha}\omega|^2\,\mathrm{d}x
+ 2a\int\limits_{\Omega}|\partial_t^{\ell}\mathcal{D}^{\alpha}\omega|^2 \,\mathrm{d}x
= I_1 + I_2,
\end{array}
\end{equation}
where
\begin{equation}\label{Sect3_Vorticity_4}
\begin{array}{ll}
I_1 := - 2k_1\int\limits_{\Omega}\partial_t^{\ell}\mathcal{D}^{\alpha}(v\cdot\nabla\omega - \omega\cdot\nabla v
+ \omega\nabla\cdot v)\cdot\partial_t^{\ell}\mathcal{D}^{\alpha}\omega \,\mathrm{d}x, \\[10pt]

I_2 := \frac{2}{k_1} \int\limits_{\Omega}\sum\limits_{k=1}^{3}\partial_t^{\ell}\mathcal{D}^{\alpha}
(\frac{\partial_i\varrho}{\varrho^2}\partial_j\xi-\frac{\partial_j\varrho}{\varrho^2}\partial_i\xi)_{k}
\cdot\partial_t^{\ell}\mathcal{D}^{\alpha}\omega_k\,\mathrm{d}x.
\end{array}
\end{equation}

When $\ell+|\alpha|<2$, it is easy to check that $I_1+I_2\lem \sqrt{\e}\mathcal{E}[\xi,v](t)$, since they are lower order terms.
\vspace{0.2cm}

When $\ell=0,|\alpha|=2$,
\begin{equation}\label{Sect3_Vorticity_5}
\begin{array}{ll}
I_1 = - 2k_1\int\limits_{\Omega}(\mathcal{D}^{\alpha}v\cdot\nabla\omega
+ \mathcal{D}^{\alpha_1}v\cdot\nabla\mathcal{D}^{\alpha_2}\omega
+ v\cdot\nabla\mathcal{D}^{\alpha}\omega)\cdot\mathcal{D}^{\alpha}\omega \,\mathrm{d}x \\[6pt]\quad
+ 2k_1\int\limits_{\Omega}(\mathcal{D}^{\alpha}\omega\cdot\nabla v
+ \mathcal{D}^{\alpha_1}\omega\cdot\nabla\mathcal{D}^{\alpha_2} v
+ \omega\cdot\nabla\mathcal{D}^{\alpha} v)\cdot\mathcal{D}^{\alpha}\omega \,\mathrm{d}x \\[6pt]\quad
- 2k_1\int\limits_{\Omega}(\mathcal{D}^{\alpha}\omega\nabla\cdot v
+ \mathcal{D}^{\alpha_1}\omega\nabla\cdot\mathcal{D}^{\alpha_2} v
+ \omega\nabla\cdot\mathcal{D}^{\alpha} v)\cdot\mathcal{D}^{\alpha}\omega \,\mathrm{d}x \\[6pt]

\lem \|\mathcal{D}^{\alpha}v\|_{L^4(\Omega)}\|\nabla\omega\|_{L^4(\Omega)}\|\mathcal{D}^{\alpha}\omega\|_{L^2(\Omega)}
+|\mathcal{D}^{\alpha_1}v|_{\infty}\|\nabla\mathcal{D}^{\alpha_2}\omega\|_{L^2(\Omega)}\|\mathcal{D}^{\alpha}\omega\|_{L^2(\Omega)}

\\[6pt]\quad - 2k_1\int\limits_{\partial\Omega}n\cdot v|\mathcal{D}^{\alpha}\omega|^2 \,\mathrm{d}S_x
+ 2k_1\int\limits_{\Omega}\nabla\cdot v|\mathcal{D}^{\alpha}\omega|^2 \,\mathrm{d}x

+|\nabla v|_{\infty}\|\mathcal{D}^{\alpha}\omega\|_{L^2(\Omega)}^2 \\[6pt]\quad
+\|\mathcal{D}^{\alpha_1}\omega\|_{L^4(\Omega)}\|\nabla\mathcal{D}^{\alpha_2} v\|_{L^4(\Omega)}
\|\mathcal{D}^{\alpha}\omega\|_{L^2(\Omega)}
+|\omega|_{\infty}\|\nabla\mathcal{D}^{\alpha} v\|_{L^2(\Omega)}\|\mathcal{D}^{\alpha}\omega\|_{L^2(\Omega)} \\[6pt]\quad

+|\nabla\cdot v|_{\infty}\|\mathcal{D}^{\alpha}\omega\|_{L^2(\Omega)}^2
+\|\mathcal{D}^{\alpha_1}\omega\|_{L^4(\Omega)}\|\nabla\cdot\mathcal{D}^{\alpha_2}v\|_{L^4(\Omega)}
\|\mathcal{D}^{\alpha}\omega\|_{L^2(\Omega)}
\\[6pt]\quad
+|\omega|_{\infty}\|\nabla\cdot\mathcal{D}^{\alpha}v\|_{L^2(\Omega)}\|\mathcal{D}^{\alpha}\omega\|_{L^2(\Omega)}
\\[6pt]
\lem \sqrt{\e}\mathcal{E}[v](t),
\end{array}
\end{equation}
where $|\alpha_1|=|\alpha_2|=1$.

\vspace{0.3cm}
When $\ell=1,|\alpha|=1$,
\begin{equation}\label{Sect3_Vorticity_6}
\begin{array}{ll}
I_1= - 2k_1\int\limits_{\Omega}(\mathcal{D}^{\alpha}v_t\cdot\nabla\omega
+v_t\cdot\nabla\mathcal{D}^{\alpha}\omega +\mathcal{D}^{\alpha}v\cdot\nabla\omega_t
+v\cdot\nabla\mathcal{D}^{\alpha}\omega_t)\cdot\mathcal{D}^{\alpha}\omega_t \,\mathrm{d}x \\[6pt]\quad
+ 2k_1\int\limits_{\Omega}(\mathcal{D}^{\alpha}\omega_t\cdot\nabla v+\omega_t\cdot\nabla\mathcal{D}^{\alpha} v
+\mathcal{D}^{\alpha}\omega\cdot\nabla v_t+\omega\cdot\nabla\mathcal{D}^{\alpha} v_t)
\cdot\mathcal{D}^{\alpha}\omega_t \,\mathrm{d}x \\[6pt]\quad
- 2k_1\int\limits_{\Omega}(\mathcal{D}^{\alpha}\omega_t\nabla\cdot v+\omega_t\nabla\cdot\mathcal{D}^{\alpha} v
+\mathcal{D}^{\alpha}\omega\nabla\cdot v_t+\omega\nabla\cdot\mathcal{D}^{\alpha} v_t)
\cdot\mathcal{D}^{\alpha}\omega_t \,\mathrm{d}x \\[6pt]

\lem \|\mathcal{D}^{\alpha}v_t\|_{L^4(\Omega)}\|\nabla\omega\|_{L^4(\Omega)}\|\mathcal{D}^{\alpha}\omega_t\|_{L^2(\Omega)}
+|v_t|_{\infty}\|\nabla\mathcal{D}^{\alpha}\omega\|_{L^2(\Omega)}\|\mathcal{D}^{\alpha}\omega_t\|_{L^2(\Omega)}
\\[6pt]\quad
+ |\mathcal{D}^{\alpha}v|_{\infty}\|\nabla\omega_t\|_{L^2(\Omega)}\|\mathcal{D}^{\alpha}\omega_t\|_{L^2(\Omega)}
+ \int\limits_{\partial\Omega}n\cdot v|\mathcal{D}^{\alpha}\omega_t|^2 \,\mathrm{d}S_x
- \int\limits_{\Omega}\nabla\cdot v|\mathcal{D}^{\alpha}\omega_t|^2 \,\mathrm{d}x \\[6pt]\quad

+ |\nabla v|_{\infty}\|\mathcal{D}^{\alpha}\omega_t\|_{L^2(\Omega)}^2
+|\omega_t|_{L^4(\Omega)}\|\nabla\mathcal{D}^{\alpha}v\|_{L^4(\Omega)}\|\mathcal{D}^{\alpha}\omega_t\|_{L^2(\Omega)}
\\[6pt]\quad
+\|\mathcal{D}^{\alpha}\omega\|_{L^4(\Omega)}\|\nabla v_t\|_{L^4(\Omega)}
\|\mathcal{D}^{\alpha}\omega_t\|_{L^2(\Omega)}
+|\omega|_{\infty}\|\nabla\mathcal{D}^{\alpha} v_t\|_{L^2(\Omega)}\|\mathcal{D}^{\alpha}\omega_t\|_{L^2(\Omega)}
\\[6pt]\quad

+|\nabla\cdot v|_{\infty}\|\mathcal{D}^{\alpha}\omega_t\|_{L^2(\Omega)}^2
+\|\omega_t\|_{L^4(\Omega)}\|\nabla\cdot\mathcal{D}^{\alpha}v\|_{L^4(\Omega)}\|\mathcal{D}^{\alpha}\omega_t\|_{L^2(\Omega)}
\\[6pt]\quad

+\|\mathcal{D}^{\alpha}\omega\|_{L^4(\Omega)}\|\nabla\cdot v_t\|_{L^4(\Omega)}
\|\mathcal{D}^{\alpha}\omega_t\|_{L^2(\Omega)}
+|\omega|_{\infty}\|\nabla\cdot\mathcal{D}^{\alpha}v_t\|_{L^2(\Omega)}\|\mathcal{D}^{\alpha}\omega_t\|_{L^2(\Omega)}
\\[6pt]
\lem \sqrt{\e}\mathcal{E}[v](t).
\end{array}
\end{equation}

When $\ell=2,|\alpha|=0$,
\begin{equation}\label{Sect3_Vorticity_7}
\begin{array}{ll}
I_1 = - 2k_1\int\limits_{\Omega}(v_{tt}\cdot\nabla\omega + 2 v_t\cdot\nabla\omega_t
+ v\cdot\nabla\omega_{tt})\cdot\omega_{tt} \,\mathrm{d}x \\[6pt]\quad
+ 2k_1\int\limits_{\Omega}(\omega_{tt}\cdot\nabla v + 2 \omega_t\cdot\nabla v_t + \omega\cdot\nabla v_{tt})
\cdot\omega_{tt} \,\mathrm{d}x \\[6pt]\quad
- 2k_1\int\limits_{\Omega}(\omega_{tt}\nabla\cdot v+ 2 \omega_t\nabla\cdot v_t + \omega\nabla\cdot v_{tt})
\cdot\omega_{tt} \,\mathrm{d}x \\[6pt]

\lem \|v_{tt}\|_{L^4(\Omega)}\|\nabla\omega\|_{L^4(\Omega)}\|\omega_{tt}\|_{L^2(\Omega)}
+|v_t|_{\infty}\|\nabla\omega_t\|_{L^2(\Omega)}\|\omega_{tt}\|_{L^2(\Omega)}
\\[6pt]\quad
+\int\limits_{\partial\Omega}n\cdot v|\omega_{tt}|^2 \,\mathrm{d}S_x
-\int\limits_{\Omega}\nabla\cdot v|\omega_{tt}|^2 \,\mathrm{d}x
+ |\nabla v|_{\infty}\|\omega_{tt}\|_{L^2(\Omega)}^2  \\[6pt]\quad

+|\omega_t|_{L^4(\Omega)}\|\nabla v_t\|_{L^4(\Omega)}\|\omega_{tt}\|_{L^2(\Omega)}
+|\omega|_{\infty}\|\nabla v_{tt}\|_{L^2(\Omega)}\|\omega_{tt}\|_{L^2(\Omega)} \\[6pt]\quad

+ |\nabla\cdot v|_{\infty}\|\omega_{tt}\|_{L^2(\Omega)}^2
+\|\omega_t\|_{L^4(\Omega)}\|\nabla\cdot v_t\|_{L^4(\Omega)}\|\omega_{tt}\|_{L^2(\Omega)} \\[6pt]\quad
+|\omega|_{\infty}\|\nabla\cdot v_{tt}\|_{L^2(\Omega)}\|\omega_{tt}\|_{L^2(\Omega)}
\\[6pt]
\lem \sqrt{\e}\mathcal{E}[v](t).
\end{array}
\end{equation}

\vspace{0.2cm}
When $\ell=0,|\alpha|=2$,
\begin{equation}\label{Sect3_Vorticity_8}
\begin{array}{ll}
I_2 = \frac{2}{k_1} \int\limits_{\Omega}\sum\limits_{k=1}^{3}
[(\mathcal{D}^{\alpha}\frac{\partial_i\varrho}{\varrho^2})\partial_j\xi
-(\mathcal{D}^{\alpha}\frac{\partial_j\varrho}{\varrho^2})\partial_i\xi]_{k}
\mathcal{D}^{\alpha}\omega_{k} \,\mathrm{d}x \\[6pt]\qquad
+\frac{2}{k_1} \int\limits_{\Omega}\sum\limits_{k=1}^{3}
[(\mathcal{D}^{\alpha_1}\frac{\partial_i\varrho}{\varrho^2})(\mathcal{D}^{\alpha_2}\partial_j\xi)
-(\mathcal{D}^{\alpha_1}\frac{\partial_j\varrho}{\varrho^2})(\mathcal{D}^{\alpha_2}\partial_i\xi)]_{k}
\mathcal{D}^{\alpha}\omega_{k} \,\mathrm{d}x \\[6pt]\qquad
+\frac{2}{k_1} \int\limits_{\Omega}\sum\limits_{k=1}^{3}
[\frac{\partial_i\varrho}{\varrho^2}(\mathcal{D}^{\alpha}\partial_j\xi)
-\frac{\partial_j\varrho}{\varrho^2}(\mathcal{D}^{\alpha}\partial_i\xi)]_{k}
\mathcal{D}^{\alpha}\omega_{k} \,\mathrm{d}x \\[6pt]\quad

\lem |\nabla\xi|_{\infty}\|\mathcal{D}^{\alpha}\nabla\varrho\|_{L^2(\Omega)}\|\mathcal{D}^{\alpha}\omega\|_{L^2(\Omega)}
+|\nabla\varrho|_{\infty}\|\mathcal{D}^{\alpha}\nabla\xi\|_{L^2(\Omega)}\|\mathcal{D}^{\alpha}\omega\|_{L^2(\Omega)}
\\[6pt]\qquad
+\|\mathcal{D}^{\alpha_1}\nabla\varrho\|_{L^4(\Omega)}\|\mathcal{D}^{\alpha_2}\nabla\xi\|_{L^4(\Omega)}
\|\mathcal{D}^{\alpha}\omega\|_{L^2(\Omega)}
\\[6pt]\quad
\lem \sqrt{\e}\mathcal{E}[\xi,v](t),
\end{array}
\end{equation}
where
$\|\mathcal{D}^{\alpha}\frac{\varrho_{i}}{\varrho^2}\|_{L^2(\Omega)}\lem \|\mathcal{D}^{\alpha}\varrho_{i}\|_{L^2(\Omega)}$, $\|\mathcal{D}^{\alpha_1}\frac{\varrho_{i}}{\varrho^2}\|_{L^2(\Omega)}\lem \|\mathcal{D}^{\alpha_1}\varrho_{i}\|_{L^2(\Omega)}$, since $\e\ll 1$,
$|\alpha_1|=|\alpha_2|=1$.

\vspace{0.6cm}
When $\ell=1,|\alpha|=1$,
\begin{equation}\label{Sect3_Vorticity_9}
\begin{array}{ll}
I_2 = \frac{2}{k_1} \int\limits_{\Omega}\sum\limits_{k=1}^{3}[
\partial_t\mathcal{D}^{\alpha}(\frac{\partial_i\varrho}{\varrho^2})\partial_j\xi
-\partial_t\mathcal{D}^{\alpha}(\frac{\partial_j\varrho}{\varrho^2})\partial_i\xi

+ \partial_t(\frac{\partial_i\varrho}{\varrho^2})\mathcal{D}^{\alpha}(\partial_j\xi) \\[6pt]\qquad
- \partial_t(\frac{\partial_j\varrho}{\varrho^2})\mathcal{D}^{\alpha}(\partial_i\xi)

+\mathcal{D}^{\alpha}(\frac{\partial_i\varrho}{\varrho^2})\partial_t(\partial_j\xi)
-\mathcal{D}^{\alpha}(\frac{\partial_j\varrho}{\varrho^2})\partial_t(\partial_i\xi) \\[6pt]\qquad

+\frac{\partial_i\varrho}{\varrho^2}\partial_t\mathcal{D}^{\alpha}(\partial_j\xi)
-\frac{\partial_j\varrho}{\varrho^2}\partial_t\mathcal{D}^{\alpha}(\partial_i\xi)]_k
\partial_t\mathcal{D}^{\alpha}\omega_{k} \,\mathrm{d}x \\[6pt]\quad

\lem |\nabla\xi|_{\infty}\|\partial_t\mathcal{D}^{\alpha}\frac{\nabla\varrho}{\varrho^2}\|_{L^2(\Omega)}
\|\mathcal{D}^{\alpha}\omega_t\|_{L^2(\Omega)} \\[6pt]\qquad
+|\frac{\nabla\varrho}{\varrho^2}|_{\infty}\|\partial_t\mathcal{D}^{\alpha}\nabla\xi\|_{L^2(\Omega)}
\|\mathcal{D}^{\alpha}\omega_t\|_{L^2(\Omega)} \\[6pt]\qquad

+\|\partial_t(\frac{\nabla\varrho}{\varrho^2})\|_{L^4(\Omega)}\|\mathcal{D}^{\alpha}\nabla\xi\|_{L^4(\Omega)}
\|\mathcal{D}^{\alpha}\omega_t\|_{L^2(\Omega)} \\[6pt]\qquad

+\|\mathcal{D}^{\alpha}(\frac{\nabla\varrho}{\varrho^2})\|_{L^4(\Omega)}\|\nabla\xi_t\|_{L^4(\Omega)}
\|\mathcal{D}^{\alpha}\omega_t\|_{L^2(\Omega)}
\\[6pt]\quad
\lem \sqrt{\e}\mathcal{E}[\xi,v](t).
\end{array}
\end{equation}

When $\ell=2,|\alpha|=0$,
\begin{equation}\label{Sect3_Vorticity_10}
\begin{array}{ll}
I_2 = \frac{2}{k_1} \int\limits_{\Omega}\sum\limits_{k=1}^{3}[
(\frac{\partial_i\varrho}{\varrho^2})_{tt}\partial_j\xi -(\frac{\partial_j\varrho}{\varrho^2})_{tt}\partial_i\xi
+ 2(\frac{\partial_i\varrho}{\varrho^2})_t(\partial_j\xi)_t -2(\frac{\partial_j\varrho}{\varrho^2})_t(\partial_i\xi)_t
\\[6pt]\qquad

+ \frac{\partial_i\varrho}{\varrho^2}(\partial_j\xi)_{tt} -\frac{\partial_j\varrho}{\varrho^2}(\partial_i\xi)_{tt}
]_k \partial_{tt}\omega_{k} \,\mathrm{d}x
\\[6pt]\quad
\lem |\nabla\xi|_{\infty}\|(\frac{\nabla\varrho}{\varrho^2})_{tt}\|_{L^2(\Omega)}\|\omega_{tt}\|_{L^2(\Omega)}

+\|(\frac{\nabla\varrho}{\varrho^2})_t\|_{L^4(\Omega)}\|\nabla\xi_t\|_{L^4(\Omega)}
\|\omega_{tt}\|_{L^2(\Omega)} \\[6pt]\qquad

+|\frac{\nabla\varrho}{\varrho^2}|_{\infty}\|\nabla\xi_{tt}\|_{L^2(\Omega)}
\|\omega_{tt}\|_{L^2(\Omega)}
\\[6pt]\quad
\lem \sqrt{\e}\mathcal{E}[\xi,v](t).
\end{array}
\end{equation}

Summing the above estimates for $0\leq \ell+|\alpha|\leq 2$, we have
\begin{equation}\label{Sect3_Vorticity_11}
\begin{array}{ll}
\frac{\mathrm{d}}{\mathrm{d}t} \mathcal{E}_1[\omega](t) + 2a\mathcal{E}_1[\omega](t)\lem \sqrt{\e}\mathcal{E}[\xi,v](t).
\end{array}
\end{equation}
Thus, Lemma $\ref{Sect3_Vorticity_Lemma}$ is proved.
\end{proof}

\vspace{0.3cm}
The following lemma states that $E[\xi](t)$ and $E_1[\xi](t)$ are equivalent, $E[v](t)$ and $E_1[v](t)$ are equivalent.
\begin{lemma}\label{Sect3_Epsilon0_Lemma}
For any given $T\in (0,+\infty]$, there exists $\e_1>0$ which is independent of $(\xi_0,v_0,\phi_0)$, such that if $\sup\limits_{0\leq t\leq T} \mathcal{E}[\xi,v,\phi](t) \leq\e_1$, then there exist $c_1>0,c_2>0$ such that
\begin{equation}\label{Sect3_Energy_Equivalence}
\begin{array}{ll}
c_1 E[\xi](t) \leq E_1[\xi](t) \leq c_2 E[\xi](t),\\[6pt]
c_1 E[v](t) \leq E_1[v](t) \leq c_2 E[v](t).
\end{array}
\end{equation}
\end{lemma}

\begin{proof}
Since $|\xi|_{\infty}\lem \|\xi\|_{H^2(\Omega)}\leq C_4\sqrt{\e_1}$, let $C_4\sqrt{\e_1} \leq \frac{\bar{p}}{3}$, i.e. $\e_1\leq \frac{\bar{p}^2}{9C_4^2}$, then $\frac{\xi}{p}=\frac{\xi/\bar{p}}{1+\xi/\bar{p}}\in [-\frac{1}{2},\frac{1}{4}]$, $E_1[\xi](t)\cong E[\xi](t)$.

Since $|\varrho-\bar{\varrho}|_{\infty}\lem \|\varrho-\bar{\varrho}\|_{H^2(\Omega)}^{\frac{1}{2}} \lem \mathcal{E}[\xi,\phi](t)^{\frac{1}{2}}\leq C_5\sqrt{\e_1}$, let $C_5\sqrt{\e_1} \leq \frac{\bar{\varrho}}{2}$, i.e. $\e_1\leq \frac{\bar{\varrho}^2}{4C_5^2}$, then
$\frac{\varrho}{\bar{\varrho}}-1\in [-\frac{1}{2},\frac{1}{2}]$, $E_1[v](t)\cong E[v](t)$.

Thus, we can take $\e_1 =\min\{\frac{\bar{p}^2}{9C_4^2},\frac{\bar{\varrho}^2}{4C_5^2}\}$.
\end{proof}

\vspace{0.3cm}
Since $E_1[\xi,v](t)\cong E[\xi,v](t)$, the following lemma gives an equivalent a priori estimate for $E[\xi,v](t)$.
\begin{lemma}\label{Sect3_Energy_Estimate_Lemma1}
For any given $T\in (0,+\infty]$, if
\begin{equation*}
\sup\limits_{0\leq t\leq T} \mathcal{E}[\xi,v,\phi](t) \leq\e,
\end{equation*}
where $0<\e\ll 1$, then for $\forall t\in [0,T]$,
\begin{equation}\label{Sect3_Estimate1_toProve}
\frac{\mathrm{d}}{\mathrm{d}t}E_1[\xi,v](t)+ 2a E_1[v](t) \leq C\sqrt{\e}\mathcal{E}[\xi,v](t).
\end{equation}
\end{lemma}

\begin{proof}
Suppose $0\leq \ell \leq 3$, apply $\partial_t^{\ell}$ to $(\ref{Sect2_Final_Eq})$, we get
\begin{equation}\label{Sect3_Estimate2_1}
\left\{\begin{array}{ll}
(\partial_t^{\ell}\xi)_t + k_2\nabla\cdot(\partial_t^{\ell}v) = -\gamma k_1\partial_t^{\ell}(\xi\nabla\cdot v)
- k_1 \partial_t^{\ell}(v\cdot\nabla \xi), \\[10pt]
(\partial_t^{\ell}v)_t + k_2 \nabla(\partial_t^{\ell}\xi) + a \partial_t^{\ell}v = - k_1 \partial_t^{\ell}(v\cdot\nabla v)
+ \frac{1}{k_1} \partial_t^{\ell}[(\frac{1}{\bar{\varrho}}- \frac{1}{\varrho})\nabla\xi].
\end{array}\right.
\end{equation}

Let $(\ref{Sect3_Estimate2_1})\cdot(\partial_t^{\ell}\xi, \partial_t^{\ell}v)$, we get
\begin{equation}\label{Sect3_Estimate2_2}
\left\{\begin{array}{ll}
(|\partial_t^{\ell}\xi|^2)_t + 2k_2 \partial_t^{\ell}\xi\nabla\cdot(\partial_t^{\ell}v)
= -2\gamma k_1 \partial_t^{\ell}\xi\partial_t^{\ell}(\xi\nabla\cdot v)
- 2k_1 \partial_t^{\ell}\xi\partial_t^{\ell}(v\cdot\nabla \xi), \\[6pt]
(|\partial_t^{\ell}v|^2)_t + 2k_2 \partial_t^{\ell}v\cdot\nabla(\partial_t^{\ell}\xi)
+ 2a |\partial_t^{\ell}v|^2
= - 2 k_1 \partial_t^{\ell}v\cdot\partial_t^{\ell}(v\cdot\nabla v) \\[6pt]\hspace{6.2cm}
+ \frac{2}{k_1} \partial_t^{\ell}v\cdot
\partial_t^{\ell}[(\frac{1}{\bar{\varrho}}- \frac{1}{\varrho})\nabla\xi].
\end{array}\right.
\end{equation}

By $(\ref{Sect3_Estimate2_2})_1+(\ref{Sect3_Estimate2_2})_2$, we get
\begin{equation}\label{Sect3_Estimate2_3}
\begin{array}{ll}
(|\partial_t^{\ell}\xi|^2+|\partial_t^{\ell}v|^2)_t
+ 2k_2 \partial_t^{\ell}\xi\nabla\cdot(\partial_t^{\ell}v)+ 2k_2 \partial_t^{\ell}v\cdot\nabla(\partial_t^{\ell}\xi)
+ 2a |\partial_t^{\ell}v|^2 \\[6pt]
= -2\gamma k_1 \partial_t^{\ell}\xi\partial_t^{\ell}(\xi\nabla\cdot v)
- 2k_1 \partial_t^{\ell}\xi\partial_t^{\ell}(v\cdot\nabla \xi)
- 2 k_1 \partial_t^{\ell}v\cdot\partial_t^{\ell}(v\cdot\nabla v) \\[6pt]\quad
+ \frac{2}{k_1} \partial_t^{\ell}v\cdot
\partial_t^{\ell}[(\frac{1}{\bar{\varrho}}- \frac{1}{\varrho})\nabla\xi].
\end{array}
\end{equation}

After integrating $(\ref{Sect3_Estimate2_3})$ in $\Omega$, we get
\begin{equation}\label{Sect3_Estimate2_4}
\begin{array}{ll}
\frac{\mathrm{d}}{\mathrm{d} t}\int\limits_{\Omega}|\partial_t^{\ell}\xi|^2+|\partial_t^{\ell}v|^2 \,\mathrm{d}x
+ 2k_2 \int\limits_{\Omega}\partial_t^{\ell}\xi\nabla\cdot(\partial_t^{\ell}v)
+ \partial_t^{\ell}v\cdot\nabla(\partial_t^{\ell}\xi) \,\mathrm{d}x
+ 2a \int\limits_{\Omega}|\partial_t^{\ell}v|^2 \,\mathrm{d}x \\[6pt]
= \int\limits_{\Omega} -2\gamma k_1 \partial_t^{\ell}\xi\partial_t^{\ell}(\xi\nabla\cdot v)
- 2k_1 \partial_t^{\ell}\xi\partial_t^{\ell}(v\cdot\nabla \xi)
- 2 k_1 \partial_t^{\ell}v\cdot\partial_t^{\ell}(v\cdot\nabla v) \\[6pt]\quad
+ \frac{2}{k_1} \partial_t^{\ell}v\cdot
\partial_t^{\ell}[(\frac{1}{\bar{\varrho}}- \frac{1}{\varrho})\nabla\xi] \,\mathrm{d}x.
\end{array}
\end{equation}

Since $\partial_t^{\ell}v\cdot n|_{\partial\Omega}=0$,
$\int\limits_{\Omega}\partial_t^{\ell}\xi\nabla\cdot(\partial_t^{\ell}v)
+ \partial_t^{\ell}v\cdot\nabla(\partial_t^{\ell}\xi) \,\mathrm{d}x
= \int\limits_{\partial\Omega}\partial_t^{\ell}\xi\partial_t^{\ell}v \cdot n \,\mathrm{d}S_x =0$,
\begin{equation}\label{Sect3_Estimate2_5}
\begin{array}{ll}
\frac{\mathrm{d}}{\mathrm{d} t}
\int\limits_{\Omega}|\partial_t^{\ell}\xi|^2+|\partial_t^{\ell}v|^2 \,\mathrm{d}x
+ 2a \int\limits_{\Omega}|\partial_t^{\ell}v|^2 \,\mathrm{d}x \\[6pt]
= \int\limits_{\Omega} -2\gamma k_1 \partial_t^{\ell}\xi\partial_t^{\ell}(\xi\nabla\cdot v)
- 2k_1 \partial_t^{\ell}\xi\partial_t^{\ell}(v\cdot\nabla \xi)
- 2 k_1 \partial_t^{\ell}v\cdot\partial_t^{\ell}(v\cdot\nabla v) \\[6pt]\quad
+ \frac{2}{k_1} \partial_t^{\ell}v\cdot
\partial_t^{\ell}[(\frac{1}{\bar{\varrho}}- \frac{1}{\varrho})\nabla\xi] \,\mathrm{d}x := I_3.
\end{array}
\end{equation}

\vspace{0.4cm}
\noindent
When $0\leq \ell\leq 2$, it is easy to check that $I_3\lem \sqrt{\e}\mathcal{E}[\xi,v](t)$, since $I_3$ is a lower order term.

\noindent
When $\ell=3$,
\begin{equation}\label{Sect3_Estimate2_6}
\begin{array}{ll}
I_3=\int\limits_{\Omega} -2\gamma k_1 (\xi_{ttt})^2 \nabla\cdot v-6\gamma k_1 \xi_{ttt}\xi_{tt}\nabla\cdot v_t
-6\gamma k_1 \xi_{ttt}\xi_t\nabla\cdot v_{tt} \\[6pt]\qquad
-2\gamma k_1 \xi_{ttt}\xi\nabla\cdot v_{ttt} - 2k_1 \xi_{ttt} v_{ttt}\cdot\nabla \xi
-6k_1 \xi_{ttt} v_{tt}\cdot\nabla \xi_t \\[6pt]\qquad
-6k_1 \xi_{ttt} v_t\cdot\nabla \xi_{tt}- 2k_1 \xi_{ttt} v\cdot\nabla \xi_{ttt}
-2k_1 v_{ttt}\cdot\nabla v\cdot v_{ttt}- 6 k_1 v_{tt}\cdot\nabla v_t\cdot v_{ttt} \\[6pt]\qquad
- 6 k_1 v_t\cdot\nabla v_{tt}\cdot v_{ttt} - 2 k_1 v\cdot\nabla v_{ttt}\cdot v_{ttt}
+ \frac{2}{k_1} (\frac{1}{\bar{\varrho}}- \frac{1}{\varrho})v_{ttt}\cdot\nabla\xi_{ttt}  \\[6pt]\qquad
+ \frac{6}{k_1} (\frac{\varrho_t}{\varrho^2}v_{ttt}\cdot\nabla\xi_{tt}
+ \frac{\varrho\varrho_{tt}-2\varrho_t^2}{\varrho^3}v_{ttt}\cdot\nabla\xi_t)
+ \frac{2}{k_1} \frac{\varrho^2\varrho_{ttt}-6\varrho\varrho_t\varrho_{tt}
+6\varrho_t^3}{\varrho^4}v_{ttt}\cdot\nabla\xi \,\mathrm{d}x.
\end{array}
\end{equation}

\noindent
Now we estimate $I_3 - \frac{\mathrm{d}}{\mathrm{d}t}\int\limits_{\Omega}\frac{\xi}{p}\xi_{ttt}^2 \,\mathrm{d}x
+ \frac{\mathrm{d}}{\mathrm{d}t}\int\limits_{\Omega}(\frac{\varrho}{\bar{\varrho}}-1)|v_{ttt}|^2 \,\mathrm{d}x$,
\begin{equation}\label{Sect3_Estimate2_7}
\begin{array}{ll}
I_3 - \frac{\mathrm{d}}{\mathrm{d}t}\int\limits_{\Omega}\frac{\xi}{p}\xi_{ttt}^2 \,\mathrm{d}x
+ \frac{\mathrm{d}}{\mathrm{d}t}\int\limits_{\Omega}(\frac{\varrho}{\bar{\varrho}}-1)|v_{ttt}|^2 \,\mathrm{d}x \\[6pt]
\lem \sqrt{\e}\mathcal{E}[\xi,v](t)
+ \|\xi_{tt}\|_{L^4(\Omega)}\|\nabla\cdot v_t\|_{L^4(\Omega)} \|\xi_{ttt}\|_{L^2(\Omega)}
-2\gamma k_1 \int\limits_{\Omega}\xi_{ttt}\xi\nabla\cdot v_{ttt}\,\mathrm{d}x \\[6pt]\quad
+\|v_{tt}\|_{L^4(\Omega)}\|\nabla \xi_t\|\|_{L^4(\Omega)} \|\xi_{ttt}\|\|_{L^2(\Omega)}
- 2k_1 \int\limits_{\Omega}\xi_{ttt} v\cdot\nabla \xi_{ttt}\,\mathrm{d}x \\[6pt]\quad

+ \|v_{tt}\|_{L^4(\Omega)}\|\nabla v_t\|_{L^4(\Omega)}\|v_{ttt}\|_{L^2(\Omega)}
- 2 k_1 \int\limits_{\Omega}v\cdot\nabla v_{ttt}\cdot v_{ttt}\,\mathrm{d}x \\[6pt]\quad

+ \frac{2}{k_1} \int\limits_{\Omega}(\frac{1}{\bar{\varrho}}
- \frac{1}{\varrho})v_{ttt}\cdot\nabla\xi_{ttt}\,\mathrm{d}x
+ \|\varrho_{tt}\|_{L^4(\Omega)}\|\nabla\xi_t\|_{L^4(\Omega)}\|v_{ttt}\|_{L^2(\Omega)} \\[6pt]\quad
- 2\int\limits_{\Omega}\frac{\xi}{p}\xi_{ttt}\xi_{tttt} \,\mathrm{d}x
- \int\limits_{\Omega}\partial_t(\frac{\xi}{p})\xi_{ttt}^2 \,\mathrm{d}x
+ 2\int\limits_{\Omega}(\frac{\varrho}{\bar{\varrho}}-1)v_{ttt}\cdot v_{tttt} \,\mathrm{d}x
+ \int\limits_{\Omega}\frac{\varrho_t}{\bar{\varrho}}|v_{ttt}|^2 \,\mathrm{d}x \\[6pt]\quad

\lem \sqrt{\e}\mathcal{E}[\xi,v](t)
- k_1 \int\limits_{\partial\Omega}(|\xi_{ttt}|^2+|v_{ttt}|^2) v\cdot n\,\mathrm{d}S_x
+ k_1 \int\limits_{\Omega}(|\xi_{ttt}|^2+|v_{ttt}|^2) \nabla\cdot v\,\mathrm{d}x \\[6pt]\quad
-2\gamma k_1 \int\limits_{\Omega}\xi\xi_{ttt}\nabla\cdot v_{ttt}\,\mathrm{d}x
+ \frac{2}{k_1} \int\limits_{\Omega}(\frac{1}{\bar{\varrho}}- \frac{1}{\varrho})v_{ttt}\cdot\nabla\xi_{ttt}\,\mathrm{d}x \\[6pt]\quad
- 2\int\limits_{\Omega}\frac{\xi}{p}\xi_{ttt}\xi_{tttt} \,\mathrm{d}x
+ 2\int\limits_{\Omega}(\frac{\varrho}{\bar{\varrho}}-1)v_{ttt}\cdot v_{tttt} \,\mathrm{d}x \\[6pt]

\lem -2\int\limits_{\Omega}\frac{\xi}{p}\xi_{ttt}(\xi_{tttt}+k_1\gamma p\nabla\cdot v_{ttt})\,\mathrm{d}x
+ \frac{2}{k_1} \int\limits_{\Omega}(\frac{1}{\bar{\varrho}}- \frac{1}{\varrho})v_{ttt}\cdot(\nabla\xi_{ttt}+k_1\varrho v_{tttt})\,\mathrm{d}x
\\[6pt]\quad
+ \sqrt{\e}\mathcal{E}[\xi,v](t).
\end{array}
\end{equation}

Apply $\partial_{ttt}$ to $(\ref{Sect2_Final_Eq})_1$, we get
\begin{equation}\label{Sect3_Estimate2_8}
\begin{array}{ll}
\xi_{tttt} + k_1\gamma p\nabla\cdot v_{ttt} = -k_1v\cdot\nabla\xi_{ttt} -3 k_1v_t\cdot\nabla\xi_{tt}
-3 k_1v_{tt}\cdot\nabla\xi_t -k_1v_{ttt}\cdot\nabla\xi \\[6pt]\hspace{3.3cm}
- 3k_1\gamma\varrho_t\nabla\cdot v_{tt} - 3k_1\gamma\xi_{tt}\nabla\cdot v_t - k_1\gamma \xi_{ttt} \nabla\cdot v.
\end{array}
\end{equation}

Plug $(\ref{Sect3_Estimate2_8})$ into the following integral, we get
\begin{equation}\label{Sect3_Estimate2_9}
\begin{array}{ll}
\int\limits_{\Omega}\frac{\xi}{p}\xi_{ttt}(\xi_{tttt}+k_1\gamma p\nabla\cdot v_{ttt})\,\mathrm{d}x
= \int\limits_{\Omega}\frac{\xi}{p}\xi_{ttt}[R.H.S.\ of\ (\ref{Sect3_Estimate2_8})]\,\mathrm{d}x \\[6pt]

\lem \sqrt{\e}\mathcal{E}[\xi,v](t)
-k_1\int\limits_{\partial\Omega}\frac{\xi}{2p}|\xi_{ttt}|^2 v\cdot n \,\mathrm{d}S_x
+ \frac{k_1}{2}\int\limits_{\Omega}|\xi_{ttt}|^2 \nabla\cdot(\frac{\xi}{p}v) \,\mathrm{d}x

\lem \sqrt{\e}\mathcal{E}[\xi,v](t).
\end{array}
\end{equation}

Apply $\partial_{ttt}$ to $(\ref{Sect2_Final_Eq})_2$, we get
\begin{equation}\label{Sect3_Estimate2_10}
\begin{array}{ll}
\nabla\xi_{ttt} + k_1\varrho v_{tttt} = -3k_1\xi_t v_{ttt} -3k_1\xi_{tt} v_{tt} -k_1\xi_{ttt}v_t
-k_1^2\xi_{ttt} v\cdot\nabla v \\[6pt]\hspace{2.8cm}
-3k_1^2\xi_{tt} v_t\cdot\nabla v -3k_1^2\xi_{tt} v\cdot\nabla v_t -3k_1^2\xi_t v_{tt}\cdot\nabla v
\\[6pt]\hspace{2.8cm}
-3k_1^2\xi_t v\cdot\nabla v_{tt} -6k_1^2\xi_t v_t\cdot\nabla v_t -k_1^2\varrho v_{ttt}\cdot\nabla v
\\[6pt]\hspace{2.8cm}
-3k_1^2\varrho v_{tt}\cdot\nabla v_t -3k_1^2\varrho v_t\cdot\nabla v_{tt} -k_1^2\varrho v\cdot\nabla v_{ttt}
\\[6pt]\hspace{2.8cm}
- ak_1\xi_{ttt} v - 3k_1a\xi_{tt} v_t - 3k_1a\xi_t v_{tt} - ak_1\varrho v_{ttt}.
\end{array}
\end{equation}

Plug $(\ref{Sect3_Estimate2_10})$ into the following integral, we get
\begin{equation}\label{Sect3_Estimate2_11}
\begin{array}{ll}
\int\limits_{\Omega}(\frac{1}{\bar{\varrho}}- \frac{1}{\varrho})v_{ttt}\cdot(\nabla\xi_{ttt}+k_1\varrho v_{tttt})\,\mathrm{d}x
= \int\limits_{\Omega}(\frac{1}{\bar{\varrho}}- \frac{1}{\varrho})v_{ttt}\cdot
[R.H.S.\ of\ (\ref{Sect3_Estimate2_10})]\,\mathrm{d}x \\[6pt]

\lem -\frac{k_1^2}{2}\int\limits_{\partial\Omega}(\frac{\varrho}{\bar{\varrho}}-1)|v_{ttt}|^2 v\cdot n \,\mathrm{d}S_x + \frac{k_1^2}{2}\int\limits_{\Omega}|\xi_{ttt}|^2 \nabla\cdot[(\frac{\varrho}{\bar{\varrho}}-1) v] \,\mathrm{d}x
+ \sqrt{\e}\mathcal{E}[\xi,v](t) \\[6pt]

\lem \sqrt{\e}\mathcal{E}[\xi,v](t).
\end{array}
\end{equation}

Plug $(\ref{Sect3_Estimate2_9})$ and $(\ref{Sect3_Estimate2_11})$ into $(\ref{Sect3_Estimate2_7})$, we get
\begin{equation}\label{Sect3_Estimate2_12}
\begin{array}{ll}
I_3 - \frac{\mathrm{d}}{\mathrm{d}t}\int\limits_{\Omega}\frac{\xi}{p}\xi_{ttt}^2 \,\mathrm{d}x
+ \frac{\mathrm{d}}{\mathrm{d}t}\int\limits_{\Omega}(\frac{\varrho}{\bar{\varrho}}-1)|v_{ttt}|^2 \,\mathrm{d}x
\lem \sqrt{\e}\mathcal{E}[\xi,v](t).
\end{array}
\end{equation}

Finally, we have
\begin{equation}\label{Sect3_Estimate2_13}
\begin{array}{ll}
\frac{\mathrm{d}}{\mathrm{d} t}\left(\sum\limits_{\ell=0}^{3}
\int\limits_{\Omega}|\partial_t^{\ell}\xi|^2+|\partial_t^{\ell}v|^2 \,\mathrm{d}x
-\int\limits_{\Omega}\frac{\xi}{p} \xi_{ttt}^2\,\mathrm{d}x
+\int\limits_{\Omega}(\frac{\varrho}{\bar{\varrho}}- 1)|v_{ttt}|^2 \,\mathrm{d}x \right) \\[6pt]\quad
+ 2a \sum\limits_{\ell=0}^{3}\int\limits_{\Omega}|\partial_t^{\ell}v|^2 \,\mathrm{d}x
\lem \sqrt{\e}\mathcal{E}[\xi,v](t).
\end{array}
\end{equation}

Then
\begin{equation}\label{Sect3_Estimate2_14}
\begin{array}{ll}
\frac{\mathrm{d}}{\mathrm{d} t}\left(\sum\limits_{\ell=0}^{3}
\int\limits_{\Omega}|\partial_t^{\ell}\xi|^2+|\partial_t^{\ell}v|^2 \,\mathrm{d}x
-\int\limits_{\Omega}\frac{\xi}{p} \xi_{ttt}^2\,\mathrm{d}x
+\int\limits_{\Omega}(\frac{\varrho}{\bar{\varrho}}- 1)|v_{ttt}|^2 \,\mathrm{d}x \right) \\[6pt]\quad
+ 2a \left(\sum\limits_{\ell=0}^{3}\int\limits_{\Omega}|\partial_t^{\ell}v|^2 \,\mathrm{d}x
+ \int\limits_{\Omega}(\frac{\varrho}{\bar{\varrho}}- 1)|v_{ttt}|^2 \,\mathrm{d}x \right) \\[6pt]
\lem \sqrt{\e}\mathcal{E}[\xi,v](t)
+ 2a\int\limits_{\Omega}\frac{|\varrho-\bar{\varrho}|_{\infty}}{\bar{\varrho}}|v_{ttt}|^2 \,\mathrm{d}x
\lem \sqrt{\e}\mathcal{E}[\xi,v](t).
\end{array}
\end{equation}

So we get
\begin{equation}\label{Sect3_Estimate2_15}
\frac{\mathrm{d}}{\mathrm{d}t}E_1[\xi,v](t)+ 2a E_1[v](t) \lem \sqrt{\e}\mathcal{E}[\xi,v](t).
\end{equation}
Thus, Lemma $\ref{Sect3_Energy_Estimate_Lemma1}$ is proved.
\end{proof}

The following lemma concerns a priori estimate for $\int\limits_{\Omega}\sum\limits_{\ell=1}^{3}\partial_t^{\ell}\xi \partial_t^{\ell-1}\xi \,\mathrm{d}x$, which introduces $E[\xi](t)$ to the inequality $(\ref{Sect3_XiT_toProve})$.
\begin{lemma}\label{Sect3_Energy_Estimate_Lemma2}
For any given $T\in (0,+\infty]$, if
\begin{equation*}
\sup\limits_{0\leq t\leq T} \mathcal{E}[\xi,v,\phi](t) \leq\e,
\end{equation*}
where $0<\e\ll 1$, then there exists $c_3>0$ such that for $\forall t\in [0,T]$,
\begin{equation}\label{Sect3_XiT_toProve}
- \frac{\mathrm{d}}{\mathrm{d}t}\int\limits_{\Omega}\sum\limits_{\ell=1}^{3}
\partial_t^{\ell}\xi\partial_t^{\ell-1}\xi \,\mathrm{d}x
+ E[\xi](t) \leq C\sqrt{\e}\mathcal{E}[\xi,v](t) +c_3 E[v](t).
\end{equation}
\end{lemma}

\begin{proof}
Apply $\partial_t$ to $(\ref{Sect2_Final_Eq})_1$, then we get
\begin{equation}\label{Sect3_XiT_1}
\begin{array}{ll}
\hspace{1.5cm} \xi_{tt} = - k_1 v\cdot\nabla \xi_t - k_1 v_t\cdot\nabla \xi
-k_1\gamma\xi_t\nabla\cdot v -k_1\gamma p\nabla\cdot v_t, \\[10pt]

-(\xi_t\xi)_t +\xi_t^2  = k_1\xi v\cdot\nabla \xi_t + k_1\xi v_t\cdot\nabla \xi
+k_1\gamma\xi\xi_t\nabla\cdot v +k_1\gamma p\xi\nabla\cdot v_t.
\end{array}
\end{equation}

After integrating $(\ref{Sect3_XiT_1})_2$ in $\Omega$, we get
\begin{equation*}
\begin{array}{ll}
- \frac{\mathrm{d}}{\mathrm{d}t}\int\limits_{\Omega}\xi_{t}\xi \,\mathrm{d}x
+ \int\limits_{\Omega}(\xi_t)^2 \,\mathrm{d}x \\[10pt]

\lem \sqrt{\e}\mathcal{E}[\xi,v](t) + k_1\gamma \int\limits_{\partial\Omega} p\xi v_t\cdot n \,\mathrm{d}S_x
- k_1\gamma \int\limits_{\Omega}\xi v_t\cdot\nabla\xi \,\mathrm{d}x
- k_1\gamma \int\limits_{\Omega}p v_t\cdot\nabla\xi \,\mathrm{d}x
\end{array}
\end{equation*}

\begin{equation}\label{Sect3_XiT_2}
\begin{array}{ll}
\lem \sqrt{\e}\mathcal{E}[\xi,v](t) + \|v_t\|_{L^2(\Omega)}^2 + \|\nabla\xi\|_{L^2(\Omega)}^2 \\[6pt]
\lem \sqrt{\e}\mathcal{E}[\xi,v](t) + \|v_t\|_{L^2(\Omega)}^2 + \|-k_1\varrho v_t
-k_1^2 \varrho v\cdot\nabla v -ak_1\varrho v\|_{L^2(\Omega)}^2 \\[6pt]

\lem \sqrt{\e}\mathcal{E}[\xi,v](t) + \|v\|_{L^2(\Omega)}^2+ \|v_t\|_{L^2(\Omega)}^2.
\end{array}
\end{equation}

Apply $\partial_{tt}$ to $(\ref{Sect2_Final_Eq})_1$, then we get
\begin{equation}\label{Sect3_XiT_3}
\begin{array}{ll}
\hspace{1.7cm} \xi_{ttt} = - k_1 v\cdot\nabla \xi_{tt} - 2k_1 v_t\cdot\nabla \xi_t - k_1 v_{tt}\cdot\nabla \xi
-k_1\gamma\xi_{tt}\nabla\cdot v \\[6pt]\hspace{2.7cm}
-2k_1\gamma \xi_t\nabla\cdot v_t - k_1\gamma p\nabla\cdot v_{tt}, \\[10pt]

-(\xi_{tt}\xi_t)_t +\xi_{tt}^2  = k_1 \xi_t v\cdot\nabla \xi_{tt} + 2k_1 \xi_t v_t\cdot\nabla \xi_t
+ k_1 \xi_t v_{tt}\cdot\nabla \xi +k_1\gamma\xi_t\xi_{tt}\nabla\cdot v \\[6pt]\hspace{2.6cm}
+2k_1\gamma \xi_t^2\nabla\cdot v_t +k_1\gamma p\xi_t\nabla\cdot v_{tt}.
\end{array}
\end{equation}

After integrating $(\ref{Sect3_XiT_3})_2$ in $\Omega$, we get
\begin{equation}\label{Sect3_XiT_4}
\begin{array}{ll}
- \frac{\mathrm{d}}{\mathrm{d}t}\int\limits_{\Omega}\xi_{tt}\xi_t \,\mathrm{d}x
+ \int\limits_{\Omega}(\xi_{tt})^2 \,\mathrm{d}x \\[10pt]

\lem \sqrt{\e}\mathcal{E}[\xi,v](t) + k_1\gamma \int\limits_{\partial\Omega} p\xi_t v_{tt}\cdot n \,\mathrm{d}S_x
- k_1\gamma \int\limits_{\Omega}\xi_t v_{tt}\cdot\nabla\xi \,\mathrm{d}x
- k_1\gamma \int\limits_{\Omega}p v_{tt}\cdot\nabla\xi_t \,\mathrm{d}x \\[6pt]

\lem \sqrt{\e}\mathcal{E}[\xi,v](t) + \|v_{tt}\|_{L^2(\Omega)}^2 + \|\nabla\xi_t\|_{L^2(\Omega)}^2
\\[6pt]

\lem \sqrt{\e}\mathcal{E}[\xi,v](t) + \|v_{tt}\|_{L^2(\Omega)}^2 + \|-k_1\varrho v_{tt} -k_1\xi_t v_t
-k_1^2 \xi_t v\cdot\nabla v \\[6pt]\quad
-k_1^2 \varrho v_t\cdot\nabla v -k_1^2 \varrho v\cdot\nabla v_t -ak_1\xi_t v -ak_1\varrho v_t\|_{L^2(\Omega)}^2 \\[6pt]

\lem \sqrt{\e}\mathcal{E}[\xi,v](t) + \|v_t\|_{L^2(\Omega)}^2+ \|v_{tt}\|_{L^2(\Omega)}^2.
\end{array}
\end{equation}

Apply $\partial_{ttt}$ to $(\ref{Sect2_Final_Eq})_1$, then we get
\begin{equation}\label{Sect3_XiT_5}
\begin{array}{ll}
\hspace{1.9cm}\xi_{tttt} = - k_1 v\cdot\nabla \xi_{ttt} - 3k_1 v_t\cdot\nabla \xi_{tt} - 3k_1 v_{tt}\cdot\nabla \xi_t - k_1 v_{ttt}\cdot\nabla \xi \\[8pt]\hspace{2.9cm}
-k_1\gamma\xi_{ttt}\nabla\cdot v -3k_1\gamma \xi_{tt}\nabla\cdot v_t -3k_1\gamma \xi_t\nabla\cdot v_{tt}
 - k_1\gamma p\nabla\cdot v_{ttt}, \\[11pt]

-(\xi_{ttt}\xi_{tt})_t +\xi_{ttt}^2 = k_1 \xi_{tt}v\cdot\nabla \xi_{ttt} + 3k_1 \xi_{tt} v_t\cdot\nabla \xi_{tt}
+ 3k_1 \xi_{tt}v_{tt}\cdot\nabla \xi_t \\[8pt]\hspace{2.9cm}
+ k_1 \xi_{tt}v_{ttt}\cdot\nabla \xi +k_1\gamma\xi_{tt}\xi_{ttt}\nabla\cdot v +3k_1\gamma \xi_{tt}^2\nabla\cdot v_t
\\[8pt]\hspace{2.9cm}
+3k_1\gamma \xi_t\xi_{tt}\nabla\cdot v_{tt} + k_1\gamma p\xi_{tt}\nabla\cdot v_{ttt}.
\end{array}
\end{equation}

After integrating $(\ref{Sect3_XiT_5})_2$ in $\Omega$, we get
\begin{equation*}
\begin{array}{ll}
- \frac{\mathrm{d}}{\mathrm{d}t}\int\limits_{\Omega}\xi_{ttt}\xi_{tt} \,\mathrm{d}x
+ \int\limits_{\Omega}(\xi_{ttt})^2 \,\mathrm{d}x \\[10pt]

\lem \sqrt{\e}\mathcal{E}[\xi,v](t) +\int\limits_{\Omega}k_1 \xi_{tt}v\cdot\nabla \xi_{ttt}
+ k_1\gamma p\xi_{tt}\nabla\cdot v_{ttt} \,\mathrm{d}x \\[6pt]

\lem \sqrt{\e}\mathcal{E}[\xi,v](t) + k_1 \int\limits_{\partial\Omega} \xi_{ttt}\xi_{tt} v\cdot n \,\mathrm{d}S_x
- k_1 \int\limits_{\Omega}\xi_{ttt} \xi_{tt}\nabla\cdot v \,\mathrm{d}x
- k_1 \int\limits_{\Omega}\xi_{ttt} v\cdot\nabla\xi_{tt} \,\mathrm{d}x \\[6pt]\quad
+ k_1\gamma \int\limits_{\partial\Omega} p\xi_{tt} v_{ttt}\cdot n \,\mathrm{d}S_x
- k_1\gamma \int\limits_{\Omega}\xi_{tt} v_{ttt}\cdot\nabla\xi \,\mathrm{d}x
- k_1\gamma \int\limits_{\Omega}p v_{ttt}\cdot\nabla\xi_{tt} \,\mathrm{d}x \\[6pt]

\lem \sqrt{\e}\mathcal{E}[\xi,v](t) + \|v_{ttt}\|_{L^2(\Omega)}^2 + \|\nabla\xi_{tt}\|_{L^2(\Omega)}^2
\end{array}
\end{equation*}

\begin{equation}\label{Sect3_XiT_6}
\begin{array}{ll}
\lem \sqrt{\e}\mathcal{E}[\xi,v](t) + \|v_{ttt}\|_{L^2(\Omega)}^2 + \|-k_1\varrho v_{ttt} -2k_1\xi_t v_{tt}
-k_1\xi_{tt} v_t
\\[6pt]\quad
-k_1^2 \xi_{tt} v\cdot\nabla v -k_1^2 \varrho v_{tt}\cdot\nabla v -k_1^2 \varrho v\cdot\nabla v_{tt}
-2k_1^2 \xi_t v_t\cdot\nabla v -2k_1^2 \varrho v_t\cdot\nabla v_t \\[6pt]\quad
-2k_1^2 \xi_t v\cdot\nabla v_t -ak_1\xi_{tt} v - 2ak_1\xi_t v_t -ak_1\varrho v_{tt}\|_{L^2(\Omega)}^2 \\[6pt]

\lem \sqrt{\e}\mathcal{E}[\xi,v](t) + \|v_{tt}\|_{L^2(\Omega)}^2+ \|v_{ttt}\|_{L^2(\Omega)}^2.
\end{array}
\end{equation}

By $(\ref{Sect3_XiT_2})+(\ref{Sect3_XiT_4})+(\ref{Sect3_XiT_6})$, we get
\begin{equation}\label{Sect3_XiT_7}
- \frac{\mathrm{d}}{\mathrm{d}t}\int\limits_{\Omega}\sum\limits_{\ell=1}^{3}
\partial_t^{\ell}\xi\partial_t^{\ell-1}\xi \,\mathrm{d}x
+ \int\limits_{\Omega}\sum\limits_{\ell=1}^{3}(\partial_t^{\ell}\xi)^2 \,\mathrm{d}x \lem \sqrt{\e}\mathcal{E}[\xi,v](t)
+\sum\limits_{\ell=0}^{3}\|v_t^{\ell}\|_{L^2(\Omega)}^2.
\end{equation}

By Lemma $\ref{Sect2_P_Infty_Lemma}$, $\bar{p}\in[\inf\limits_{x\in\Omega}p,\sup\limits_{x\in\Omega}p]$, then for any $t\geq 0$, there exists $x_t \in\Omega$ such that $\xi(x_t,t)=0$. Assume $\ell(s)$ is a curve with finite length parameter $s$ such that $\ell(0)=x_t,\ \ell(s_x)=x$, then
\begin{equation}\label{Sect3_XiT_8}
\begin{array}{ll}
\|\xi(x,t)\|_{L^2(\Omega)}^2
= \|\xi(x_t,t) + \int\limits_0^{s_x} \nabla\xi[\ell(s)]\cdot \ell(s) \,\mathrm{d}s\|_{L^2(\Omega)}^2
\\[6pt]\hspace{2.15cm}
\leq C|Diam(\Omega)|^2\|\nabla\xi\|_{L^2(\Omega)}^2
\lem \|\nabla\xi\|_{L^2(\Omega)}^2 \\[6pt]\hspace{2.15cm}
\lem \sqrt{\e}\mathcal{E}[\xi,v](t)+\|v\|_{L^2(\Omega)}^2+\|v_t\|_{L^2(\Omega)}^2.
\end{array}
\end{equation}

Summing $(\ref{Sect3_XiT_7})$ and $(\ref{Sect3_XiT_8})$, we get
\begin{equation}\label{Sect3_XiT_9}
\begin{array}{ll}
- \frac{\mathrm{d}}{\mathrm{d}t}\int\limits_{\Omega}\sum\limits_{\ell=1}^{3}
\partial_t^{\ell}\xi\partial_t^{\ell-1}\xi \,\mathrm{d}x
+ \int\limits_{\Omega}\sum\limits_{\ell=0}^{3}(\partial_t^{\ell}\xi)^2 \,\mathrm{d}x
\lem \sqrt{\e}\mathcal{E}[\xi,v](t)
+ \sum\limits_{\ell=0}^{3}\|v_t^{\ell}\|_{L^2(\Omega)}^2.
\end{array}
\end{equation}

Then there exist two constants $C>0$, $c_3>0$ such that
\begin{equation}\label{Sect3_XiT_10}
\begin{array}{ll}
- \frac{\mathrm{d}}{\mathrm{d}t}\int\limits_{\Omega}\sum\limits_{\ell=1}^{3}
\partial_t^{\ell}\xi\partial_t^{\ell-1}\xi \,\mathrm{d}x
+ \int\limits_{\Omega}\sum\limits_{\ell=0}^{3}(\partial_t^{\ell}\xi)^2 \,\mathrm{d}x \leq C\sqrt{\e}\mathcal{E}[\xi,v](t) +c_3\sum\limits_{\ell=0}^{3}\|v_t^{\ell}\|_{L^2(\Omega)}^2.
\end{array}
\end{equation}
Thus, Lemma $\ref{Sect3_Energy_Estimate_Lemma2}$ is proved.
\end{proof}

Based on the above a priori estimates, we prove the exponential decay of $\mathcal{E}[\xi,v](t)$ and $\mathcal{E}_1[\omega](t)$ in the following lemma.
\begin{lemma}\label{Sect3_Decay_Lemma}
For any given $T\in (0,+\infty]$, if
\begin{equation*}
\sup\limits_{0\leq t\leq T} \mathcal{E}[\xi,v,\phi](t) \leq\e,
\end{equation*}
there exists $\e_2>0$, which is independent of $(\xi_0,v_0,\phi_0)$, such that if $0<\e\ll \min\{1,\e_0,\e_1,\e_2\}$, then for $\forall t\in [0,T]$,
\begin{equation}\label{Sect3_Decay_Lemma_Eq}
\begin{array}{ll}
\mathcal{E}[\xi,v](t) \leq \beta_1\|(\xi_0,v_0)\|_{H^3(\Omega)}^2\exp\{-\beta_2 t\},\\[6pt]
\mathcal{E}_1[\omega](t) \leq \beta_3\|(\xi_0,v_0)\|_{H^3(\Omega)}^2\exp\{-\beta_2 t\},
\end{array}
\end{equation}
where $\beta_1, \beta_2, \beta_3$ are three positive numbers.
\end{lemma}

\begin{proof}
In view of Lemmas $\ref{Sect3_Vorticity_Lemma}$, $\ref{Sect3_Energy_Estimate_Lemma1}$ and $\ref{Sect3_Energy_Estimate_Lemma2}$, we have obtained global a priori estimates as follows:
\begin{equation}\label{Sect3_Obtained_Estimates}
\left\{\begin{array}{ll}
\frac{\mathrm{d}}{\mathrm{d}t} \mathcal{E}_1[\omega](t) + 2a\mathcal{E}_1[\omega](t)\leq C\sqrt{\e}\mathcal{E}[\xi,v](t),
\\[6pt]
\frac{\mathrm{d}}{\mathrm{d}t}E_1[\xi,v](t)+ 2a E_1[v](t) \leq C\sqrt{\e}\mathcal{E}[\xi,v](t), \\[6pt]

- \frac{\mathrm{d}}{\mathrm{d}t}\int\limits_{\Omega}\sum\limits_{\ell=1}^{3}
\partial_t^{\ell}\xi\partial_t^{\ell-1}\xi \,\mathrm{d}x
+ E[\xi](t) \leq C\sqrt{\e}\mathcal{E}[\xi,v](t) +c_3 E[v](t).
\end{array}\right.
\end{equation}

Let $\lambda_1=\max\{\frac{4}{3},\frac{c_3}{2a}\}+1$, we define
\begin{equation}\label{Sect3_Define_E2}
\begin{array}{ll}
E_2[\xi,v](t) := \lambda_1 E_1[\xi,v](t)
-\sum\limits_{\ell=1}^{3}\int\limits_{\Omega}\partial_t^{\ell-1}\xi \partial_t^{\ell} \xi\,\mathrm{d}x.
\end{array}
\end{equation}
where $E_2> 0$ by Cauchy-Schwarz inequality. Since $\lambda_1>\frac{4}{3}$, $E_2\cong E$, i.e., there exist $c_4>0,c_5>0$ such that
\begin{equation}\label{Sect3_Equivalence}
\begin{array}{ll}
c_4 E[\xi,v](t) \leq E_2[\xi,v](t) \leq c_5 E[\xi,v](t).
\end{array}
\end{equation}

By $(\ref{Sect3_Obtained_Estimates})_2\times\lambda_1 + (\ref{Sect3_Obtained_Estimates})_3$, we get
\begin{equation}\label{Sect3_Decay_1}
\begin{array}{ll}
\frac{\mathrm{d}}{\mathrm{d}t}E_2[\xi,v](t) + (2a\lambda_1 - c_3) E_1[v](t)+ E[\xi](t)\lem \sqrt{\e}\mathcal{E}[\xi,v](t).
\end{array}
\end{equation}

Since $E_1[\xi](t)\leq c_2 E[\xi](t)$, we have
\begin{equation}\label{Sect3_Decay_2}
\begin{array}{ll}
\frac{\mathrm{d}}{\mathrm{d}t}E_2[\xi,v](t) + (2a\lambda_1 - c_3) E_1[v](t)+ \frac{1}{c_2} E_1[\xi](t)\lem \sqrt{\e}\mathcal{E}[\xi,v](t). \\[6pt]
\end{array}
\end{equation}

Let $c_6 = \min{(2a\lambda_1 - c_3, \frac{1}{c_2})}>0$, it follows from $(\ref{Sect3_Decay_2})$ that
\begin{equation}\label{Sect3_Decay_3}
\frac{\mathrm{d}}{\mathrm{d}t}E_2[\xi,v](t) + c_6 E_1[\xi,v](t) \lem \sqrt{\e}\mathcal{E}[\xi,v](t).
\end{equation}

By $(\ref{Sect3_Obtained_Estimates})_1 + (\ref{Sect3_Decay_3})$, we get
\begin{equation}\label{Sect3_Decay_4}
\begin{array}{ll}
\frac{\mathrm{d}}{\mathrm{d}t} [E_2[\xi,v](t) + \mathcal{E}_1[\omega](t)] + [c_6 E_1[\xi,v](t)+ 2a\mathcal{E}_1[\omega](t)]\lem \sqrt{\e}\mathcal{E}[\xi,v] \\[6pt]
\lem c_0\sqrt{\e}(E[\xi,v](t) + \mathcal{E}_1[\omega](t))\leq C_6c_0\sqrt{\e}(\frac{1}{c_1}E_1[\xi,v](t) + \mathcal{E}_1[\omega](t)),
\end{array}
\end{equation}
for some $C_6>0$.

\vspace{0.3cm}
Let $\e_2=\min\{\frac{c_1^2c_6^2}{4C_6^2c_0^2},\frac{a^2}{C_6^2c_0^2}\}$. When $\e<\min\{1,\e_0,\e_1,\e_2\}$, we have
\begin{equation}\label{Sect3_Decay_5}
\begin{array}{ll}
\frac{\mathrm{d}}{\mathrm{d}t} [E_2[\xi,v](t) + \mathcal{E}_1[\omega](t)] + [\frac{c_6}{2} E_1[\xi,v](t)+ a\mathcal{E}_1[\omega](t)]\leq 0, \\[6pt]
\frac{\mathrm{d}}{\mathrm{d}t} [E_2[\xi,v](t) + \mathcal{E}_1[\omega](t)] + [\frac{c_6c_1}{2c_5} E_2[\xi,v](t)+ a\mathcal{E}_1[\omega](t)]\leq 0, \\[6pt]
\frac{\mathrm{d}}{\mathrm{d}t} [E_2[\xi,v](t) + \mathcal{E}_1[\omega](t)] + c_7[E_2[\xi,v](t)+ \mathcal{E}_1[\omega](t)]\leq 0,
\end{array}
\end{equation}
where $c_7=\min\{\frac{c_6c_1}{2c_5},a\}$.

After integrating $(\ref{Sect3_Decay_5})_3$, we get
\begin{equation}\label{Sect3_Decay_6}
\begin{array}{ll}
E_2[\xi,v](t) + \mathcal{E}_1[\omega](t)\leq (E_2[\xi,v](0) + \mathcal{E}_1[\omega](0))\exp\{-c_7 t\}, \\[6pt]
\mathcal{E}_1[\omega](t)\leq (E_2[\xi,v](0) + \mathcal{E}_1[\omega](0))\exp\{-c_7 t\} \\[6pt]\hspace{1.2cm}
\leq (c_5 E[\xi,v](0) + \mathcal{E}[v](0))\exp\{-c_7 t\} \\[6pt]\hspace{1.2cm}
\leq (c_5+1)\mathcal{E}[\xi,v](0)\exp\{-c_7 t\}. \\[6pt]\hspace{1.2cm}
\leq C_7(c_5+1)\|(\xi_0,v_0)\|_{H^3(\Omega)}^2\exp\{-c_7 t\}. \\[8pt]

c_4E[\xi,v](t)\leq E_2[\xi,v](t)\leq (E_2[\xi,v](0) + \mathcal{E}_1[\omega](0))\exp\{-c_7 t\}, \\[6pt]

\mathcal{E}[\xi,v](t)\leq c_0(E[\xi,v](t) + \mathcal{E}_1[\omega](t)) \\[6pt]\hspace{1.4cm}
\leq (\frac{c_0}{c_4}+c_0)(E_2[\xi,v](0) + \mathcal{E}_1[\omega](0))\exp\{-c_7 t\} \\[6pt]\hspace{1.4cm}
\leq (\frac{c_0}{c_4}+c_0)(c_5+1)\mathcal{E}[\xi,v](0)\exp\{-c_7 t\} \\[6pt]\hspace{1.4cm}
\leq C_7(\frac{c_0}{c_4}+c_0)(c_5+1)\|(\xi_0,v_0)\|_{H^3(\Omega)}^2\exp\{-c_7 t\}
\end{array}
\end{equation}

Take $\beta_1=C_7(\frac{c_0}{c_4}+c_0)(c_5+1)$, $\beta_2=c_7$, $\beta_3=C_7(c_5+1)$, the exponential decay in $(\ref{Sect3_Decay_Lemma_Eq})$ is obtained. Thus, Lemma $\ref{Sect3_Decay_Lemma}$ is proved.
\end{proof}

Finally, we prove the uniform bound of $S-\bar{S}$ and its derivatives $\mathcal{D}^{\alpha}\partial_t^{\ell}S$ and the exponential decay of $\partial_t^{\ell}S$ and $|\partial_t S|_{\infty}$ under the condition that $v$ decays exponentially. $S-\bar{S}$ and $\mathcal{D}^{\alpha}S$ may not decay due to the transportation of $S$, but they are uniformly bounded.

The following lemma concerns the uniform bound of $\mathcal{E}[\phi](t)$ on the condition that $v$ decays exponentially.
\begin{lemma}\label{Sect3_Entropy_Lemma}
For any given $T\in (0,+\infty]$, if
\begin{equation*}
\sup\limits_{0\leq t\leq T} \mathcal{E}[\xi,v,\phi](t) \leq\e,
\end{equation*}
where $0<\e\ll \min\{1,\e_0,\e_1,\e_2\}$, then for $\forall t\in [0,T]$,
\begin{equation}\label{Sect3_Entropy_toProve_1}
\frac{\mathrm{d}}{\mathrm{d}t}\mathcal{E}[\phi](t)\leq \beta_4\mathcal{E}[v](t)^{\frac{1}{2}}\mathcal{E}[\phi](t).
\end{equation}

If $\mathcal{E}[v](t)\leq \beta_1 \|(\xi_0,v_0)\|_{H^3(\Omega)}^2\exp\{-\beta_2 t\}$, then $\mathcal{E}[\phi](t)$ has uniform bound:
\begin{equation}\label{Sect3_Entropy_toProve_2}
\mathcal{E}[\phi](t) \leq \beta_5\|\phi_0\|_{H^3(\Omega)}^2
\left(\exp\{\|(\xi_0,v_0)\|_{H^3(\Omega)}\}\right)^{c_8}.
\end{equation}
for some $c_8>0$.
\end{lemma}

\begin{proof}
Apply $\partial_t^{\ell}\mathcal{D}^{\alpha}$ to $(\ref{Sect2_Final_Eq})_3$, $\ell+|\alpha|\leq 3$. We have
\begin{equation}\label{Sect3_Entropy_1}
\begin{array}{ll}
(\partial_t^{\ell}\mathcal{D}^{\alpha}\phi)_t = - k_1\partial_t^{\ell}\mathcal{D}^{\alpha}(v\cdot\nabla\phi).
\end{array}
\end{equation}

When $\ell+|\alpha|<3$, it is easy to check $(\ref{Sect3_Entropy_toProve_1})$, since they are lower order terms.

When $\ell=0,|\alpha|=3$, assume $\alpha=\alpha_1 + \alpha_2, |\alpha_1|=1, |\alpha_2|=2$.
\begin{equation}\label{Sect3_Entropy_2}
\begin{array}{ll}
(|\mathcal{D}^{\alpha}\phi|^2)_t =
- 2 k_1 [\mathcal{D}^{\alpha}v\cdot\nabla\phi
+(\mathcal{D}^{\alpha_1}v)\cdot\nabla(\mathcal{D}^{\alpha_2}\phi)
+(\mathcal{D}^{\alpha_2}v)\cdot\nabla(\mathcal{D}^{\alpha_1}\phi) \\[6pt]\hspace{2cm}
+ v\cdot\nabla(\mathcal{D}^{\alpha}\phi)]\mathcal{D}^{\alpha}\phi.
\end{array}
\end{equation}

Integrating $(\ref{Sect3_Entropy_2})$ in $\Omega$, we get
\begin{equation}\label{Sect3_Entropy_3}
\begin{array}{ll}
\frac{\mathrm{d}}{\mathrm{d}t}\int\limits_{\Omega}|\mathcal{D}^{\alpha}\phi|^2 \,\mathrm{d}x
= \int\limits_{\Omega} [R.H.S.\ of\ (\ref{Sect3_Entropy_2})] \,\mathrm{d}x \\[6pt]

\lem \|\mathcal{D}^{\alpha}\phi\|_{L^2(\Omega)}(|\nabla\phi|_{\infty}\|\mathcal{D}^{\alpha}v\|_{L^2(\Omega)}
+ |\mathcal{D}^{\alpha_1}v|_{\infty}\|\nabla(\mathcal{D}^{\alpha_2}\phi)\|_{L^2(\Omega)}
\\[6pt]\quad
+ \|\mathcal{D}^{\alpha_2}v\|_{L^4(\Omega)}\|\nabla(\mathcal{D}^{\alpha_1}\phi)\|_{L^4(\Omega)})

- k_1 \int\limits_{\partial\Omega}|\mathcal{D}^{\alpha}\phi|^2 v\cdot n  \,\mathrm{d}S_x \\[6pt]\quad
+ k_1 \int\limits_{\Omega}|\mathcal{D}^{\alpha}\phi|^2 \nabla\cdot v \,\mathrm{d}x

\lem \mathcal{E}[v](t)^{\frac{1}{2}}\mathcal{E}[\phi](t).
\end{array}
\end{equation}

When $\ell=1,|\alpha|=2$, assume $\alpha=\alpha_1 + \alpha_2, |\alpha_1|=1, |\alpha_2|=1$.
\begin{equation}\label{Sect3_Entropy_4}
\begin{array}{ll}
(|\mathcal{D}^{\alpha}\phi_t|^2)_t =
- 2 k_1 [\mathcal{D}^{\alpha}v_t\cdot\nabla\phi +\sum\limits_{\alpha_1+\alpha_2=\alpha}(\mathcal{D}^{\alpha_1}v_t)\cdot\nabla(\mathcal{D}^{\alpha_2}\phi)
+ v_t\cdot\nabla(\mathcal{D}^{\alpha}\phi) \\[6pt]\hspace{2cm}

+\mathcal{D}^{\alpha}v\cdot\nabla\phi_t
+\sum\limits_{\alpha_1+\alpha_2=\alpha}(\mathcal{D}^{\alpha_1}v)\cdot\nabla(\mathcal{D}^{\alpha_2}\phi_t)
+ v\cdot\nabla(\mathcal{D}^{\alpha}\phi_t)]\mathcal{D}^{\alpha}\phi_t.
\end{array}
\end{equation}

Integrating $(\ref{Sect3_Entropy_4})$ in $\Omega$, we get
\begin{equation}\label{Sect3_Entropy_5}
\begin{array}{ll}
\frac{\mathrm{d}}{\mathrm{d}t}\int\limits_{\Omega}|\mathcal{D}^{\alpha}\phi_t|^2 \,\mathrm{d}x
= \int\limits_{\Omega} [R.H.S.\ of\ (\ref{Sect3_Entropy_4})] \,\mathrm{d}x \\[6pt]

\lem (|\nabla\phi|_{\infty}\|\mathcal{D}^{\alpha}v_t\|_{L^2(\Omega)}
+ |v_t|_{\infty}\|\nabla\mathcal{D}^{\alpha}\phi\|_{L^2(\Omega)}
+ \|\mathcal{D}^{\alpha_1}v_t\|_{L^4(\Omega)}\|\nabla(\mathcal{D}^{\alpha_2}\phi)\|_{L^4(\Omega)} \\[6pt]\quad

+ \|\mathcal{D}^{\alpha}v\|_{L^4(\Omega)}\|\nabla\phi_t\|_{L^4(\Omega)}
+ |\mathcal{D}^{\alpha_1}v|_{\infty}\|\nabla(\mathcal{D}^{\alpha_2}\phi_t)\|_{L^2(\Omega)})
\|\mathcal{D}^{\alpha}\phi_t\|_{L^2(\Omega)} \\[6pt]\quad

- k_1 \int\limits_{\partial\Omega}|\mathcal{D}^{\alpha}\phi_t|^2 v\cdot n  \,\mathrm{d}S_x
+ k_1 \int\limits_{\Omega}|\mathcal{D}^{\alpha}\phi_t|^2 \nabla\cdot v \,\mathrm{d}x

\lem \mathcal{E}[v](t)^{\frac{1}{2}}\mathcal{E}[\phi](t).
\end{array}
\end{equation}

When $\ell=2,|\alpha|\leq 1$,
\begin{equation}\label{Sect3_Entropy_6}
\begin{array}{ll}
(|\mathcal{D}^{\alpha}\phi_{tt}|^2)_t =
- 2 k_1 [(\mathcal{D}^{\alpha}v_{tt})\cdot\nabla\phi +v_{tt}\cdot\nabla(\mathcal{D}^{\alpha}\phi)
+2(\mathcal{D}^{\alpha}v_t)\cdot\nabla\phi_t  \\[6pt]\hspace{2.1cm}
+2v_t\cdot\nabla(\mathcal{D}^{\alpha}\phi_t) + (\mathcal{D}^{\alpha}v)\cdot\nabla\phi_{tt}
+ v\cdot\nabla(\mathcal{D}^{\alpha}\phi_{tt})]\mathcal{D}^{\alpha}\phi_{tt}.
\end{array}
\end{equation}

Integrating $(\ref{Sect3_Entropy_6})$ in $\Omega$, we get
\begin{equation}\label{Sect3_Entropy_7}
\begin{array}{ll}
\frac{\mathrm{d}}{\mathrm{d}t}\int\limits_{\Omega}|\mathcal{D}^{\alpha}\phi_{tt}|^2 \,\mathrm{d}x
= \int\limits_{\Omega}[R.H.S.\ of\ (\ref{Sect3_Entropy_6})] \,\mathrm{d}x \\[6pt]

\lem (|\nabla\phi|_{\infty}\|\mathcal{D}^{\alpha}v_{tt}\|_{L^2(\Omega)}
+ |v_{tt}|_{L^4(\Omega)}\|\nabla\mathcal{D}^{\alpha}\phi\|_{L^4(\Omega)}
+ \|\mathcal{D}^{\alpha}v_t\|_{L^4(\Omega)}\|\nabla\phi_t\|_{L^4(\Omega)}
 \\[6pt]\quad
+ |v_t|_{\infty}\|\nabla(\mathcal{D}^{\alpha}\phi_t)\|_{L^2(\Omega)}
+ |\mathcal{D}^{\alpha}v|_{\infty}\|\nabla\phi_{tt}\|_{L^2(\Omega)})
\|\mathcal{D}^{\alpha}\phi_{tt}\|_{L^2(\Omega)} \\[6pt]\quad

- k_1 \int\limits_{\partial\Omega}|\mathcal{D}^{\alpha}\phi_{tt}|^2 v\cdot n  \,\mathrm{d}S_x
+ k_1 \int\limits_{\Omega}|\mathcal{D}^{\alpha}\phi_{tt}|^2 \nabla\cdot v \,\mathrm{d}x

\lem \mathcal{E}[v](t)^{\frac{1}{2}}\mathcal{E}[\phi](t).
\end{array}
\end{equation}

When $\ell=3,|\alpha|=0$,
\begin{equation}\label{Sect3_Entropy_8}
\begin{array}{ll}
(|\phi_{ttt}|^2)_t =
- 2 k_1 [v_{ttt}\cdot\nabla\phi + 3v_{tt}\cdot\nabla\phi_t + 3v_t\cdot\nabla\phi_{tt} +v\cdot\nabla\phi_{ttt}]\phi_{ttt}.
\end{array}
\end{equation}

Integrating in $\Omega$, we get
\begin{equation}\label{Sect3_Entropy_9}
\begin{array}{ll}
\frac{\mathrm{d}}{\mathrm{d}t}\int\limits_{\Omega}|\phi_{ttt}|^2 \,\mathrm{d}x
= \int\limits_{\Omega}[R.H.S.\ of\ (\ref{Sect3_Entropy_8})] \,\mathrm{d}x \\[6pt]

\lem (|\nabla\phi|_{\infty}\|v_{ttt}\|_{L^2(\Omega)}+ |v_{tt}|_{L^4(\Omega)}\|\nabla\phi_t\|_{L^4(\Omega)}
+ |v_t|_{\infty}\|\nabla\phi_{tt}\|_{L^2(\Omega)})\|\phi_{ttt}\|_{L^2(\Omega)} \\[6pt]\quad
- k_1 \int\limits_{\partial\Omega}|\phi_{ttt}|^2 v\cdot n  \,\mathrm{d}S_x
+ k_1 \int\limits_{\Omega}|\phi_{ttt}|^2 \nabla\cdot v \,\mathrm{d}x

\lem \mathcal{E}[v](t)^{\frac{1}{2}}\mathcal{E}[\phi](t).
\end{array}
\end{equation}

In views of $(\ref{Sect3_Entropy_3}),(\ref{Sect3_Entropy_5}),(\ref{Sect3_Entropy_7}),(\ref{Sect3_Entropy_9})$, we have, for some constant $\beta_4>0$,
\begin{equation}\label{Sect3_Entropy_10}
\begin{array}{ll}
\frac{\mathrm{d}}{\mathrm{d}t}\mathcal{E}[\phi](t)\leq \beta_4\mathcal{E}[v](t)^{\frac{1}{2}}\mathcal{E}[\phi](t), \\[6pt]
\mathcal{E}[\phi](t) \leq \mathcal{E}[\phi](0)
\exp\{\int\limits_{0}^{t}\beta_4\mathcal{E}[v](\tau)^{\frac{1}{2}}\,\mathrm{d}\tau\}. \\[6pt]
\end{array}
\end{equation}

If $\mathcal{E}[v](t)\leq \beta_1 \|(\xi_0,v_0)\|_{H^3(\Omega)}^2\exp\{-\beta_2 t\}$, then
\begin{equation}\label{Sect3_Entropy_11}
\begin{array}{ll}
\mathcal{E}[\phi](t) \leq \mathcal{E}[\phi](0)
\exp\{\int\limits_{0}^{t}\beta_4\mathcal{E}[v](s)^{\frac{1}{2}}\,\mathrm{d}s\} \\[6pt]\hspace{1.1cm}
\leq \mathcal{E}[\phi](0)
\exp\{\int\limits_{0}^{t}\beta_4\sqrt{\beta_1}\|(\xi_0,v_0)\|_{H^3(\Omega)}
\exp\{-\beta_2 s\}^{\frac{1}{2}}\,\mathrm{d}s\}
\\[6pt]\hspace{1.1cm}

\leq \beta_5\|\phi_0\|_{H^3(\Omega)}^2
\exp\{\frac{2\beta_4\sqrt{\beta_1}\|(\xi_0,v_0)\|_{H^3(\Omega)}}{\beta_2}(1-\exp\{-\frac{\beta_2}{2}t\})\}
\\[6pt]\hspace{1.1cm}

\leq \beta_5\|\phi_0\|_{H^3(\Omega)}^2
\exp\{\frac{2\beta_4\sqrt{\beta_1}}{\beta_2}\|(\xi_0,v_0)\|_{H^3(\Omega)}\} \\[6pt]\hspace{1.1cm}

=\beta_5\|\phi_0\|_{H^3(\Omega)}^2
\left(\exp\{\|(\xi_0,v_0)\|_{H^3(\Omega)}\}\right)^{c_8},
\end{array}
\end{equation}
where $c_8=\frac{2\beta_4\sqrt{\beta_1}}{\beta_2}$, $\beta_5>0$.

Therefore $\mathcal{E}[\phi](t)$ is uniformly bounded when $\mathcal{E}[v](t)$ decays exponentially.
Thus, Lemma $\ref{Sect3_Entropy_Lemma}$ is proved.
\end{proof}

The following lemma concerns the exponential decay of $\sum\limits_{\ell=1}^{3}\|\partial_t^{\ell}S\|_{H^{3-\ell}(\Omega)}^2$ on the condition that $v$ decays exponentially.
\begin{lemma}\label{Sect3_Entropy_Decay_Lemma}
For any given $T\in (0,+\infty]$, if
\begin{equation*}
\sup\limits_{0\leq t\leq T} \mathcal{E}[\xi,v,\phi](t) \leq\e,
\end{equation*}
where $0<\e\ll \min\{1,\e_0,\e_1,\e_2\}$, then for $\forall t\in [0,T]$,
\begin{equation}\label{Sect3_Entropy_T_toProve}
\begin{array}{ll}
\sum\limits_{\ell=1}^{3}\|\partial_t^{\ell}\phi\|_{H^{3-\ell}(\Omega)}^2
\lem \mathcal{E}[v](t)\mathcal{E}[\phi](t) \\[6pt]
\leq c_9\|(\xi_0,v_0)\|_{H^3(\Omega)}^2\|\phi_0\|_{H^3(\Omega)}^2
\left(\exp\{\|(\xi_0,v_0)\|_{H^3(\Omega)}\}\right)^{c_8}\exp\{-\beta_2 t\},
\end{array}
\end{equation}
for some $c_9>0$.
\end{lemma}

\begin{proof}
It follows from Lemma $\ref{Sect3_Decay_Lemma}$ that
$\mathcal{E}[v](t)\leq \beta_1\|(\xi_0,v_0)\|_{H^3(\Omega)}^2\exp\{-\beta_2 t\}$.
It follows from Lemma $\ref{Sect3_Entropy_Lemma}$ that
$\mathcal{E}[\phi](t) \leq \beta_5\|\phi_0\|_{H^3(\Omega)}^2
\left(\exp\{\|(\xi_0,v_0)\|_{H^3(\Omega)}\}\right)^{c_8}$. \\[10pt]
\indent
By $\phi_t = -k_1 v\cdot\nabla\phi$, we get
\begin{equation}\label{Sect3_Entropy_T_Prove1}
\begin{array}{ll}
\|\phi_t\|_{H^{2}(\Omega)}^2 =
k_1^2 \sum\limits_{|\alpha|\leq 2}\int\limits_{\Omega}|\mathcal{D}^{\alpha}v|^2 |\nabla\phi|^2 \,\mathrm{d}x
+k_1^2 \sum\limits_{|\alpha|\leq 2}
\int\limits_{\Omega}|v|^2 |\mathcal{D}^{\alpha}\nabla\phi|^2 \,\mathrm{d}x \\[6pt]\hspace{2cm}

+ k_1^2\sum\limits_{|\alpha_1|\leq 1,|\alpha_2|\leq 1}\int\limits_{\Omega}|\mathcal{D}^{\alpha_1}v|^2 |\mathcal{D}^{\alpha_2}\nabla\phi|^2 \,\mathrm{d}x \\[6pt]\hspace{1.6cm}

\lem \sum\limits_{|\alpha|\leq 2}|\nabla\phi|_{\infty}^2\|\mathcal{D}^{\alpha}v\|_{L^2(\Omega)}^2
+ \sum\limits_{|\alpha|\leq 2}|v|_{\infty}^2\|\mathcal{D}^{\alpha}\nabla\phi\|_{L^2(\Omega)}^2 \\[6pt]\hspace{2cm}
+ \sum\limits_{|\alpha_1|\leq 1,|\alpha_2|\leq 1}|\mathcal{D}^{\alpha_1} v|_{\infty}^2\|\mathcal{D}^{\alpha_2}\nabla\phi\|_{L^2(\Omega)}^2
\\[6pt]\hspace{1.6cm}

\lem \mathcal{E}[v](t)\mathcal{E}[\phi](t).
\end{array}
\end{equation}
where $0\leq|\alpha_1|\leq 1,0\leq|\alpha_2|\leq 1$.

\vspace{0.3cm}
By $\phi_{tt} = -k_1 v_t\cdot\nabla\phi -k_1 v\cdot\nabla\phi_t$, we get
\begin{equation}\label{Sect3_Entropy_T_Prove2}
\begin{array}{ll}
\|\phi_{tt}\|_{H^{1}(\Omega)}^2 \leq k_1^2\sum\limits_{0\leq|\alpha|\leq 1}\int\limits_{\Omega}
(\mathcal{D}^{\alpha}v_t\cdot\nabla\phi)^2 +(v_t\cdot\nabla\mathcal{D}^{\alpha}\phi)^2
+(\mathcal{D}^{\alpha}v\cdot\nabla\phi_t)^2 \\[6pt]\hspace{2.1cm}
+(v\cdot\nabla\mathcal{D}^{\alpha}\phi_t)^2 \,\mathrm{d}x, \\[6pt]\hspace{1.7cm}

\lem \sum\limits_{0\leq|\alpha|\leq 1}(
|\nabla\phi|_{\infty}^2\|\mathcal{D}^{\alpha}v_t\|_{L^2(\Omega)}^2
+|v_t|_{\infty}^2\|\nabla\mathcal{D}^{\alpha}\phi\|_{L^2(\Omega)}^2 \\[6pt]\hspace{2.1cm}
+|\mathcal{D}^{\alpha}v|_{\infty}^2\|\nabla\phi_t\|_{L^2(\Omega)}^2
+|v|_{\infty}^2\|\nabla\mathcal{D}^{\alpha}\phi_t\|_{L^2(\Omega)}^2) \\[6pt]\hspace{1.7cm}
\lem \mathcal{E}[v](t)\mathcal{E}[\phi](t).
\end{array}
\end{equation}

By $\phi_{ttt} = -k_1 v_{tt}\cdot\nabla\phi -2k_1 v_t\cdot\nabla\phi_t -k_1 v\cdot\nabla\phi_{tt}$, we get
\begin{equation}\label{Sect3_Entropy_T_Prove3}
\begin{array}{ll}
\|\phi_{ttt}\|_{L^{2}(\Omega)}^2 \leq k_1^2\int\limits_{\Omega}(v_{tt}\cdot\nabla\phi)^2 \,\mathrm{d}x
+ 4k_1^2\int\limits_{\Omega}(v_t\cdot\nabla\phi_t)^2 \,\mathrm{d}x
+ k_1^2\int\limits_{\Omega}(v\cdot\nabla\phi_{tt})^2 \,\mathrm{d}x \\[6pt]\hspace{1.7cm}
\lem |\nabla\phi|_{\infty}^2\|v_{tt}\|_{L^2(\Omega)}^2
+ |v_t|_{\infty}^2\|\nabla\phi_t\|_{L^2(\Omega)}^2
+ |v|_{\infty}^2\|\nabla\phi_{tt}\|_{L^2(\Omega)}^2 \\[6pt]\hspace{1.7cm}
\lem \mathcal{E}[v](t)\mathcal{E}[\phi](t).
\end{array}
\end{equation}

Summing $(\ref{Sect3_Entropy_T_Prove1}),(\ref{Sect3_Entropy_T_Prove2}),(\ref{Sect3_Entropy_T_Prove3})$, we get
\begin{equation}\label{Sect3_Entropy_T_Prove4}
\begin{array}{ll}
\sum\limits_{\ell=1}^{3}\|\partial_t^{\ell}\phi\|_{H^{3-\ell}(\Omega)}^2
\lem \mathcal{E}[v](t)\mathcal{E}[\phi](t) \\[6pt]\quad
\leq c_9\|(\xi_0,v_0)\|_{H^3(\Omega)}^2\|\phi_0\|_{H^3(\Omega)}^2
\left(\exp\{\|(\xi_0,v_0)\|_{H^3(\Omega)}\}\right)^{c_8}\exp\{-\beta_2 t\},
\end{array}
\end{equation}
where $c_9=\beta_1\beta_5$. Thus, Lemma $\ref{Sect3_Entropy_Decay_Lemma}$ is proved.
\end{proof}

\begin{remark}\label{Sect3_Varrho_Remark}
When $\mathcal{E}[p-\bar{p}](t)$ and $\mathcal{E}[S-\bar{S}](t)$ are uniformly bounded, $\mathcal{E}[\varrho-\bar{\varrho}](t)$ is also uniformly bounded due to $\varrho = \frac{1}{\sqrt[\gamma]{A}}p^{\frac{1}{\gamma}}\exp\{-\frac{S}{\gamma}\}$.

After differentiating this formula with respect to $t$, we have
\begin{equation}\label{Sect3_Estimate_Density_1}
\begin{array}{ll}
\|\partial_t\varrho\|_{H^2(\Omega)}^2
\leq \beta_1\|(\xi_0,v_0)\|_{H^3(\Omega)}^2\exp\{-\beta_2 t\} \\[6pt]\hspace{2cm}
+c_9\|(\xi_0,v_0)\|_{H^3(\Omega)}^2\|\phi_0\|_{H^3(\Omega)}^2
\left(\exp\{\|(\xi_0,v_0)\|_{H^3(\Omega)}\}\right)^{c_8}\exp\{-\beta_2 t\},
\end{array}
\end{equation}
Similarly,
\begin{equation}\label{Sect3_Estimate_Density_2}
\begin{array}{ll}
\|\partial_{tt}\varrho\|_{H^1(\Omega)}^2
\leq C\exp\{-\beta_2 t\} +C\exp\{-2\beta_2 t\}, \\[6pt]

\|\partial_{ttt}\varrho\|_{L^2(\Omega)}^2
\leq C\exp\{-\beta_2 t\} + C\exp\{-2\beta_2 t\} \ + C\exp\{-3\beta_2 t\}.
\end{array}
\end{equation}

Thus, for any given $T\in (0,+\infty]$, if $\sup\limits_{0\leq t\leq T} \mathcal{E}[\xi,v,\phi](t) \leq\e$,
where $0<\e\ll \min\{1,\e_0,\e_1,\e_2\}$, then $\sum\limits_{\ell=1}^{3}\|\partial_t^{\ell}\varrho\|_{H^{3-\ell}(\Omega)}^2$ also decays at an exponential rate of $C\exp\{-\beta_2 t\}$.
\end{remark}

\section{Global Existence and Equilibrium States of Non-Isentropic Euler Equations with Damping}
In this section, we prove the global existence of classical solutions to the non-isentropic Euler equations with damping $(\ref{Sect2_Final_Eq})$ under small data assumption and the singularity formation for a class of large data.

The proof of local existence of classical solutions to IBVP $(\ref{Sect2_Final_Eq})$ is standard (see \cite{Majda_1984},\cite{Schochet_1986}), so we give a lemma on the local existence without proof here.

\begin{lemma}\label{Sect4_LocalExistence}
$(Local\ Existence)$\\[6pt]
If $(\xi_0,v_0,\phi_0)\in H^3(\Omega)$, $\inf\limits_{x\in\Omega}p_0(x)>0$ and $\partial_t^{\ell} v(x,0)\cdot n|_{\partial\Omega}=0$, $0\leq \ell\leq 3$, then there exists a finite time $T_{\ast}>0$, such that IBVP (\ref{Sect2_Final_Eq}) admits a unique local classical solution $(\xi,v,\phi)\in \underset{0\leq \ell\leq 3}{\cap}C^{\ell}([0,T_{\ast}),H^{3-\ell}(\Omega))$.
\end{lemma}

Based on the global a priori estimates for $(\xi,v,\phi)$, we obtained the global existence of classical solutions to IBVP $(\ref{Sect2_Final_Eq})$.

\begin{theorem}\label{Sect4_GlobalExistence_Thm}
$(Global\ Existence)$\\[6pt]
Assume $(\xi_0,v_0,\phi_0)\in H^3(\Omega)$, $\inf\limits_{x\in\Omega}p_0(x)>0$, $\partial_t^{\ell} v(x,0)\cdot n|_{\partial\Omega}=0$, $0\leq \ell\leq 3$.
There exists a sufficiently small number $\delta_1>0$, such that if $\|\xi_0,v_0,\phi_0\|_{H^3(\Omega)}\leq \delta_1$, then IBVP $(\ref{Sect2_Final_Eq})$ admits a unique global classical solution $$(\xi,v,\phi)\in \underset{0\leq \ell\leq 3}{\cap}C^{\ell}([0,+\infty),H^{3-\ell}(\Omega)),$$
moreover, $\varrho=\varrho(\xi,\phi)\in \underset{0\leq \ell\leq 3}{\cap}C^{\ell}([0,+\infty),H^{3-\ell}(\Omega)).$
$\forall t\geq 0$, $\mathcal{E}[\xi,v](t)$, $\mathcal{E}_1[\omega](t)$ and $\sum\limits_{\ell=1}^{3}\|\partial_t^{\ell}\phi\|_{H^{3-\ell}(\Omega)}^2$ decays exponentially,
$\mathcal{E}[\phi](t)$ is uniformly bounded.
\end{theorem}

\begin{proof}
In view of Lemmas $\ref{Sect3_Decay_Lemma}$ and $\ref{Sect3_Entropy_Lemma}$, we have the following global a priori estimates: for any given $T\in (0,+\infty]$, if
\begin{equation}\label{Sect4_Assumption_1}
\sup\limits_{0\leq t\leq T} \mathcal{E}[\xi,v,\phi](t) \leq\e,
\end{equation}
where $0<\e\ll \min\{1,\e_0,\e_1,\e_2\}$, then
\begin{equation}\label{Sect4_Decay}
\begin{array}{ll}
\mathcal{E}[\xi,v](t)\leq \beta_1\|(\xi_0,v_0)\|_{H^3(\Omega)}^2\exp\{-\beta_2 t\}, \\[6pt]
\mathcal{E}[\phi](t)\leq \beta_5\|\phi_0\|_{H^3(\Omega)}^2
\left(\exp\{\|(\xi_0,v_0)\|_{H^3(\Omega)}\}\right)^{c_8}.
\end{array}
\end{equation}

The constants $\e_0,\e_1,\e_2$ are independent of $(\xi_0,v_0,\phi_0)$, so we can choose $\e$ which is independent of $(\xi_0,v_0,\phi_0)$.

Take $\delta_1=\min\{\sqrt{\e},\sqrt{\frac{\e}{2\beta_1}},\sqrt{\frac{\e}{2\beta_5}}
\left(\exp\{\sqrt{\frac{\e}{2\beta_1}}\}\right)^{-\frac{c_8}{2}}\}$, then if $\mathcal{E}[\xi,v,\phi](0)\leq\delta_1$, we have
\begin{equation}\label{Sect4_Data_Condition}
\left\{\begin{array}{ll}
\|(\xi_0,v_0)\|_{H^3(\Omega)} \leq \sqrt{\frac{\e}{2\beta_1}}, \\[6pt]
\|\phi_0\|_{H^3(\Omega)} \leq \sqrt{\frac{\e}{2\beta_5}}\left(\exp\{\sqrt{\frac{\e}{2\beta_1}}\}\right)^{-\frac{c_8}{2}}.
\end{array}\right.
\end{equation}

Due to the estimates in $(\ref{Sect4_Decay})$, the solutions $(\xi,v,\phi)$ satisfy
\begin{equation}\label{Sect4_Solution_Condition}
\begin{array}{ll}
\mathcal{E}[\xi,v](t)\leq \frac{\e}{2},\quad \mathcal{E}[\phi](t)\leq \frac{\e}{2},\quad \forall t\in [0,T].
\end{array}
\end{equation}
This implies the a priori assumption $(\ref{Sect4_Assumption_1})$ is satisfied, the validity of the former a priori estimates is verified.

Due to the global a priori estimates for $(\xi,v,\phi)$ and Lemma $\ref{Sect4_LocalExistence}$ on the local existence result, the classical solution $(\xi,v,\phi)$ can be extended to $[0,+\infty)$. Thus, Theorem $\ref{Sect4_GlobalExistence_Thm}$ on the global existence of classical solutions to IBVP (\ref{Sect2_Final_Eq}) is proved.
\end{proof}

\begin{remark}\label{Sect4_Friction_Coefficient}
Our proof requires $a\geq C\sqrt{\e}$ where $C>0$ is large enough. If $a\rto 0$,  $(p_0,u_0)\rto(\bar{p},0)$ is required.
\end{remark}

Since $(\xi,v,\phi)\in C^1(\Omega\times[0,+\infty))$ is the global classical solution to IBVP $(\ref{Sect2_Final_Eq})$, then $(p=\bar{p}+\xi,u=k_1v,S=\bar{S}+\phi)$ is the global classical solution to IBVP for non-isentropic Euler equations with damping $(\ref{Sect1_NonIsentropic_EulerEq})$.
The following theorem describes the asymptotical behavior of $(p,v,S,\varrho)$ relating to their equilibrium states $(p_{\infty},v_{\infty},S_{\infty},\varrho_{\infty})$.

\begin{theorem}\label{Sect4_Convergence_Rate_Thm}
Assume the conditions in Theorem $\ref{Sect4_GlobalExistence_Thm}$ hold. Let  $(p,u,S)\in C^1(\Omega\times[0,+\infty))$ be the global classical solution to IBVP $(\ref{Sect1_NonIsentropic_EulerEq})$.
$p_{\infty}=\bar{p}$, $u_{\infty}=v_{\infty}=\omega_{\infty}=0$. If $S_0\neq const$, then $S_{\infty}\neq const$.
$\varrho_{\infty}(x)\neq const$, the temperature $\theta_{\infty}(x)\neq const$, the internal energy $e_{\infty}(x)\neq const$.
As $t\rto +\infty$, $(p,u,S,\varrho)$ converge to $(\bar{p},0,S_{\infty},\varrho_{\infty})$ exponentially in $|\cdot|_{\infty}$ norm.
\end{theorem}

\begin{proof}
$|\nabla p|_{\infty}+|\nabla v|_{\infty}\lem \mathcal{E}[\xi,v](t)^{\frac{1}{2}}\lem \|(\xi_0,v_0)\|_{H^3(\Omega)}\exp\{-\frac{\beta_2}{2} t\}\rto 0$, as $t\rto +\infty$.
Thus $p_{\infty}=\bar{p},\ v_{\infty}=u_{\infty}=0,\omega_{\infty}=0$.

By Lemma $\ref{Sect3_Entropy_Decay_Lemma}$, we have
\begin{equation}\label{Sect4_S_Infty_1}
\begin{array}{ll}
|S_t|_{\infty}\lem
\left( \sum\limits_{1\leq\ell\leq 3}\|\partial_t^{\ell}S\|_{H^{3-\ell}}^2 \right)^{\frac{1}{2}}
\\[6pt]\hspace{0.9cm}
\lem \|(\xi_0,v_0)\|_{H^3(\Omega)}\|\phi_0\|_{H^3(\Omega)}
\left(\exp\{\|(\xi_0,v_0)\|_{H^3(\Omega)}\}\right)^{\frac{c_8}{2}} \exp\{-\frac{\beta_2}{2} t\}.
\end{array}
\end{equation}

So $\int\limits_{0}^{\infty}S_s(x,s) \,\mathrm{d}s$ converges, then
$S_{\infty}(x)=S_0(x) + \int\limits_{0}^{\infty}S_s(x,s) \,\mathrm{d}s$
is bounded. Thus, $S_{\infty}(x)$ exists uniquely.

In order to prove that $S_{\infty}\neq const$ if $S_0\neq const$, we assume $S_{\infty} = const$.

For any $\alpha\in \mathbb{R}$, we have
\begin{equation}\label{Sect4_S_Infty_2}
\begin{array}{ll}
(\varrho \exp\{\frac{S}{\gamma} +\alpha S \})_t + \nabla\cdot (\varrho \exp\{\frac{S}{\gamma} +\alpha S \}u),\\[10pt]
\frac{\mathrm{d}}{\mathrm{d}t}\int\limits_{\Omega} \varrho \exp\{\frac{S}{\gamma} +\alpha S \}\,\mathrm{d}x =0, \\[10pt]
\int\limits_{\Omega} \varrho_{\infty} \exp\{\frac{S_{\infty}}{\gamma} +\alpha S_{\infty} \}\,\mathrm{d}x
= \int\limits_{\Omega} \varrho_0 \exp\{\frac{S_0}{\gamma} +\alpha S_0 \}\,\mathrm{d}x, \\[10pt]

\int\limits_{\Omega} p_{\infty}^{\frac{1}{\gamma}} \exp\{\alpha S_{\infty} \}\,\mathrm{d}x
= \int\limits_{\Omega} p_0^{\frac{1}{\gamma}} \exp\{\alpha S_0 \}\,\mathrm{d}x.
\end{array}
\end{equation}

By assumption $S_{\infty}=const$, then we have
\begin{equation}\label{Sect4_S_Infty_3}
\begin{array}{ll}
\int\limits_{\Omega} p_0^{\frac{1}{\gamma}} (\exp\{S_0 -S_{\infty}\})^{\alpha}\,\mathrm{d}x
=\int\limits_{\Omega} p_{\infty}^{\frac{1}{\gamma}} \,\mathrm{d}x
=\int\limits_{\Omega} p_0^{\frac{1}{\gamma}} \,\mathrm{d}x, \\[10pt]

\left(\int\limits_{\Omega} (p_0^{\frac{1}{\alpha\gamma}} \exp\{S_0 -S_{\infty}\})^{\alpha}\,\mathrm{d}x \right)^{\frac{1}{\alpha}}
=\left(\int\limits_{\Omega} p_0^{\frac{1}{\gamma}} \,\mathrm{d}x \right)^{\frac{1}{\alpha}}
\end{array}
\end{equation}

When $\alpha>0$, for any $0<\delta\ll 1$, there exists $\alpha \geq \max\{\frac{\log p_0}{\gamma \log(1+\delta)},\frac{\log p_0}{\gamma \log(1-\delta)}\}$ such that $1-\delta \leq p_0^{\frac{1}{\alpha\gamma}} \leq 1+\delta$, then
\begin{equation}\label{Sect4_S_Infty_4}
\begin{array}{ll}
(1-\delta)\left(\int\limits_{\Omega} (\exp\{S_0 -S_{\infty}\})^{\alpha}\,\mathrm{d}x \right)^{\frac{1}{\alpha}}
\leq\left(\int\limits_{\Omega} p_0^{\frac{1}{\gamma}} \,\mathrm{d}x \right)^{\frac{1}{\alpha}}
\\[8pt]\hspace{5.3cm}
\leq (1+\delta)\left(\int\limits_{\Omega} (\exp\{S_0 -S_{\infty}\})^{\alpha}\,\mathrm{d}x \right)^{\frac{1}{\alpha}}
\end{array}
\end{equation}

Let $\delta\rto 0$, $\alpha\rto +\infty$, we have $\|\exp\{S_0 -S_{\infty}\}\|_{\infty} =1$.

When $\alpha<0$, for any $0<\delta\ll 1$, there exists $\alpha \leq \min\{-\frac{\log p_0}{\gamma \log(1+\delta)},-\frac{\log p_0}{\gamma \log(1-\delta)}\}$ such that $1-\delta \leq p_0^{\frac{1}{-\alpha\gamma}} \leq 1+\delta$, then
\begin{equation}\label{Sect4_S_Infty_5}
\begin{array}{ll}
(1-\delta)\left(\int\limits_{\Omega} (\exp\{S_{\infty}-S_0\})^{-\alpha}\,\mathrm{d}x \right)^{\frac{1}{-\alpha}}
\leq\left(\int\limits_{\Omega} p_0^{\frac{1}{\gamma}} \,\mathrm{d}x \right)^{\frac{1}{-\alpha}}
\\[8pt]\hspace{5.7cm}
\leq (1+\delta)\left(\int\limits_{\Omega} (\exp\{S_{\infty}-S_0\})^{-\alpha}\,\mathrm{d}x \right)^{\frac{1}{-\alpha}}
\end{array}
\end{equation}

Let $\delta\rto 0$, $\alpha\rto -\infty$, we have $\|\exp\{S_{\infty}- S_0\}\|_{\infty} =1$.

So, $S_0 \equiv S_{\infty}= const$, it contradicts with the assumption $S_0\neq const$. Thus, we proved that $S_{\infty}\neq const$ if $S_0\neq const$.

Moreover, $\varrho_{\infty}(x) = \frac{1}{\sqrt[\gamma]{A}}\bar{p}^{\frac{1}{\gamma}}\exp\{-\frac{S_{\infty}(x)}{\gamma}\}\neq const$, due to the pressure law $(\ref{Sect1_Pressure})$. $\theta_{\infty}(x)=\frac{\bar{p}}{\mathcal{R}\varrho_{\infty}(x)}\neq const$, where $\mathcal{R}$ is universal gas constant. $e_{\infty}(x)=C_V \theta_{\infty}(x)\neq const$, were $C_V>0$ is constant.

\vspace{0.2cm}
The exponential decay rates of $(\xi,v,\phi_{t})$ provides exponential convergence rates of $(p,u,S,\varrho)$ to their equilibrium states as follows:
\begin{equation}\label{Sect4_Convergence_Rate}
\left\{\begin{array}{ll}
|p-p_{\infty}|_{\infty} =|p-\bar{p}|_{\infty}\lem \|(\xi_0,v_0)\|_{H^3(\Omega)}\exp\{-\frac{\beta_2}{2} t\}, \\[8pt]
|u-0|_{\infty}= k_1|v|_{\infty}\lem \|(\xi_0,v_0)\|_{H^3(\Omega)}\exp\{-\frac{\beta_2}{2} t\}, \\[8pt]

|S(x,t)-S_{\infty}(x)|_{\infty} = |-\int\limits_{t}^{\infty}S_s(x,s) \,\mathrm{d}s|_{\infty}
\leq \int\limits_{t}^{\infty} |\phi_s(x,s)|_{\infty} \,\mathrm{d}s \\[8pt]\qquad
\lem \int\limits_{t}^{\infty} \|(\xi_0,v_0)\|_{H^3(\Omega)}\|\phi_0\|_{H^3(\Omega)}
\left(\exp\{\|(\xi_0,v_0)\|_{H^3(\Omega)}\}\right)^{\frac{c_8}{2}} \exp\{-\frac{\beta_2}{2} s\} \,\mathrm{d}s \\[8pt]\qquad
\lem \|(\xi_0,v_0)\|_{H^3(\Omega)}\|\phi_0\|_{H^3(\Omega)}
\left(\exp\{\|(\xi_0,v_0)\|_{H^3(\Omega)}\}\right)^{\frac{c_8}{2}}\exp\{-\frac{\beta_2}{2} t\}, \\[10pt]

|\varrho(x,t)-\varrho_{\infty}(x)|_{\infty}\lem \exp\{-\frac{\beta_2}{2} t\}.
\end{array}\right.
\end{equation}

So $(p,u,S,\varrho)\rto (\bar{p},0,S_{\infty},\varrho_{\infty})$ exponentially in $|\cdot|_{\infty}$ norm as $t\rto +\infty$.
\end{proof}

\begin{remark}\label{Sect4_Eigenvalue_Remark}
For Cauchy problem, it is easier to understand that the equilibrium states are not constant states. The linear equations of the nonlinear equations in $(\ref{Sect2_Final_Eq})$ are
\begin{equation}\label{Sect4_Linear_Eq}
\left\{\begin{array}{ll}
\xi_t + k_2 \nabla\cdot v = 0, \\[6pt]
v_t + k_2 \nabla\xi + a v = 0, \\[6pt]
\phi_t = 0.
\end{array}\right.
\end{equation}

Let $\eta\in\mathbb{R}^3$ is a vector in Fourier space while $x$ is a vector in physical space.
After Fourier transformation of $(\ref{Sect4_Linear_Eq})$, we get
\begin{equation}\label{Sect4_Decay_FourierT_Matrix_1}
\partial_t \left(\begin{array}{c} \hat{\xi}(k,t) \\ \hat{v}(k,t) \\ \hat{\phi}(k,t) \end{array}\right) =
\left(\begin{array}{ccc} 0 & -ik_2\eta^{\top} & 0 \\
-ik_2\eta & -a\mathbf{I}_3 & 0  \\
0 & 0 & 0 \end{array}\right)
\left(\begin{array}{c} \hat{\xi}(k,t) \\ \hat{v}(k,t) \\ \hat{\phi}(k,t) \end{array}\right),
\end{equation}
where the coefficient matrix is denoted by $\mathbf{M}(\eta)$.

Then the eigenvalues of $\mathbf{M}(\eta)$ satisfy the following equation
\begin{equation}\label{Sect4_Decay_FourierT_Matrix_3}
\begin{array}{ll}
|\lambda \mathbf{I}_{5} - \mathbf{M}(\eta)| =
\left|\begin{array}{ccc} \lambda & ik_2\eta^{\top} & 0 \\
ik_2\eta & (\lambda+a)\mathbf{I}_3 & 0  \\
0 & 0 & \lambda \end{array}\right|
= \lambda \left|\begin{array}{ccc} \lambda & ik_2\eta^{\top} \\
ik_2\eta & (\lambda+a)\mathbf{I}_3  \\
\end{array}\right| =0.
\end{array}
\end{equation}

So the coefficient matrix $\mathbf{M}(\eta)$ has one eigenvalue $\lambda_5 =0$, other eigenvalues have negative reals. $\lambda_5=0$ corresponds to its eigenvector $(0,\cdots,0,1)^{\top}$, thus $\phi$ may not decay to zero.
\end{remark}

However, the damping effect on the velocity is weakly dissipative, i.e., there are a class of initial data such that the classical solutions blows up. Our proof is based on the analysis of the moment $M_{\varrho}(t)$ and the finite propagation speed of $(\xi,v,\phi)$. So, we need to prove the following lemma which states the classical solutions possess finite propagation speed, especially near the boundary.
\begin{lemma}\label{Sect4_Speed_Lemma}
Assume $\inf\limits_{x\in\Omega}p_0(x)>0$, $(\xi,v,\phi)\in C^1(\Omega\times[0,\tau))$ is the classical solutions to IBVP $(\ref{Sect2_Final_Eq})$, $(p,u,S)\in C^1(\Omega\times[0,\tau))$ is the classical solutions to IBVP $(\ref{Sect1_NonIsentropic_EulerEq})$. For any $x\in\Omega\cup\partial\Omega, 0\leq t_1<t_2<\tau$, if $(\xi,v,\phi)(y,t_1)=0$ in $\{|y-x|\leq k_2(t_2-t_1)\}\cap\Omega$, then $(\xi,v,\phi)(x,t_2)=0$.
Equivalently, if $(p-\bar{p},u,S-\bar{S})(y,t_1)=0$ in $\{|y-x|\leq k_2(t_2-t_1)\}\cap\Omega$, then $(p-\bar{p},u,S-\bar{S})(x,t_2)=0$.
\end{lemma}

\begin{proof}
For any fixed $(x,t_2)\in (\Omega\cup\partial\Omega)\times \{t_2\}$, we define
the intersection of a truncated cone and $\Omega$ as
\begin{equation}\label{Sect4_Speed_Cone}
\mathcal{O}(s) := \{(y,\iota)|\ |y-x|\leq k_2(t_2 - \iota), t_1\leq\iota\leq s\leq t_2\}\cap\Omega.
\end{equation}
and the energy at the time $s$
\begin{equation}\label{Sect4_Speed_Energy}
\begin{array}{ll}
e(s) = \int\limits_{\{|y-x|\leq k_2(t_2-s)\}\cap\Omega}\xi^2 + |v|^2 + \phi^2 \,\mathrm{d}y.
\end{array}
\end{equation}

By $(\ref{Sect2_Final_Eq})\cdot(\xi,v,\phi)$, we get
\begin{equation}\label{Sect4_Speed_1}
\begin{array}{ll}
(\xi^2 + |v|^2 + \phi^2)_t + 2k_2\xi\nabla\cdot v + 2k_2 v\cdot\nabla\xi + 2a |v|^2 \\[6pt]
= -2 \gamma k_1\xi^2\nabla\cdot v - 2 k_1 \xi v\cdot\nabla \xi - 2k_1 v\cdot\nabla v\cdot v
+ \frac{2}{k_1} (\frac{1}{\bar{\varrho}}- \frac{1}{\varrho})v\cdot\nabla\xi -2 k_1\phi v\cdot\nabla\phi.
\end{array}
\end{equation}

After integrating $(\ref{Sect4_Speed_1})$ in $\mathcal{O}(s)$, we get
\begin{equation}\label{Sect4_Speed_2}
\begin{array}{ll}
\int\limits_{\{|y-x|\leq k_2(t_2-s)\}\cap\Omega}\xi^2 + |v|^2 + \phi^2 \,\mathrm{d}y
- \int\limits_{\{|y-x|\leq k_2(t_2-t_1)\}\cap\Omega}\xi^2 + |v|^2 + \phi^2 \,\mathrm{d}y  \\[15pt]\quad

+\frac{1}{\sqrt{1+k_2^2}} \int\limits_{t_1}^{s}\int\limits_{\{|y-x|= k_2(t-\iota)\}\cap\Omega}
k_2(\xi^2 + |v|^2 + \phi^2) + 2k_2 \xi v \cdot \frac{y-x}{|y-x|} \,\mathrm{d}y \,\mathrm{d}\iota \\[15pt]\quad

+ \int\limits_{t_1}^{s}\int\limits_{\partial\Omega\cap \{|y-x|\leq k_2(t-\iota)\}}
2k_2 \xi v \cdot n \,\mathrm{d}S_y \,\mathrm{d}\iota
+ 2a \int\limits_{\mathcal{O}(s)}|v|^2 \,\mathrm{d}y\,\mathrm{d}\iota \\[18pt]

= -2\iint\limits_{\mathcal{O}(s)} \gamma k_1\xi^2\nabla\cdot v + k_1 \xi v\cdot\nabla \xi
+ k_1 v\cdot\nabla v\cdot v - \frac{1}{k_1} (\frac{1}{\bar{\varrho}}- \frac{1}{\varrho})v\cdot\nabla\xi \\[15pt]\quad
+ k_1\phi v\cdot\nabla\phi \,\mathrm{d}y \,\mathrm{d}\iota

\leq C\max\limits_{\mathcal{O}(t_2)}\{\nabla\xi, \nabla v, \nabla\phi\}
\iint\limits_{\mathcal{O}(s)} \xi^2 + |v|^2 +\phi^2 \,\mathrm{d}y \,\mathrm{d}\iota.
\end{array}
\end{equation}

While we have the following three estimates:
\begin{equation}\label{Sect4_Speed_3}
\begin{array}{ll}
\int\limits_{t_1}^{s}\int\limits_{\{|y-x|= k_2(t-\iota)\}\cap\Omega}
k_2(\xi^2 + |v|^2 + \phi^2) + 2k_2 \xi v \cdot \frac{y-x}{|y-x|} \,\mathrm{d}y \,\mathrm{d}\iota
\geq 0, \\[14pt]

\int\limits_{t_1}^{s}\int\limits_{\partial\Omega\cap \{|y-x|\leq k_2(t-\iota)\}}
2k_2 \xi v \cdot n \,\mathrm{d}S_y \,\mathrm{d}\iota =0, \\[14pt]

\iint\limits_{\mathcal{O}(s)}|v|^2 \,\mathrm{d}y\,\mathrm{d}\iota \geq 0.
\end{array}
\end{equation}

By $(\ref{Sect4_Speed_2})$ and $(\ref{Sect4_Speed_3})$, we get
\begin{equation}\label{Sect4_Speed_4}
\begin{array}{ll}
e(s)- e(t_1) \leq C\max\limits_{\mathcal{O}(t_2)}\{\nabla\xi, \nabla v, \nabla\phi\}
\int\limits_{t_1}^{s} e(\iota) \,\mathrm{d}\iota.
\end{array}
\end{equation}

By Gronwall's inequality, for $\forall s\in [t_1,t_2]$,
\begin{equation}\label{Sect4_Speed_5}
\begin{array}{ll}
e(s)\leq e(t_1)\exp\{C\max\limits_{\mathcal{O}(t_2)}\{\nabla\xi, \nabla v, \nabla\phi\}(s-t_1)\}.
\end{array}
\end{equation}

Thus, if $(\xi,v,\phi)(y,t_1)=0$ in $\{|y-x|\leq k_2(t_2-t_1)\}\cap\Omega$, then $(\xi,v,\phi)(x,t_2)=0$. Equivalently, if $(p-\bar{p},u,S-\bar{S})(y,t_1)=0$ in $\{|y-x|\leq k_2(t_2-t_1)\}\cap\Omega$, then $(p-\bar{p},u,S-\bar{S})(x,t_2)=0$.
Thus, Lemma $\ref{Sect4_Speed_Lemma}$ is proved.
\end{proof}

To verify the weak dissipativity of non-isentropic Euler equations with damping, we proved the following theorem which states that for a class of large data, the singularities form in the interior of ideal gases. In our proof,  $Supp(p_0-\bar{p},u_0,S_0-\bar{S})$ is required to be away from $\partial\Omega$. Thus,
$\partial_t^{\ell}u(x,0)\cdot n|_{\partial\Omega}=0, 0\leq\ell\leq 3$ are satisfied automatically.
\begin{theorem}\label{Sect4_Blowup_Thm}
Assume $0\in\Omega$, $(p_0,u_0,S_0)\in H^3(\Omega)$, $\inf\limits_{x\in\Omega}p_0(x)>0$,
$h=dist\{\partial\Omega, Supp(p_0-\bar{p},u_0,S_0-\bar{S})\}>0$, $(p,u,S)\in C^1(\Omega\times[0,\tau))$ is the classical solution to IBVP $(\ref{Sect1_NonIsentropic_EulerEq})$ where $\tau>0$ is the lifespan of $(p,u,S)$. For any fixed $T$ satisfying $0<T<\min\{\frac{h}{k_2},\frac{\pi}{2}\frac{B_0}{r}\}$, if the condition
$(\ref{Sect2_Blowup_Condition})$ holds, then $\tau\leq T$.
\end{theorem}

\begin{proof}
The quantities $M_{\varrho}(t),B_0,B_1,r$ are defined in $(\ref{Sect2_blowup_Def_Quantities})$.
The set of initial data satisfying the inequality $(\ref{Sect2_Blowup_Condition})$ and the conditions in Theorem $\ref{Sect4_Blowup_Thm}$ is not empty set if $u_0$ is large enough, since the right hand side of $(\ref{Sect2_Blowup_Condition})$ does not contain $u_0$.

It follows from IBVP $(\ref{Sect1_NonIsentropic_EulerEq_Original})$ that for $\forall t\in [0,T)$,
\begin{equation}\label{Sect4_Blowup_Prove1}
\begin{array}{ll}
\frac{\mathrm{d}}{\mathrm{d}t}M_{\varrho}(t) + a M_{\varrho}(t) \\[10pt]

= \int\limits_{\Omega} \varrho_t u\cdot x \,\mathrm{d}x + \int\limits_{\Omega} \varrho u_t \cdot x \,\mathrm{d}x
+ a \int\limits_{\Omega} \varrho u \cdot x \,\mathrm{d}x \\[10pt]

= -\int\limits_{\Omega} \nabla\cdot (\varrho u) (u\cdot x) \,\mathrm{d}x
- \int\limits_{\Omega} \varrho u\cdot\nabla u\cdot x + x\cdot \nabla p \,\mathrm{d}x \\[10pt]

= -\int\limits_{\Omega} \nabla\cdot(\varrho u) (u\cdot x) \,\mathrm{d}x
- \int\limits_{\Omega} \varrho u\cdot\nabla (u\cdot x)
- \varrho u\cdot\nabla x \cdot u + x\cdot\nabla p \,\mathrm{d}x \\[10pt]

= - \int\limits_{\partial\Omega} (u\cdot x)\varrho u\cdot n \,\mathrm{d}S_x
+ \int\limits_{\Omega} \varrho |u|^2 \,\mathrm{d}x
- \int\limits_{\Omega} x \cdot\nabla (p-\bar{p}) \,\mathrm{d}x \\[10pt]

= \int\limits_{\Omega} \varrho |u|^2 \,\mathrm{d}x
- \int\limits_{\partial\Omega} (p-\bar{p}) (x\cdot n) \,\mathrm{d}S_x
+ 3\int\limits_{\Omega} p \,\mathrm{d}x - 3\int\limits_{\Omega} \bar{p} \,\mathrm{d}x.
\end{array}
\end{equation}

By Lemma $\ref{Sect4_Speed_Lemma}$ on the finite propagation speed of classical solutions near the boundary, for any $(x,t)\in\partial\Omega\times[0,T)$, $p-\bar{p}=0$ due to $h>k_2 T$. So in the time interval $[0,T)$, we have
\begin{equation}\label{Sect4_Blowup_Prove2_1}
\begin{array}{ll}
\int\limits_{\partial\Omega} (p-\bar{p}) (x\cdot n) \,\mathrm{d}S_x =0.
\end{array}
\end{equation}

Since $S$ is invariant along the particle paths, $S(x,t)\geq S_{-}$, then we get
\begin{equation}\label{Sect4_Blowup_Prove2_2}
\begin{array}{ll}
\int\limits_{\Omega} p \,\mathrm{d}x \geq A e^{S_{-}}\int\limits_{\Omega} \varrho^{\gamma} \,\mathrm{d}x.
\end{array}
\end{equation}

By H$\ddot{o}$lder's inequality, we have $\int\limits_{\Omega} \varrho \,\mathrm{d}x \leq
\left(\int\limits_{\Omega} \varrho^{\gamma} \,\mathrm{d}x\right)^{\frac{1}{\gamma}}
\left(\int\limits_{\Omega}1 \,\mathrm{d}x\right)^{1-\frac{1}{\gamma}}$, then we get
\begin{equation}\label{Sect4_Blowup_Prove2_3}
\begin{array}{ll}
\int\limits_{\Omega} \varrho^{\gamma} \,\mathrm{d}x \geq \frac{1}{|\Omega|^{\gamma-1}}
\left(\int\limits_{\Omega} \varrho \,\mathrm{d}x\right)^{\gamma}
= \frac{1}{|\Omega|^{\gamma-1}}
\left(\int\limits_{\Omega} \varrho_0 \,\mathrm{d}x\right)^{\gamma}.
\end{array}
\end{equation}

Therefore, it follows from $(\ref{Sect4_Blowup_Prove1})$ that
\begin{equation}\label{Sect4_Blowup_Prove3}
\begin{array}{ll}
\frac{\mathrm{d}}{\mathrm{d}t}M_{\varrho}(t) + a M_{\varrho}(t)

\geq \int\limits_{\Omega} \varrho |u|^2 \,\mathrm{d}x + \frac{3A e^{S_{-}}}{|\Omega|^{\gamma-1}}
\left(\int\limits_{\Omega} \varrho_0 \,\mathrm{d}x\right)^{\gamma} - 3\int\limits_{\Omega} \bar{p} \,\mathrm{d}x
\\[10pt]\hspace{2.9cm}

= \int\limits_{\Omega} \varrho |u|^2 \,\mathrm{d}x + B_1.
\end{array}
\end{equation}

While, by Cauchy-Schwarz inequality, we get
\begin{equation}\label{Sect4_Blowup_Prove4}
\begin{array}{ll}
M_{\varrho}(t)^2 = \left(\int\limits_{\Omega} \varrho u \cdot x \,\mathrm{d}x \right)^2
\leq \int\limits_{\Omega} \varrho |u|^2 \,\mathrm{d}x \int\limits_{\Omega} \varrho |x|^2 \,\mathrm{d}x
\\[10pt]\hspace{1.1cm}

\leq \int\limits_{\Omega} \varrho |u|^2 \,\mathrm{d}x
\left(|Diam(\Omega)|^2\int\limits_{\Omega} \varrho \,\mathrm{d}x \right)

= B_0\int\limits_{\Omega} \varrho |u|^2 \,\mathrm{d}x.
\end{array}
\end{equation}

If $B_1 =0$, it follows from $(\ref{Sect4_Blowup_Prove1})$ that
\begin{equation}\label{Sect4_Blowup_Prove5}
\begin{array}{ll}
\frac{\mathrm{d}}{\mathrm{d}t}M_{\varrho}(t) + a M_{\varrho}(t)
\geq \int\limits_{\Omega} \varrho |u|^2 \,\mathrm{d}x \geq \frac{M_{\varrho}(t)^2}{B_0}, \\[6pt]

\frac{\mathrm{d}}{\mathrm{d}t}[M_{\varrho}(t)\exp{(at)}] \geq \frac{(M_{\varrho}(t)\exp{(at)})^2}{B_0\exp{(at)}}, \\[6pt]

\frac{1}{M_{\varrho}(t)\exp{(at)}} - \frac{1}{M_{\varrho}(0)}
\leq \int\limits_0^t -\frac{1}{B_0\exp{(a s)}} \,\mathrm{d}s = \frac{1}{a B_0}(\exp\{-at\}-1), \\[6pt]

\frac{1}{M_{\varrho}(t)}\leq \exp\{-at\}[\frac{1}{M_{\varrho}(0)}+ \frac{1}{a B_0}(\exp\{-at\}-1)].
\end{array}
\end{equation}

Since $M_{\varrho}(0)>\frac{aB_0}{1-\exp\{-aT\}}>0$, when $t$ is small, $M_{\varrho}(t)>0$ by its continuity.
\begin{equation}\label{Sect4_Blowup_Prove6}
\begin{array}{ll}
M_{\varrho}(t)\geq \cfrac{\exp\{at\}}{\frac{1}{M_{\varrho}(0)}+ \frac{1}{a B_0}(\exp\{-at\}-1)} >0.
\end{array}
\end{equation}

When $t\rto T-$, the denominator $\frac{1}{M_{\varrho}(0)}+ \frac{1}{a B_0}(\exp\{-at\}-1) \rto 0+$, $M_{\varrho}(t)\rto +\infty$. Thus $\tau\leq T$.

\vspace{0.4cm}
If $B_1 \neq0$, it follows from $(\ref{Sect4_Blowup_Prove1})$ that
\begin{equation}\label{Sect4_Blowup_Prove7}
\begin{array}{ll}
\frac{\mathrm{d}}{\mathrm{d}t}M_{\varrho}(t) \geq \frac{M_{\varrho}(t)^2}{B_0}- a M_{\varrho}(t) + B_1
= \frac{1}{B_0} (M_{\varrho}(t)-\frac{a B_0}{2})^2 + B_1-\frac{a^2 B_0}{4}.
\end{array}
\end{equation}

Denote
\begin{equation}\label{Sect4_Blowup_Prove8}
\begin{array}{ll}
N_{\varrho}(t) = M_{\varrho}(t)-\frac{a B_0}{2},\ r^2 = \left|B_1 -\frac{a^2 B_0}{4}\right|,\ r>0.
\end{array}
\end{equation}

\begin{equation}\label{Sect4_Blowup_Prove9}
\frac{\mathrm{d}}{\mathrm{d}t}N_{\varrho}(t) \geq \left\{\begin{array}{ll}
\frac{1}{B_0}(N_{\varrho}(t)^2 + r^2), \ i\! f\ B_1 > \frac{a^2 B_0}{4}, \\[6pt]
\frac{1}{B_0}(N_{\varrho}(t)^2 - r^2), \ i\! f\ B_1 < \frac{a^2 B_0}{4}.
\end{array}\right.
\end{equation}

Integrating from $0$ to $\tau$, we obtain,
\begin{equation}\label{Sect4_Blowup_Prove10}
\frac{t}{B_0} \leq \left\{\begin{array}{ll}
\frac{1}{r}\arctan(\frac{N_{\varrho}(t)}{r}) -\frac{1}{r}\arctan(\frac{N_{\varrho}(0)}{r})
,\ i\! f\ B_1 > \frac{a^2 B_0}{4}, \\[10pt]
\frac{1}{2r}\ln|\frac{N_{\varrho}(t)-r}{N_{\varrho}(t)+r}| -\frac{1}{2r}\ln|\frac{N_{\varrho}(0)-r}{N_{\varrho}(0)+r}|
,\ i\! f\ B_1 < \frac{a^2 B_0}{4}.
\end{array}\right.
\end{equation}

As to $(\ref{Sect4_Blowup_Prove10})_1$,
\begin{equation}\label{Sect4_Blowup_Prove11}
\begin{array}{ll}
N_{\varrho}(t) \geq \cfrac{r\tan(\frac{rt}{B_0})+N_{\varrho}(0)}{1-\frac{N_{\varrho}(0)}{r}\tan(\frac{rt}{B_0})}\ , \\[15pt]
M_{\varrho}(t) \geq \frac{a B_0}{2} +
\cfrac{r\tan(\frac{rt}{B_0})+M_{\varrho}(0)-\frac{a B_0}{2}}{1-\frac{M_{\varrho}(0)-\frac{a B_0}{2}}{r}\tan(\frac{rt}{B_0})}\ .
\end{array}
\end{equation}

Since $M_{\varrho}(0)>\frac{aB_0}{2} + r\cot(\frac{rT}{B_0})$, when $t$ is small, $M_{\varrho}(t)>0$ by its continuity.
When $t\rto T-$, the denominator $1-\frac{M_{\varrho}(0)-\frac{a B_0}{2}}{r}\tan(\frac{rt}{B_0})\rto 0+$, $M_{\varrho}(t)\rto +\infty$. Since $T<\frac{\pi}{2}\frac{B_0}{r}$ and $[0,\tau)$ as the lifespan of the solution is simply connected, $\tau\leq T$.

As to $(\ref{Sect4_Blowup_Prove10})_2$, since $M_{\varrho}(0)>\frac{aB_0}{2}+r$, $N_{\varrho}(0)>r$, $N_{\varrho}(t)$ is increasing due to $\frac{\mathrm{d}}{\mathrm{d}t}N_{\varrho}(t) \geq \frac{1}{B_0}(N_{\varrho}(t)^2 - r^2)>0$. then $\frac{N_{\varrho}(t)-r}{N_{\varrho}(t)+r}>0$.

\begin{equation}\label{Sect4_Blowup_Prove12}
\begin{array}{ll}
N_{\varrho}(t) \geq -r+
\cfrac{2r}{1-\exp\{\frac{2rt}{B_0}+\ln|\frac{N_{\varrho}(0)-r}{N_{\varrho}(0)+r}|\}}\ , \\[15pt]
M_{\varrho}(t) \geq \frac{a B_0}{2} -r+
\cfrac{2r}{1-\exp\{\frac{2rt}{B_0}+\ln|\frac{M_{\varrho}(0)-\frac{a B_0}{2}-r}{M_{\varrho}(0)-\frac{a B_0}{2}+r}|\}}
\ .
\end{array}
\end{equation}

Since $M_{\varrho}(0)>\max\{\frac{aB_0}{2} -r + \cfrac{2r}{1-\exp\{-\frac{2rT}{B_0}\}},\ \frac{aB_0}{2} +r\}$, when $t$ is small, $M_{\varrho}(t)>0$ by its continuity.
When $t\rto T-$, the denominator $1-\exp\{\frac{2rt}{B_0}+\ln|\frac{M_{\varrho}(0)-\frac{a B_0}{2}-r}{M_{\varrho}(0)-\frac{a B_0}{2}+r}|\}\rto 0+$, $M_{\varrho}(t)=M_{\varrho}(\tau)\rto +\infty$. Thus $\tau\leq T$.

Thus, Theorem $\ref{Sect4_Blowup_Thm}$ on the singularity formation of classical solutions for a class of large data is proved.
\end{proof}

\begin{remark}\label{Sect4_Blowup_Remark}
We cannot remove the condition $dist(\partial\Omega, Supp(p_0-\bar{p},u_0,S_0-\bar{S}))>0$, otherwise the term $-\int\limits_{\partial\Omega} (p-\bar{p})(x\cdot n) \,\mathrm{d}S_x$ in $(\ref{Sect4_Blowup_Prove1})$ cannot be bounded. Denote the area of $\partial\Omega$ by $|\partial\Omega|$, we have
\begin{equation}\label{Sect4_Blowup_Remark_Eq}
\begin{array}{ll}
-\int\limits_{\partial\Omega} (p-\bar{p})(x\cdot n) \,\mathrm{d}S_x \\[10pt]
\geq -\bar{p}|Diam(\Omega)||\partial\Omega| - A|Diam(\Omega)|e^{S_{+}}\int\limits_{\partial\Omega} \varrho^{\gamma}\,\mathrm{d}S_x \\[10pt]
\geq -\bar{p}|Diam(\Omega)||\partial\Omega| - CA|Diam(\Omega)|e^{S_{+}} \|\varrho\|_{L^{\gamma}(\Omega)}^{\gamma-1}\|\nabla\varrho\|_{L^{\gamma}(\Omega)} \\[10pt]
\geq -\bar{p}|Diam(\Omega)||\partial\Omega| - \cfrac{CA|Diam(\Omega)|e^{S_{+}}}
{|\Omega|^{\frac{(\gamma-1)^2}{\gamma}}}\left(\int\limits_{\Omega} \varrho_0 \,\mathrm{d}x\right)^{\gamma-1}
\|\nabla\varrho\|_{L^{\gamma}(\Omega)} .
\end{array}
\end{equation}
where $\|\varrho\|_{L^{\gamma}(\partial\Omega)}^{\gamma}\leq C \|\varrho\|_{L^{\gamma}(\Omega)}^{\gamma-1} \|\nabla\varrho\|_{L^{\gamma}(\Omega)}$ (see \cite{Adams_Fournier_2009}) is applied.
While $\|\nabla\varrho\|_{L^{\gamma}(\Omega)}$ is not conserved and may not be bounded by the initial data.
\end{remark}

\section{Global A Priori Estimates for Diffusion Equations}
In this section, we derive global a priori estimates for the nonlinear diffusion equations $(\ref{Sect2_Final_Diffusion})$. For simplicity, we omit the symbol $\ \hat{}\ $ over the variables and constants in this section if there is no ambiguity.

The following lemma is also an application of Lemma $\ref{Sect3_DivCurl_Decomposition}$, which states that the spatial derivatives are bounded by the temporal derivatives and the vorticity, thus the total energy $\mathcal{F}[\xi,v](t)$ can be bounded by $E[\xi,v](t)$ and $\mathcal{E}_1[\omega](t)$.
\begin{lemma}\label{Sect5_Total_Energy_Lemma}
For any given $T\in (0,+\infty]$, there exists $\ve_0>0$ which is independent of $(\xi_0,v_0,\phi_0)$, such that if
\begin{equation*}
\sup\limits_{0\leq t\leq T} \mathcal{F}[\xi,v,\phi](t) \leq\ve,
\end{equation*}
where $\ve\ll \min\{1,\ve_0\}$, then for $\forall t\in [0,T]$,
\begin{equation}\label{Sect5_Total_Energy_Eq}
\begin{array}{ll}
\mathcal{F}[\xi,v](t) \leq c_{10}(E[\xi,v](t) + \mathcal{E}_1[\omega](t)),
\end{array}
\end{equation}
for some $c_{10}>0$.
\end{lemma}

\begin{proof}
Since the equation $(\ref{Sect2_Final_Diffusion})_1$ is the same with the equation $(\ref{Sect2_Final_Eq})_1$,
similar to the estimates in Lemma $\ref{Sect3_Total_Energy_Lemma}$, we have
\begin{equation}\label{Sect5_Total_Energy_Prove0}
\begin{array}{ll}
\|\nabla v\|_{L^2(\Omega)}^2
\lem \|\xi_t\|_{L^(\Omega)}^2 +\|\omega\|_{L^2(\Omega)}^2 +\|v\|_{L^2(\Omega)}^2 + \sqrt{\ve}\mathcal{F}[\xi,v](t), \\[6pt]
\|\nabla v_t\|_{L^2(\Omega)}^2
\lem \|\xi_{tt}\|_{L^2(\Omega)}^2 + \|\omega_t\|_{L^2(\Omega)}^2 + \|v_t\|_{L^2(\Omega)}^2
+ \sqrt{\ve}\mathcal{F}[\xi,v](t),\\[6pt]

\|\nabla v_{tt}\|_{L^2(\Omega)}^2
\lem \|\xi_{ttt}\|_{L^2(\Omega)}^2 +\|\omega_{tt}\|_{L^2(\Omega)}^2 +\|v_{tt}\|_{L^2(\Omega)}^2
+ \sqrt{\ve}\mathcal{F}[\xi,v](t).
\end{array}
\end{equation}

Moreover, when $|\alpha|=1$,
\begin{equation}\label{Sect5_Total_Energy_Prove1}
\begin{array}{ll}
\|\mathcal{D}^{\alpha}\nabla\cdot v\|_{L^2(\Omega)}^2
\lem \|\mathcal{D}^{\alpha}\xi_t\|_{L^2(\Omega)}^2 + \sqrt{\ve}\mathcal{F}[\xi,v](t)
\lem \|v_t\|_{L^2(\Omega)}^2 + \sqrt{\ve}\mathcal{F}[\xi,v](t), \\[6pt]

\|\mathcal{D}^{\alpha}\nabla v\|_{L^2(\Omega)}^2
\lem \|v_t\|_{L^2(\Omega)}^2 + \|\omega\|_{H^1(\Omega)}^2
+ \|v\|_{L^2(\Omega)}^2 +\|\xi_t\|_{L^2(\Omega)}^2 + \sqrt{\ve}\mathcal{F}[\xi,v](t).
\end{array}
\end{equation}

\begin{equation}\label{Sect5_Total_Energy_Prove2}
\begin{array}{ll}
\|\mathcal{D}^{\alpha}\nabla\cdot v_t\|_{L^2(\Omega)}^2 \lem \|\mathcal{D}^{\alpha}\xi_{tt}\|_{L^2(\Omega)}^2
+\sqrt{\ve}\mathcal{F}[\xi,v](t)
\lem \|v_{tt}\|_{L^2(\Omega)}^2 +\sqrt{\ve}\mathcal{F}[\xi,v](t), \\[6pt]

\|\mathcal{D}^{\alpha}\nabla v_t\|_{L^2(\Omega)}^2
\lem \|\xi_{tt}\|_{L^2(\Omega)}^2 + \|v_{tt}\|_{L^2(\Omega)}^2
+ \|\omega_t\|_{H^1(\Omega)}^2 + \|v_t\|_{L^2(\Omega)}^2 \\[6pt]\hspace{2.7cm}
+ \sqrt{\ve}\mathcal{F}[\xi,v](t).
\end{array}
\end{equation}

When $|\alpha|=2$,
\begin{equation}\label{Sect5_Total_Energy_Prove3}
\begin{array}{ll}
\|\mathcal{D}^{\alpha}\nabla\cdot v\|_{L^2(\Omega)}^2 \lem \|\mathcal{D}^{\alpha}\xi_t\|_{L^2(\Omega)}^2+ \sqrt{\ve}\mathcal{F}[\xi,v](t) \\[6pt]\hspace{2.4cm}
\lem \|\xi_{tt}\|_{H^1(\Omega)}^2 +\|\omega_t\|_{L^2(\Omega)}^2 +\|v_t\|_{L^2\Omega)}^2
+ \sqrt{\ve}\mathcal{F}[\xi,v](t), \\[6pt]

\|\mathcal{D}^{\alpha}\nabla v\|_{L^2(\Omega)}^2
\lem \|\xi_{tt}\|_{H^1(\Omega)}^2 +\|\omega_t\|_{H^1(\Omega)}^2 +\|\omega\|_{H^2(\Omega)}^2 +\|v_t\|_{L^2\Omega)}^2
\\[6pt]\hspace{2.6cm}
+\|v\|_{L^2\Omega)}^2 +\|\xi_t\|_{L^2\Omega)}^2 + \sqrt{\ve}\mathcal{F}[\xi,v](t).
\end{array}
\end{equation}

By $(\ref{Sect2_Final_Diffusion_Vt})_2$, we get
\begin{equation}\label{Sect5_Total_Energy_Prove4}
\begin{array}{ll}
\nabla\xi = - ak_1\varrho v, \\[6pt]
\|\nabla\xi\|_{L^2(\Omega)}^2 \lem \|v\|_{L^2(\Omega)}^2.
\end{array}
\end{equation}

Apply $\partial_t$ to $(\ref{Sect5_Total_Energy_Prove4})_1$, we get
\begin{equation}\label{Sect5_Total_Energy_Prove5}
\begin{array}{ll}
\nabla\xi_t = - ak_1\xi_t v -ak_1\varrho v_t,\\[6pt]

\|\nabla\xi_t\|_{L^2(\Omega)}^2 \lem \|v_t\|_{L^2(\Omega)}^2 +\sqrt{\ve}\mathcal{F}[\xi,v](t).
\end{array}
\end{equation}

Apply $\partial_{tt}$ to $(\ref{Sect5_Total_Energy_Prove4})_1$, we get
\begin{equation}\label{Sect5_Total_Energy_Prove6}
\begin{array}{ll}
\nabla\xi_{tt} = - ak_1\xi_{tt} v - 2ak_1\xi_t v_t -ak_1\varrho v_{tt}, \\[6pt]

\|\nabla\xi_{tt}\|_{L^2(\Omega)}^2 \lem \|v_{tt}\|_{L^2(\Omega)}^2 +\sqrt{\ve}\mathcal{F}[\xi,v](t).
\end{array}
\end{equation}

Apply $\partial_{ttt}$ to $(\ref{Sect5_Total_Energy_Prove4})_1$, we get
\begin{equation}\label{Sect5_Total_Energy_Prove7}
\begin{array}{ll}
\nabla\xi_{ttt} = - ak_1\xi_{ttt} v - 3ak_1\xi_{tt} v_t - 3ak_1\xi_t v_{tt} -ak_1\varrho v_{ttt}, \\[6pt]

\|\nabla\xi_{ttt}\|_{L^2(\Omega)}^2 \lem \|v_{ttt}\|_{L^2(\Omega)}^2 +\sqrt{\ve}\mathcal{F}[\xi,v](t).
\end{array}
\end{equation}

Apply $\mathcal{D}^{\alpha}$ to $(\ref{Sect5_Total_Energy_Prove4})_1$, where $|\alpha|=1$, we get
\begin{equation}\label{Sect5_Total_Energy_Prove8}
\begin{array}{ll}
\mathcal{D}^{\alpha}\nabla\xi = - ak_1(\mathcal{D}^{\alpha}\xi) v - ak_1\varrho \mathcal{D}^{\alpha}v, \\[6pt]

\|\mathcal{D}^{\alpha}\nabla\xi\|_{L^2(\Omega)}^2 \lem \|\mathcal{D}^{\alpha}v\|_{L^2(\Omega)}^2 +\sqrt{\ve}\mathcal{F}[\xi,v](t) \\[6pt]\hspace{2.1cm}

\lem \|\xi_t\|_{L^2(\Omega)}^2 +\|\omega\|_{L^2(\Omega)}^2 +\|v\|_{L^2(\Omega)}^2 +\sqrt{\ve}\mathcal{F}[\xi,v](t).
\end{array}
\end{equation}

Apply $\mathcal{D}^{\alpha}$ to $(\ref{Sect5_Total_Energy_Prove5})_1$, where $|\alpha|=1$, we get
\begin{equation}\label{Sect5_Total_Energy_Prove9}
\begin{array}{ll}
\mathcal{D}^{\alpha}\nabla\xi_t = - ak_1\mathcal{D}^{\alpha}(\xi_t v)
-ak_1(\mathcal{D}^{\alpha}\xi) v_t -ak_1\varrho\mathcal{D}^{\alpha} v_t,\\[6pt]

\|\mathcal{D}^{\alpha}\nabla\xi_t\|_{L^2(\Omega)}^2 \lem
\|\mathcal{D}^{\alpha} v_t\|_{L^2(\Omega)}^2 +\sqrt{\ve}\mathcal{F}[\xi,v](t) \\[6pt]\hspace{2.3cm}

\lem \|\xi_{tt}\|_{L^2(\Omega)}^2 +\|\omega_t\|_{L^2(\Omega)}^2 +\|v_t\|_{L^2(\Omega)}^2
+\sqrt{\ve}\mathcal{F}[\xi,v](t).
\end{array}
\end{equation}

Apply $\mathcal{D}^{\alpha}$ to $(\ref{Sect5_Total_Energy_Prove6})_1$, where $|\alpha|=1$, we get
\begin{equation}\label{Sect5_Total_Energy_Prove10}
\begin{array}{ll}
\mathcal{D}^{\alpha}\nabla\xi_{tt} = - ak_1\mathcal{D}^{\alpha}(\xi_{tt}v) - 2ak_1\mathcal{D}^{\alpha}(\xi_t v_t)
- ak_1 (\mathcal{D}^{\alpha}\varrho) v_{tt} - ak_1\varrho \mathcal{D}^{\alpha}v_{tt}, \\[6pt]

\|\mathcal{D}^{\alpha}\nabla\xi_{tt}\|_{L^2(\Omega)}^2 \lem \|\mathcal{D}^{\alpha}v_{tt}\|_{L^2(\Omega)}^2 +\sqrt{\ve}\mathcal{F}[\xi,v](t) \\[6pt]\hspace{2.4cm}

\lem \|\xi_{ttt}\|_{L^2(\Omega)}^2 +\|\omega_{tt}\|_{L^2(\Omega)}^2 +\|v_{tt}\|_{L^2(\Omega)}^2 +\sqrt{\ve}\mathcal{F}[\xi,v](t).
\end{array}
\end{equation}

Apply $\mathcal{D}^{\alpha}$ to $(\ref{Sect5_Total_Energy_Prove4})_1$, where $|\alpha|=2$, $\alpha=\alpha_1+\alpha_2$, we get
\begin{equation}\label{Sect5_Total_Energy_Prove11}
\begin{array}{ll}
\mathcal{D}^{\alpha}\nabla\xi = - ak_1\varrho\mathcal{D}^{\alpha}v
- ak_1\sum\limits_{\alpha_1>0}(\mathcal{D}^{\alpha_1}\xi)\mathcal{D}^{\alpha_2}v, \\[6pt]

\|\mathcal{D}^{\alpha}\nabla\xi\|_{L^2(\Omega)}^2 \lem \|\mathcal{D}^{\alpha}v\|_{L^2(\Omega)}^2 +\sqrt{\ve}\mathcal{F}[\xi,v](t) \\[6pt]\hspace{2.2cm}

\lem \|v_t\|_{L^2(\Omega)}^2 + \|\omega\|_{H^1(\Omega)}^2 +\|\xi_t\|_{L^2(\Omega)}^2
+\|v\|_{L^2(\Omega)}^2+ \sqrt{\ve}\mathcal{F}[\xi,v](t).
\end{array}
\end{equation}

Apply $\mathcal{D}^{\alpha}$ to $(\ref{Sect5_Total_Energy_Prove5})_1$, where $|\alpha|=2$, $\alpha=\alpha_1+\alpha_2$, we get
\begin{equation}\label{Sect5_Total_Energy_Prove12}
\begin{array}{ll}
\mathcal{D}^{\alpha}\nabla\xi_t =  - ak_1\mathcal{D}^{\alpha}(\xi_t v)
- ak_1\sum\limits_{\alpha_1>0}(\mathcal{D}^{\alpha_1}\varrho\mathcal{D}^{\alpha_2}v_t)
- ak_1\varrho\mathcal{D}^{\alpha}v_t, \\[6pt]

\|\mathcal{D}^{\alpha}\nabla\xi_t\|_{L^2(\Omega)}^2 \lem \|\mathcal{D}^{\alpha}v_t\|_{L^2(\Omega)}^2 +\sqrt{\ve}\mathcal{F}[\xi,v](t) \\[6pt]\hspace{2.2cm}

\lem \|\xi_{tt}\|_{L^2(\Omega)}^2 + \|v_{tt}\|_{L^2(\Omega)}^2 + \|\omega_t\|_{H^1(\Omega)}^2
+ \|v_{t}\|_{L^2(\Omega)}^2 \\[6pt]\hspace{2.6cm}
+ \sqrt{\ve}\mathcal{F}[\xi,v](t).
\end{array}
\end{equation}

Apply $\mathcal{D}^{\alpha}$ to $(\ref{Sect5_Total_Energy_Prove4})_1$, where $|\alpha|=3$, $\alpha=\alpha_1+\alpha_2$, we get
\begin{equation}\label{Sect5_Total_Energy_Prove13}
\begin{array}{ll}
\mathcal{D}^{\alpha}\nabla\xi = - ak_1\varrho\mathcal{D}^{\alpha}v
- ak_1\sum\limits_{\alpha_1>0}(\mathcal{D}^{\alpha_1}\xi)\mathcal{D}^{\alpha_2}v, \\[6pt]

\|\mathcal{D}^{\alpha}\nabla\xi\|_{L^2(\Omega)}^2 \lem \|\mathcal{D}^{\alpha}v\|_{L^2(\Omega)}^2 +\sqrt{\ve}\mathcal{F}[\xi,v](t) \\[6pt]\hspace{2.2cm}

\lem \|v_{tt}\|_{L^2(\Omega)}^2 +\|\xi_{tt}\|_{H^1(\Omega)}^2
+\|\omega_t\|_{H^1(\Omega)}^2 +\|\omega\|_{H^2(\Omega)}^2 \\[6pt]\hspace{2.6cm}
+\|v_t\|_{L^2\Omega)}^2 + \|v\|_{L^2(\Omega)}^2 +\|\xi_t\|_{L^2(\Omega)}^2
+ \sqrt{\ve}\mathcal{F}[\xi,v](t).
\end{array}
\end{equation}

By $(\ref{Sect2_Velocity_Solve 3})$, we have
\begin{equation}\label{Sect5_Velocity_Prove14}
\begin{array}{ll}
\nabla(\nabla\cdot v) = \frac{a\varrho}{\gamma p}v_t + k_1(\gamma-1)\frac{a\varrho}{\gamma p}v(\nabla\cdot v)
+ \frac{k_1 a\varrho}{\gamma p}v\cdot\nabla v + \frac{k_1 a\varrho}{2\gamma p}\nabla(|v|^2),\\[6pt]

\|\nabla\cdot v\|_{H^3{(\Omega)}}^2 \lem \|\nabla(\nabla\cdot v)\|_{H^2{(\Omega)}}^2
\lem \|\frac{\varrho}{p}v_t\|_{H^2{(\Omega)}}^2 + \|\frac{\varrho}{p}v(\nabla\cdot v)\|_{H^2{(\Omega)}}^2
\\[6pt]\hspace{2.4cm}
+ \|\frac{\varrho}{p}v\cdot\nabla v\|_{H^2{(\Omega)}}^2 + \|\frac{\varrho}{p}\nabla(|v|^2)\|_{H^2{(\Omega)}}^2 \\[6pt]\hspace{2cm}
\lem \|v_t\|_{H^2{(\Omega)}}^2 + \sqrt{\ve}\mathcal{F}[\xi,v](t) \\[6pt]\hspace{2cm}
\lem \|\xi_{tt}\|_{L^2(\Omega)}^2 + \|v_{tt}\|_{L^2(\Omega)}^2
+ \|\omega_t\|_{H^1(\Omega)}^2 + \|v_t\|_{L^2(\Omega)}^2
+ \sqrt{\ve}\mathcal{F}[\xi,v](t).
\end{array}
\end{equation}

Apply $\partial_t$ to $(\ref{Sect5_Velocity_Prove14})_1$, we get
\begin{equation}\label{Sect5_Velocity_Prove15}
\begin{array}{ll}
\nabla(\nabla\cdot v_t) = \frac{a\varrho}{\gamma p}v_{tt} + \frac{a}{\gamma}(\frac{\varrho}{p})_t v_t + k_1(\gamma-1)\frac{a}{\gamma}\partial_t[\frac{\varrho}{p}v(\nabla\cdot v)]
+ \frac{k_1 a}{\gamma}\partial_t[\frac{\varrho}{p}v\cdot\nabla v] \\[6pt]\hspace{1.9cm}
+ \frac{k_1 a}{2\gamma}\partial_t[\frac{\varrho}{p}\nabla(|v|^2)],\\[6pt]

\|\nabla\cdot v_t\|_{H^2{(\Omega)}}^2 \lem \|\nabla(\nabla\cdot v_t)\|_{H^1{(\Omega)}}^2
\lem \|v_{tt}\|_{H^1{(\Omega)}}^2 + \sqrt{\ve}\mathcal{F}[\xi,v](t) \\[6pt]\hspace{2.1cm}
\lem \|\xi_{ttt}\|_{L^2(\Omega)}^2 +\|\omega_{tt}\|_{L^2(\Omega)}^2 +\|v_{tt}\|_{L^2(\Omega)}^2
+ \sqrt{\ve}\mathcal{F}[\xi,v](t).
\end{array}
\end{equation}

Apply $\partial_{tt}$ to $(\ref{Sect5_Velocity_Prove14})_1$, we get
\begin{equation}\label{Sect5_Velocity_Prove16}
\begin{array}{ll}
\nabla(\nabla\cdot v_{tt}) = \frac{a\varrho}{\gamma p}v_{ttt} + \frac{2a}{\gamma}(\frac{\varrho}{p})_t v_{tt}
+ \frac{a}{\gamma}(\frac{\varrho}{p})_{tt} v_t
+ k_1(\gamma-1)\frac{a}{\gamma}\partial_{tt}[\frac{\varrho}{p}v(\nabla\cdot v)] \\[6pt]\hspace{2cm}
+ \frac{k_1 a}{\gamma}\partial_{tt}[\frac{\varrho}{p}v\cdot\nabla v]
+ \frac{k_1 a}{2\gamma}\partial_{tt}[\frac{\varrho}{p}\nabla(|v|^2)],\\[6pt]

\|\nabla\cdot v_{tt}\|_{H^1{(\Omega)}}^2 \lem \|\nabla(\nabla\cdot v_{tt})\|_{L^2{(\Omega)}}^2
\lem \|v_{ttt}\|_{L^2{(\Omega)}}^2 + \sqrt{\ve}\mathcal{F}[\xi,v](t).
\end{array}
\end{equation}

Thus, $\mathcal{F}[\xi,v](t) \leq C_8E[\xi,v](t) + C_8\mathcal{E}_1[\omega](t)+ C_8\sqrt{\ve}\mathcal{F}[\xi,v](t)$,
where $C_8>0$.

Let $\ve_0=\frac{1}{4C_8^2}$, when $\ve\ll\min\{1,\ve_0\}$, we have
\begin{equation}\label{Sect5_Total_Energy_Prove18}
\begin{array}{ll}
\mathcal{F}[\xi,v](t)\leq 2C_8 \{E[\xi,v](t) + \mathcal{E}_1[\omega](t)\}.
\end{array}
\end{equation}
Let $c_{10}=2C_8$. Thus, Lemma $\ref{Sect5_Total_Energy_Lemma}$ is proved.
\end{proof}

Next, in order to prove the exponential decay of $\mathcal{F}[\xi,v](t)$ and $\mathcal{E}_1[\omega](t)$, we need to prove a priori estimates for $\mathcal{E}_1[\omega](t)$, $E[\xi](t)$,
$\int\limits_{\Omega}\sum\limits_{\ell=1}^{3}\partial_t^{\ell}\xi \partial_t^{\ell-1}\xi \,\mathrm{d}x$
and $E[v](t)$ separatively.

The following lemma concerns a priori estimate for $\mathcal{E}_1[\omega](t)$.
\begin{lemma}\label{Sect5_Vorticity_Lemma}
For any given $T\in (0,+\infty]$, if
\begin{equation*}
\sup\limits_{0\leq t\leq T} \mathcal{F}[\xi,v,\phi](t) \leq\ve,
\end{equation*}
where $0<\ve\ll 1$, then for $\forall t\in [0,T]$,
\begin{equation}\label{Sect5_Vorticity_toProve}
\begin{array}{ll}
\frac{\mathrm{d}}{\mathrm{d}t}\int\limits_{\Omega}\mathcal{E}_1[\omega](t)\,\mathrm{d}x + 2a\mathcal{E}_1[\omega](t)
\leq C\sqrt{\ve}\mathcal{F}[\xi,v](t).
\end{array}
\end{equation}
\end{lemma}

\begin{proof}
By $\nabla\times (\ref{Sect2_Final_Diffusion_Vt})_2$, we get
\begin{equation}\label{Sect5_Vorticity_1}
\begin{array}{ll}
\partial_t \omega + a\omega = k_1(1-\gamma)[\omega(\nabla\cdot v) - v\times\nabla(\nabla\cdot v)]
-k_1 v\cdot\nabla\omega -k_1 \omega(\nabla\cdot v) \\[10pt]\hspace{1.8cm}
+ k_1 \omega\cdot\nabla v +\frac{\gamma}{a}\nabla(\frac{p}{\varrho})\times\nabla(\nabla\cdot v)
+ \frac{1}{k_1}(\frac{\partial_i\varrho}{\varrho^2}\partial_j\xi
-\frac{\partial_j\varrho}{\varrho^2}\partial_i\xi)_{k}.
\end{array}
\end{equation}

Apply $\partial_t^{\ell}\mathcal{D}^{\alpha}$ to $(\ref{Sect5_Vorticity_1})$, where $0\leq \ell+|\alpha|\leq 2$, we get
\begin{equation}\label{Sect5_Vorticity_2}
\begin{array}{ll}
\partial_t (\partial_t^{\ell}\mathcal{D}^{\alpha}\omega) + a\partial_t^{\ell}\mathcal{D}^{\alpha}\omega \\[10pt]
= k_1(1-\gamma)\partial_t^{\ell}\mathcal{D}^{\alpha}[\omega(\nabla\cdot v) - v\times\nabla(\nabla\cdot v)]
-k_1 \partial_t^{\ell}\mathcal{D}^{\alpha}[v\cdot\nabla\omega + \omega(\nabla\cdot v) \\[10pt]\quad
- \omega\cdot\nabla v]
+\frac{\gamma}{a}\partial_t^{\ell}\mathcal{D}^{\alpha}[\nabla(\frac{p}{\varrho})\times\nabla(\nabla\cdot v)]
+ \frac{1}{k_1} \partial_t^{\ell}\mathcal{D}^{\alpha}
(\frac{\partial_i\varrho}{\varrho^2}\partial_j\xi-\frac{\partial_j\varrho}{\varrho^2}\partial_i\xi)_{k}.
\end{array}
\end{equation}

Let $\partial_t^{\ell}\mathcal{D}^{\alpha}\omega\cdot (\ref{Sect5_Vorticity_2})$ and integrate in $\Omega$, we get
\begin{equation}\label{Sect5_Vorticity_3}
\begin{array}{ll}
\frac{\mathrm{d}}{\mathrm{d}t}\int\limits_{\Omega}|\partial_t^{\ell}\mathcal{D}^{\alpha}\omega|^2 \,\mathrm{d}x
+ 2a\int\limits_{\Omega}|\partial_t^{\ell}\mathcal{D}^{\alpha}\omega|^2 \,\mathrm{d}x = I_4 + I_5,
\end{array}
\end{equation}
where
\begin{equation}\label{Sect5_Vorticity_4}
\begin{array}{ll}
I_4 := 2\int\limits_{\Omega}k_1(1-\gamma)\partial_t^{\ell}\mathcal{D}^{\alpha}\omega\cdot
\partial_t^{\ell}\mathcal{D}^{\alpha}[\omega(\nabla\cdot v) - v\times\nabla(\nabla\cdot v)]
\\[8pt]\qquad
-k_1 \partial_t^{\ell}\mathcal{D}^{\alpha}\omega\cdot
\partial_t^{\ell}\mathcal{D}^{\alpha}[v\cdot\nabla\omega + \omega(\nabla\cdot v) - \omega\cdot\nabla v]
\\[8pt]\qquad
+\frac{\gamma}{a}\partial_t^{\ell}\mathcal{D}^{\alpha}\omega\cdot
\partial_t^{\ell}\mathcal{D}^{\alpha}[\nabla(\frac{p}{\varrho})\times\nabla(\nabla\cdot v)]\,\mathrm{d}x,

\\[8pt]
I_5:= \frac{2}{k_1} \int\limits_{\Omega}\sum\limits_{k=1}^{3}\partial_t^{\ell}\mathcal{D}^{\alpha}
(\frac{\partial_i\varrho}{\varrho^2}\partial_j\xi-\frac{\partial_j\varrho}{\varrho^2}\partial_i\xi)_{k}
\cdot\partial_t^{\ell}\mathcal{D}^{\alpha}\omega_k\,\mathrm{d}x.
\end{array}
\end{equation}

When $\ell+|\alpha|<2$, it is easy to check that $I_4+I_5\lem \sqrt{\ve}\mathcal{F}[\xi,v](t)$, since they are lower order terms.

Since $I_5$ has the same form with $I_2$ in the proof of Lemma $\ref{Sect3_Vorticity_Lemma}$, repeat those estimates $(\ref{Sect3_Vorticity_8}),(\ref{Sect3_Vorticity_9}),(\ref{Sect3_Vorticity_10})$, we have
\begin{equation}\label{Sect5_Vorticity_8}
\begin{array}{ll}
I_5\lem \sqrt{\ve}\mathcal{E}[\xi,v](t)\lem \sqrt{\ve}\mathcal{F}[\xi,v](t).
\end{array}
\end{equation}

Thus, we just need to estimate $I_4$, when $\ell=0,|\alpha|=2$,
\begin{equation}\label{Sect5_Vorticity_5}
\begin{array}{ll}
I_4\lem \|\mathcal{D}^{\alpha}\omega\|_{L^2(\Omega)}
(|\nabla\cdot v|_{\infty}\|\mathcal{D}^{\alpha}\omega\|_{L^2(\Omega)}
+ \|\mathcal{D}^{\alpha_1}\omega\|_{L^4(\Omega)}\|\mathcal{D}^{\alpha_2}\nabla\cdot v\|_{L^4(\Omega)}
\\[6pt]\qquad
+ |\omega|_{\infty}\|\mathcal{D}^{\alpha}\nabla\cdot v\|_{L^2(\Omega)}
+ |\nabla(\nabla\cdot v)|_{\infty}\|\mathcal{D}^{\alpha} v\|_{L^2(\Omega)} \\[6pt]\qquad
+ |\mathcal{D}^{\alpha_1}v|_{\infty}\|\mathcal{D}^{\alpha_2}\nabla(\nabla\cdot v)\|_{L^2(\Omega)}
+ |v|_{\infty}\|\mathcal{D}^{\alpha}\nabla(\nabla\cdot v)\|_{L^2(\Omega)} \\[6pt]\qquad

+ \|\mathcal{D}^{\alpha}v\|_{L^4(\Omega)}\|\nabla\omega\|_{L^4(\Omega)}
+|\mathcal{D}^{\alpha_1}v|_{\infty}\|\mathcal{D}^{\alpha_2}\nabla\omega\|_{L^2(\Omega)} \\[6pt]\qquad
+|\nabla\cdot v|_{\infty}\|\mathcal{D}^{\alpha}\omega\|_{L^2(\Omega)}
+\|\mathcal{D}^{\alpha_1}\omega\|_{L^4(\Omega)}\|\mathcal{D}^{\alpha_2}\nabla\cdot v\|_{L^4(\Omega)}
\\[6pt]\qquad

+ |\omega|_{\infty}\|\mathcal{D}^{\alpha}\nabla\cdot v\|_{L^2(\Omega)}
+|\omega|_{\infty}\|\mathcal{D}^{\alpha}\nabla v\|_{L^2(\Omega)}
+ \|\mathcal{D}^{\alpha_1}\omega\|_{L^4(\Omega)}\|\mathcal{D}^{\alpha_2}\nabla v\|_{L^4(\Omega)} \\[6pt]\qquad

+|\nabla v|_{\infty}\|\mathcal{D}^{\alpha}\omega\|_{L^2(\Omega)}
+ |\nabla(\frac{p}{\varrho})|_{\infty}\|\mathcal{D}^{\alpha}\nabla(\nabla\cdot v)\|_{L^2(\Omega)} \\[6pt]\qquad
+ \|\mathcal{D}^{\alpha_1}\nabla(\frac{p}{\varrho})\|_{L^4(\Omega)}\|\mathcal{D}^{\alpha_2}\nabla(\nabla\cdot v)\|_{L^4(\Omega)}
+ |\nabla(\nabla\cdot v)|_{\infty}\|\mathcal{D}^{\alpha}\nabla(\frac{p}{\varrho})\|_{L^2(\Omega)}) \\[6pt]\qquad
-2k_1(\int\limits_{\partial\Omega}v\cdot n |\mathcal{D}^{\alpha}\omega|^2 \,\mathrm{d}S_x
-\int\limits_{\Omega}\nabla\cdot v |\mathcal{D}^{\alpha}\omega|^2 \,\mathrm{d}x)

\lem \sqrt{\ve}\mathcal{F}[\xi,v](t),
\end{array}
\end{equation}
where $|\alpha_1|=|\alpha_2|=1$.

When $\ell=1,|\alpha|=1$,
\begin{equation}\label{Sect5_Vorticity_6}
\begin{array}{ll}
I_4\lem \|\mathcal{D}^{\alpha}\omega_t\|_{L^2(\Omega)}
(|\nabla\cdot v|_{\infty}\|\mathcal{D}^{\alpha}\omega_t\|_{L^2(\Omega)}
+ \|\mathcal{D}^{\alpha}\omega\|_{L^4(\Omega)}\|\nabla\cdot v_t\|_{L^4(\Omega)}\\[6pt]\qquad
+ \|\omega_t\|_{L^4(\Omega)}\|\mathcal{D}^{\alpha}\nabla\cdot v\|_{L^4(\Omega)}
+ |\omega|_{\infty}\|\mathcal{D}^{\alpha}\nabla\cdot v_t\|_{L^2(\Omega)}\\[6pt]\qquad
+ |\nabla(\nabla\cdot v)|_{\infty}\|\mathcal{D}^{\alpha} v_t\|_{L^2(\Omega)}
+ |\mathcal{D}^{\alpha}v|_{\infty}\|\nabla(\nabla\cdot v_t)\|_{L^2(\Omega)} \\[6pt]\qquad
+ |v_t|_{\infty}\|\mathcal{D}^{\alpha}\nabla(\nabla\cdot v)\|_{L^2(\Omega)}
+ |v|_{\infty}\|\mathcal{D}^{\alpha}\nabla(\nabla\cdot v_t)\|_{L^2(\Omega)} \\[6pt]\qquad

+ \|\mathcal{D}^{\alpha} v_{t}\|_{L^4(\Omega)}\|\nabla\omega\|_{L^4(\Omega)}
+|v_t|_{\infty}\|\mathcal{D}^{\alpha}\nabla\omega\|_{L^2(\Omega)}
+|\mathcal{D}^{\alpha} v|_{\infty}\|\nabla\omega_t\|_{L^2(\Omega)} \\[6pt]\qquad
+|\nabla\cdot v|_{\infty}\|\mathcal{D}^{\alpha}\omega_t\|_{L^2(\Omega)}
+\|\omega_t\|_{L^4(\Omega)}\|\mathcal{D}^{\alpha}\nabla\cdot v\|_{L^4(\Omega)} \\[6pt]\qquad
+\|\mathcal{D}^{\alpha}\omega\|_{L^4(\Omega)}\|\nabla\cdot v_t\|_{L^4(\Omega)}
+ |\omega|_{\infty}\|\mathcal{D}^{\alpha}\nabla\cdot v_t\|_{L^2(\Omega)} \\[6pt]\qquad

+|\omega|_{\infty}\|\mathcal{D}^{\alpha}\nabla v_t\|_{L^2(\Omega)}
+ \|\omega_{t}\|_{L^4(\Omega)}\|\mathcal{D}^{\alpha}\nabla v\|_{L^4(\Omega)}
+ \|\mathcal{D}^{\alpha}\omega\|_{L^4(\Omega)}\|\nabla v_t\|_{L^4(\Omega)} \\[6pt]\qquad
+|\nabla v|_{\infty}\|\mathcal{D}^{\alpha}\omega_t\|_{L^2(\Omega)}

+ |\nabla(\frac{p}{\varrho})|_{\infty}\|\mathcal{D}^{\alpha}\nabla(\nabla\cdot v_t)\|_{L^2(\Omega)} \\[6pt]\qquad
+ \|\mathcal{D}^{\alpha}\nabla(\frac{p}{\varrho})\|_{L^4(\Omega)}\|\nabla(\nabla\cdot v_t)\|_{L^4(\Omega)}
+ \|\partial_t\nabla(\frac{p}{\varrho})\|_{L^4(\Omega)}\|\mathcal{D}^{\alpha}\nabla(\nabla\cdot v)\|_{L^4(\Omega)}
\\[6pt]\qquad
+ |\nabla(\nabla\cdot v)|_{\infty}\|\partial_t\mathcal{D}^{\alpha}\nabla(\frac{p}{\varrho})\|_{L^2(\Omega)})
-2k_1(\int\limits_{\partial\Omega}v\cdot n |\mathcal{D}^{\alpha}\omega_t|^2 \,\mathrm{d}S_x  \\[6pt]\qquad
-\int\limits_{\Omega}\nabla\cdot v |\mathcal{D}^{\alpha}\omega_t|^2 \,\mathrm{d}x)

\lem \sqrt{\ve}\mathcal{F}[\xi,v](t).
\end{array}
\end{equation}

When $\ell=2,|\alpha|=0$,
\begin{equation*}
\begin{array}{ll}
I_4\lem \|\omega_{tt}\|_{L^2(\Omega)}(|\nabla\cdot v|_{\infty}\|\omega_{tt}\|_{L^2(\Omega)}
+ \|\omega_t\|_{L^4(\Omega)}\|\nabla\cdot v_t\|_{L^4(\Omega)}\\[6pt]\qquad
+ |\omega|_{\infty}\|\nabla\cdot v_{tt}\|_{L^2(\Omega)}
+ |\nabla(\nabla\cdot v)|_{\infty}\|v_{tt}\|_{L^2(\Omega)}
+ |v_t|_{\infty}\|\nabla(\nabla\cdot v_t)\|_{L^2(\Omega)} \\[6pt]\qquad
+ |v|_{\infty}\|\nabla(\nabla\cdot v_{tt})\|_{L^2(\Omega)}

+ \|v_{tt}\|_{L^4(\Omega)}\|\nabla\omega\|_{L^4(\Omega)}
+|v_t|_{\infty}\|\nabla\omega_t\|_{L^2(\Omega)} \\[6pt]\qquad
+|\nabla\cdot v|_{\infty}\|\omega_{tt}\|_{L^2(\Omega)}+\|\omega_t\|_{L^4(\Omega)}\|\nabla\cdot v_t\|_{L^4(\Omega)}
+ |\omega|_{\infty}\|\nabla\cdot v_{tt}\|_{L^2(\Omega)} \\[6pt]\qquad

+|\omega|_{\infty}\|\nabla v_{tt}\|_{L^2(\Omega)} + \|\omega_{t}\|_{L^4(\Omega)}\|\nabla v_t\|_{L^4(\Omega)}
+|\nabla v|_{\infty}\|\omega_{tt}\|_{L^2(\Omega)}
\end{array}
\end{equation*}

\begin{equation}\label{Sect5_Vorticity_7}
\begin{array}{ll}
\qquad
+ |\nabla(\frac{p}{\varrho})|_{\infty}\|\nabla(\nabla\cdot v_{tt})\|_{L^2(\Omega)}
+ \|\partial_t\nabla(\frac{p}{\varrho})\|_{L^4(\Omega)}\|\nabla(\nabla\cdot v_t)\|_{L^4(\Omega)} \\[6pt]\qquad
+ |\nabla(\nabla\cdot v)|_{\infty}\|\partial_{tt}\nabla(\frac{p}{\varrho})\|_{L^2(\Omega)})
-2k_1(\int\limits_{\partial\Omega}v\cdot n |\omega_{tt}|^2 \,\mathrm{d}S_x
-\int\limits_{\Omega}\nabla\cdot v |\omega_{tt}|^2 \,\mathrm{d}x) \\[6pt]\quad

\lem \sqrt{\ve}\mathcal{F}[\xi,v](t).
\end{array}
\end{equation}

After summing the above estimates, we have
\begin{equation}\label{Sect5_Vorticity_9}
\begin{array}{ll}
\frac{\mathrm{d}}{\mathrm{d}t}\int\limits_{\Omega}\mathcal{E}_1[\omega](t)\,\mathrm{d}x
+ 2a\mathcal{E}_1[\omega](t)
\lem \sqrt{\ve}\mathcal{F}[\xi,v](t).
\end{array}
\end{equation}
Thus, Lemma $\ref{Sect5_Vorticity_Lemma}$ is proved.
\end{proof}

Similar to Lemma $\ref{Sect3_Epsilon0_Lemma}$, the following lemma states that $E[\xi](t)$ and $E_1[\xi](t)$ are equivalent, $E[v](t)$ and $E_1[v](t)$ are equivalent.
\begin{lemma}\label{Sect5_Epsilon0_Lemma}
For any given $T\in (0,+\infty]$, there exists $\ve_1>0$ which is independent of $(\hat{\xi}_0,\hat{v}_0,\hat{\phi}_0)$, such that if $\sup\limits_{0\leq t\leq T} \mathcal{F}[\xi,v,\phi](t) \leq\ve_1$, then there exist $c_{11}>0,c_{12}>0$ such that
\begin{equation}\label{Sect5_Energy_Equivalence}
\begin{array}{ll}
c_{11} E[\xi](t) \leq E_1[\xi](t) \leq c_{12} E[\xi](t),\\[6pt]
c_{11} E[v](t) \leq E_1[v](t) \leq c_{12} E[v](t).
\end{array}
\end{equation}
\end{lemma}

To make calculations simpler, we calculate $\frac{\mathrm{d}}{\mathrm{d}t}E_1[\xi](t)$ and $\frac{\mathrm{d}}{\mathrm{d}t}E[v](t)$ separately.

Since $E_1[\xi](t)\cong E[\xi](t)$, the following lemma gives an equivalent a priori estimate for $E[\xi](t)$.
\begin{lemma}\label{Sect5_Energy_Estimate_Lemma2}
For any given $T\in (0,+\infty]$, if
\begin{equation*}
\sup\limits_{0\leq t\leq T} \mathcal{F}[\xi,v,\phi](t) \leq\ve,
\end{equation*}
where $0<\ve\ll 1$, then for $\forall t\in [0,T]$,
\begin{equation}\label{Sect5_Estimate1_toProve}
\frac{\mathrm{d}}{\mathrm{d}t}E_1[\xi](t)+ 2a E[v](t) \leq C\sqrt{\ve}\mathcal{F}[\xi,v](t).
\end{equation}
\end{lemma}

\begin{proof}
Suppose $0\leq \ell \leq 3$, apply $\partial_t^{\ell}$ to $(\ref{Sect2_Final_Eq})$, we get
\begin{equation}\label{Sect5_Estimate2_1}
\left\{\begin{array}{ll}
(\partial_t^{\ell}\xi)_t + k_2\nabla\cdot(\partial_t^{\ell}v) = -\gamma k_1\partial_t^{\ell}(\xi\nabla\cdot v)
- k_1 \partial_t^{\ell}(v\cdot\nabla \xi), \\[6pt]
k_2 \nabla(\partial_t^{\ell}\xi) + a \partial_t^{\ell}v =
\frac{1}{k_1} \partial_t^{\ell}[(\frac{1}{\bar{\varrho}}- \frac{1}{\varrho})\nabla\xi].
\end{array}\right.
\end{equation}

Let $(\ref{Sect5_Estimate2_1})\cdot(\partial_t^{\ell}\xi, \partial_t^{\ell}v)$, we get
\begin{equation}\label{Sect5_Estimate2_2}
\left\{\begin{array}{ll}
(|\partial_t^{\ell}\xi|^2)_t + 2k_2 \partial_t^{\ell}\xi\nabla\cdot(\partial_t^{\ell}v)
= -2\gamma k_1 \partial_t^{\ell}\xi\partial_t^{\ell}(\xi\nabla\cdot v)
- 2k_1 \partial_t^{\ell}\xi\partial_t^{\ell}(v\cdot\nabla \xi), \\[6pt]
2k_2 \partial_t^{\ell}v\cdot\nabla(\partial_t^{\ell}\xi)
+ 2a |\partial_t^{\ell}v|^2
= \frac{2}{k_1} \partial_t^{\ell}v\cdot
\partial_t^{\ell}[(\frac{1}{\bar{\varrho}}- \frac{1}{\varrho})\nabla\xi].
\end{array}\right.
\end{equation}

By $(\ref{Sect5_Estimate2_2})_1+(\ref{Sect5_Estimate2_2})_2$, we get
\begin{equation}\label{Sect5_Estimate2_3}
\begin{array}{ll}
(|\partial_t^{\ell}\xi|^2)_t
+ 2k_2 \partial_t^{\ell}\xi\nabla\cdot(\partial_t^{\ell}v)+ 2k_2 \partial_t^{\ell}v\cdot\nabla(\partial_t^{\ell}\xi)
+ 2a |\partial_t^{\ell}v|^2 \\[6pt]
= -2\gamma k_1 \partial_t^{\ell}\xi\partial_t^{\ell}(\xi\nabla\cdot v)
- 2k_1 \partial_t^{\ell}\xi\partial_t^{\ell}(v\cdot\nabla \xi)
+ \frac{2}{k_1} \partial_t^{\ell}v\cdot
\partial_t^{\ell}[(\frac{1}{\bar{\varrho}}- \frac{1}{\varrho})\nabla\xi].
\end{array}
\end{equation}

After integrating $(\ref{Sect5_Estimate2_3})$ in $\Omega$, we get
\begin{equation}\label{Sect5_Estimate2_4}
\begin{array}{ll}
\frac{\mathrm{d}}{\mathrm{d} t}\int\limits_{\Omega}|\partial_t^{\ell}\xi|^2 \,\mathrm{d}x
+ 2k_2 \int\limits_{\Omega}\partial_t^{\ell}\xi\nabla\cdot(\partial_t^{\ell}v)
+ \partial_t^{\ell}v\cdot\nabla(\partial_t^{\ell}\xi) \,\mathrm{d}x
+ 2a \int\limits_{\Omega}|\partial_t^{\ell}v|^2 \,\mathrm{d}x \\[6pt]
= \int\limits_{\Omega} -2\gamma k_1 \partial_t^{\ell}\xi\partial_t^{\ell}(\xi\nabla\cdot v)
- 2k_1 \partial_t^{\ell}\xi\partial_t^{\ell}(v\cdot\nabla \xi)
+ \frac{2}{k_1} \partial_t^{\ell}v\cdot
\partial_t^{\ell}[(\frac{1}{\bar{\varrho}}- \frac{1}{\varrho})\nabla\xi] \,\mathrm{d}x.
\end{array}
\end{equation}

Since $\partial_t^{\ell}v\cdot n|_{\partial\Omega}=0$,
$\int\limits_{\Omega}\partial_t^{\ell}\xi\nabla\cdot(\partial_t^{\ell}v)
+ \partial_t^{\ell}v\cdot\nabla(\partial_t^{\ell}\xi) \,\mathrm{d}x
= \int\limits_{\partial\Omega}\partial_t^{\ell}\xi\partial_t^{\ell}v \cdot n \,\mathrm{d}S_x =0$,
\begin{equation}\label{Sect5_Estimate2_5}
\begin{array}{ll}
\frac{\mathrm{d}}{\mathrm{d} t}
\int\limits_{\Omega}|\partial_t^{\ell}\xi|^2 \,\mathrm{d}x
+ 2a \int\limits_{\Omega}|\partial_t^{\ell}v|^2 \,\mathrm{d}x \\[6pt]
= \int\limits_{\Omega} -2\gamma k_1 \partial_t^{\ell}\xi\partial_t^{\ell}(\xi\nabla\cdot v)
- 2k_1 \partial_t^{\ell}\xi\partial_t^{\ell}(v\cdot\nabla \xi)
+ \frac{2}{k_1} \partial_t^{\ell}v\cdot
\partial_t^{\ell}[(\frac{1}{\bar{\varrho}}- \frac{1}{\varrho})\nabla\xi] \,\mathrm{d}x := I_6.
\end{array}
\end{equation}

\vspace{0.4cm}
\noindent
When $0\leq \ell\leq 2$, it is easy to check that $I_6\lem \sqrt{\ve}\mathcal{F}[\xi,v](t)$, since $I_6$ is a lower order term.

\noindent
When $\ell=3$,
\begin{equation}\label{Sect5_Estimate2_6}
\begin{array}{ll}
I_6= \int\limits_{\Omega} -2\gamma k_1 \xi_{ttt}\partial_{ttt}(\xi\nabla\cdot v)
- 2k_1 \xi_{ttt}\partial_{ttt}(v\cdot\nabla \xi)
+ \frac{2}{k_1} v_{ttt}\cdot\partial_{ttt}[(\frac{1}{\bar{\varrho}}- \frac{1}{\varrho})\nabla\xi] \,\mathrm{d}x \\[6pt]\quad

=\int\limits_{\Omega} -2\gamma k_1 (\xi_{ttt})^2 \nabla\cdot v-6\gamma k_1 \xi_{ttt}\xi_{tt}\nabla\cdot v_t
-6\gamma k_1 \xi_{ttt}\xi_t\nabla\cdot v_{tt} \\[6pt]\qquad
-2\gamma k_1 \xi_{ttt}\xi\nabla\cdot v_{ttt} - 2k_1 \xi_{ttt} v_{ttt}\cdot\nabla \xi
-6k_1 \xi_{ttt} v_{tt}\cdot\nabla \xi_t \\[6pt]\qquad
-6k_1 \xi_{ttt} v_t\cdot\nabla \xi_{tt}- 2k_1 \xi_{ttt} v\cdot\nabla \xi_{ttt}
+ \frac{2}{k_1} (\frac{1}{\bar{\varrho}}- \frac{1}{\varrho})v_{ttt}\cdot\nabla\xi_{ttt}  \\[6pt]\qquad
+ \frac{6}{k_1} (\frac{\varrho_t}{\varrho^2}v_{ttt}\cdot\nabla\xi_{tt}
+ \frac{\varrho\varrho_{tt}-2\varrho_t^2}{\varrho^3}v_{ttt}\cdot\nabla\xi_t)
+ \frac{2}{k_1} \frac{\varrho^2\varrho_{ttt}-6\varrho\varrho_t\varrho_{tt}
+6\varrho_t^3}{\varrho^4}v_{ttt}\cdot\nabla\xi \,\mathrm{d}x.
\end{array}
\end{equation}

\noindent
Now we estimate $I_6 - \frac{\mathrm{d}}{\mathrm{d}t}\int\limits_{\Omega}\frac{\xi}{p}\xi_{ttt}^2 \,\mathrm{d}x$,
\begin{equation}\label{Sect5_Estimate2_7}
\begin{array}{ll}
I_6 - \frac{\mathrm{d}}{\mathrm{d}t}\int\limits_{\Omega}\frac{\xi}{p}\xi_{ttt}^2 \,\mathrm{d}x \\[6pt]
\lem \sqrt{\ve}\mathcal{F}[\xi,v](t)
+ \|\xi_{tt}\|_{L^4(\Omega)}\|\nabla\cdot v_t\|_{L^4(\Omega)} \|\xi_{ttt}\|_{L^2(\Omega)}
-2\gamma k_1 \int\limits_{\Omega}\xi_{ttt}\xi\nabla\cdot v_{ttt}\,\mathrm{d}x \\[6pt]\quad

+\|v_{tt}\|_{L^4(\Omega)}\|\nabla \xi_t\|\|_{L^4(\Omega)} \|\xi_{ttt}\|\|_{L^2(\Omega)}
- 2k_1 \int\limits_{\Omega}\xi_{ttt} v\cdot\nabla \xi_{ttt}\,\mathrm{d}x \\[6pt]\quad

+ \frac{2}{k_1} \int\limits_{\Omega}(\frac{1}{\bar{\varrho}}
- \frac{1}{\varrho})v_{ttt}\cdot\nabla\xi_{ttt}\,\mathrm{d}x
+ \|\varrho_{tt}\|_{L^4(\Omega)}\|\nabla\xi_t\|_{L^4(\Omega)}\|v_{ttt}\|_{L^2(\Omega)} \\[6pt]\quad
- 2\int\limits_{\Omega}\frac{\xi}{p}\xi_{ttt}\xi_{tttt} \,\mathrm{d}x
- \int\limits_{\Omega}\partial_t(\frac{\xi}{p})\xi_{ttt}^2 \,\mathrm{d}x
\\[6pt]
\lem \sqrt{\ve}\mathcal{F}[\xi,v](t)
- k_1 \int\limits_{\partial\Omega}|\xi_{ttt}|^2 v\cdot n\,\mathrm{d}S_x
+ k_1 \int\limits_{\Omega}|\xi_{ttt}|^2 \nabla\cdot v\,\mathrm{d}x \\[6pt]\quad
-2\gamma k_1 \int\limits_{\Omega}\xi\xi_{ttt}\nabla\cdot v_{ttt}\,\mathrm{d}x
+ \frac{2}{k_1} \int\limits_{\Omega}(\frac{1}{\bar{\varrho}}- \frac{1}{\varrho})v_{ttt}\cdot\nabla\xi_{ttt}\,\mathrm{d}x
- 2\int\limits_{\Omega}\frac{\xi}{p}\xi_{ttt}\xi_{tttt} \,\mathrm{d}x \\[6pt]

\lem -2\int\limits_{\Omega}\frac{\xi}{p}\xi_{ttt}(\xi_{tttt}+k_1\gamma p\nabla\cdot v_{ttt})\,\mathrm{d}x
+ \frac{2}{k_1} \int\limits_{\Omega}(\frac{1}{\bar{\varrho}}- \frac{1}{\varrho})v_{ttt}\cdot\nabla\xi_{ttt}\,\mathrm{d}x + \sqrt{\ve}\mathcal{F}[\xi,v](t).
\end{array}
\end{equation}

Apply $\partial_{ttt}$ to $(\ref{Sect2_Final_Diffusion})_1$, we get
\begin{equation}\label{Sect5_Estimate2_8}
\begin{array}{ll}
\xi_{tttt} + k_1\gamma p\nabla\cdot v_{ttt} = -k_1v\cdot\nabla\xi_{ttt} -3 k_1v_t\cdot\nabla\xi_{tt}
-3 k_1v_{tt}\cdot\nabla\xi_t -k_1v_{ttt}\cdot\nabla\xi \\[6pt]\hspace{3.3cm}
- 3k_1\gamma\varrho_t\nabla\cdot v_{tt} - 3k_1\gamma\xi_{tt}\nabla\cdot v_t - k_1\gamma \xi_{ttt} \nabla\cdot v.
\end{array}
\end{equation}

Plug $(\ref{Sect5_Estimate2_8})$ into the following integral, we get
\begin{equation}\label{Sect5_Estimate2_9}
\begin{array}{ll}
\int\limits_{\Omega}\frac{\xi}{p}\xi_{ttt}(\xi_{tttt}+k_1\gamma p\nabla\cdot v_{ttt})\,\mathrm{d}x
= \int\limits_{\Omega}\frac{\xi}{p}\xi_{ttt}[R.H.S.\ of\ (\ref{Sect5_Estimate2_8})]\,\mathrm{d}x \\[6pt]
\lem \sqrt{\ve}\mathcal{F}[\xi,v](t)
-k_1\int\limits_{\partial\Omega}\frac{\xi}{2p}|\xi_{ttt}|^2 v\cdot n \,\mathrm{d}S_x
+ \frac{k_1}{2}\int\limits_{\Omega}|\xi_{ttt}|^2 \nabla\cdot(\frac{\xi}{p}v) \,\mathrm{d}x

\lem \sqrt{\ve}\mathcal{F}[\xi,v](t).
\end{array}
\end{equation}

Apply $\partial_{ttt}$ to $(\ref{Sect2_Final_Diffusion})_2$, we get
\begin{equation}\label{Sect5_Estimate2_10}
\begin{array}{ll}
\nabla\xi_{ttt} = - ak_1 \varrho v_{ttt} -3 ak_1 \varrho_t v_{tt}
-3 ak_1 \varrho_{tt}v_t - ak_1 \varrho_{ttt}.
\end{array}
\end{equation}

Plug $(\ref{Sect5_Estimate2_10})$ into the following integral, we get
\begin{equation}\label{Sect5_Estimate2_11}
\begin{array}{ll}
\int\limits_{\Omega}(\frac{1}{\bar{\varrho}}-\frac{1}{\varrho})v_{ttt}\cdot\nabla\xi_{ttt}\,\mathrm{d}x
= \int\limits_{\Omega}(\frac{1}{\bar{\varrho}}-\frac{1}{\varrho})v_{ttt}\cdot
[R.H.S.\ of\ (\ref{Sect5_Estimate2_10})]\,\mathrm{d}x \lem \sqrt{\ve}\mathcal{F}[\xi,v](t).
\end{array}
\end{equation}

Plug $(\ref{Sect5_Estimate2_9}),(\ref{Sect5_Estimate2_9})$ into $(\ref{Sect5_Estimate2_7})$, we get
\begin{equation}\label{Sect5_Estimate2_12}
\begin{array}{ll}
I_6 - \frac{\mathrm{d}}{\mathrm{d}t}\int\limits_{\Omega}\frac{\xi}{p}\xi_{ttt}^2 \,\mathrm{d}x
\lem \sqrt{\ve}\mathcal{F}[\xi,v](t).
\end{array}
\end{equation}

Finally, we have
\begin{equation}\label{Sect5_Estimate2_13}
\begin{array}{ll}
\frac{\mathrm{d}}{\mathrm{d} t}\left(\sum\limits_{\ell=0}^{3}
\int\limits_{\Omega}|\partial_t^{\ell}\xi|^2 \,\mathrm{d}x
-\int\limits_{\Omega}\frac{\xi}{p} \xi_{ttt}^2\,\mathrm{d}x \right)
+ 2a E[v](t) \lem \sqrt{\ve}\mathcal{F}[\xi,v](t).
\end{array}
\end{equation}
Thus, Lemma $\ref{Sect5_Energy_Estimate_Lemma2}$ is proved.
\end{proof}

The following lemma concerns a priori estimate for $E[v](t)$.
\begin{lemma}\label{Sect5_Velocity_Lemma}
For any given $T\in (0,+\infty]$, if
\begin{equation*}
\sup\limits_{0\leq t\leq T} \mathcal{F}[\xi,v,\phi](t) \leq\ve,
\end{equation*}
where $0<\ve\ll 1$, then for $\forall t\in [0,T]$,
\begin{equation}\label{Sect5_Velocity_toProve}
\begin{array}{ll}
\frac{\mathrm{d}}{\mathrm{d}t} E[v](t)
+ \frac{\gamma}{a}\int\limits_{\Omega} \frac{p}{\varrho}\sum\limits_{\ell=0}^3|\nabla\cdot \partial_t^{\ell}v|^2\,\mathrm{d}x \lem \sqrt{\ve}\mathcal{F}[\xi,v](t).
\end{array}
\end{equation}
\end{lemma}

\begin{proof}
Apply $\partial_t^{\ell}$ to the equation $(\ref{Sect2_Velocity_Solve 3})$, where $1\leq \ell\leq 3$, we get
\begin{equation}\label{Sect5_Velocity_1}
\begin{array}{ll}
\partial_t^{\ell+1} v = \partial_t^{\ell}
[k_1(1-\gamma)v(\nabla\cdot v)- k_1 v\cdot\nabla v
- \frac{k_1}{2} \nabla(|v|^2)] \\[6pt]\hspace{1.5cm}

+ \frac{\gamma}{a}\sum\limits_{0\leq \mu<\ell}
[\partial_t^{\ell-\mu}(\frac{p}{\varrho})\partial_t^{\mu}\nabla(\nabla\cdot v)]
+ \frac{\gamma p}{a\varrho}\partial_t^{\ell}\nabla(\nabla\cdot v).
\end{array}
\end{equation}

Let $(\ref{Sect5_Velocity_1})\cdot \partial_t^{\ell} v$, we get
\begin{equation}\label{Sect5_Velocity_2}
\begin{array}{ll}
\partial_t|\partial_t^{\ell} v|^2 = 2\partial_t^{\ell} v\cdot\partial_t^{\ell}
[k_1(1-\gamma)v(\nabla\cdot v)- k_1 v\cdot\nabla v - \frac{k_1}{2} \nabla(|v|^2)] \\[6pt]\hspace{1.5cm}

+ \frac{2\gamma}{a}\sum\limits_{0\leq \mu<\ell}
[\partial_t^{\ell-\mu}(\frac{p}{\varrho})\partial_t^{\mu}\nabla(\nabla\cdot v)]\cdot\partial_t^{\ell} v
+ \frac{2\gamma p}{a\varrho}\partial_t^{\ell}\nabla(\nabla\cdot v)\cdot\partial_t^{\ell} v.
\end{array}
\end{equation}

Integrate $(\ref{Sect5_Velocity_2})$ in $\Omega$, we get
\begin{equation}\label{Sect5_Velocity_3}
\begin{array}{ll}
\frac{\mathrm{d}}{\mathrm{d}t}\int\limits_{\Omega}|\partial_t^{\ell} v|^2 \,\mathrm{d}x
= 2\int\limits_{\Omega}\partial_t^{\ell} v\cdot\partial_t^{\ell}
[k_1(1-\gamma)v(\nabla\cdot v)- k_1 v\cdot\nabla v
- \frac{k_1}{2} \nabla(|v|^2)]\,\mathrm{d}x \\[6pt]\hspace{2.4cm}

+ \frac{2\gamma}{a}\int\limits_{\Omega}\sum\limits_{0\leq \mu<\ell}
[\partial_t^{\ell-\mu}(\frac{p}{\varrho})\partial_t^{\mu}\nabla(\nabla\cdot v)]\cdot\partial_t^{\ell} v \,\mathrm{d}x
\\[6pt]\hspace{2.4cm}

+ \frac{2\gamma}{a}\int\limits_{\Omega}\frac{p}{\varrho}\partial_t^{\ell}\nabla(\nabla\cdot v)
\cdot\partial_t^{\ell} v \,\mathrm{d}x := I_7.
\end{array}
\end{equation}

When $\ell=0$,
\begin{equation}\label{Sect5_Velocity_4}
\begin{array}{ll}
\int\limits_{\Omega} \frac{p}{\varrho}v\cdot\nabla(\nabla\cdot v)\,\mathrm{d}x
= \int\limits_{\partial\Omega} \frac{p}{\varrho}v\cdot n \nabla\cdot v\,\mathrm{d}S_x
- \int\limits_{\Omega} \frac{p}{\varrho}|\nabla\cdot v|^2\,\mathrm{d}x
- \int\limits_{\Omega} \nabla\cdot v \frac{\varrho v\cdot\nabla\xi - p v\cdot\nabla\varrho}{\varrho^2}\,\mathrm{d}x
\\[10pt]\hspace{2.8cm}
\lem \sqrt{\ve}\mathcal{F}[\xi,v] - \int\limits_{\Omega} \frac{p}{\varrho}|\nabla\cdot v|^2\,\mathrm{d}x
\end{array}
\end{equation}

Then
\begin{equation}\label{Sect5_Velocity_5}
\begin{array}{ll}
\frac{\mathrm{d}}{\mathrm{d}t} \int\limits_{\Omega}|v|^2 \,\mathrm{d}x
+ \frac{2\gamma}{a}\int\limits_{\Omega} \frac{p}{\varrho}|\nabla\cdot v|^2\,\mathrm{d}x
\lem \sqrt{\ve}\mathcal{F}[\xi,v]
\end{array}
\end{equation}

When $\ell=1$,
\begin{equation}\label{Sect5_Velocity_6}
\begin{array}{ll}
\int\limits_{\Omega} \frac{p}{\varrho}v_t\cdot\nabla(\nabla\cdot v_t)\,\mathrm{d}x
= \int\limits_{\partial\Omega} \frac{p}{\varrho}v_t\cdot n \nabla\cdot v_t\,\mathrm{d}S_x
- \int\limits_{\Omega} \frac{p}{\varrho}|\nabla\cdot v_t|^2\,\mathrm{d}x  \\[10pt]\hspace{3.4cm}
- \int\limits_{\Omega} (\nabla\cdot v_t)v_t\cdot\nabla(\frac{p}{\varrho}) \,\mathrm{d}x
\\[10pt]\hspace{3cm}
\lem \sqrt{\ve}\mathcal{F}[\xi,v] - \int\limits_{\Omega} \frac{p}{\varrho}|\nabla\cdot v_t|^2\,\mathrm{d}x
\end{array}
\end{equation}

Then
\begin{equation}\label{Sect5_Velocity_7}
\begin{array}{ll}
\frac{\mathrm{d}}{\mathrm{d}t} \int\limits_{\Omega}|v_t|^2 \,\mathrm{d}x
+ \frac{2\gamma}{a}\int\limits_{\Omega} \frac{p}{\varrho}|\nabla\cdot v_t|^2\,\mathrm{d}x
\lem \sqrt{\ve}\mathcal{F}[\xi,v]
\end{array}
\end{equation}

When $\ell=2$,
\begin{equation}\label{Sect5_Velocity_8}
\begin{array}{ll}
\int\limits_{\Omega} \frac{p}{\varrho}v_{tt}\cdot\nabla(\nabla\cdot v_{tt})\,\mathrm{d}x
= \int\limits_{\partial\Omega} \frac{p}{\varrho}v_{tt}\cdot n \nabla\cdot v_{tt}\,\mathrm{d}S_x
- \int\limits_{\Omega} \frac{p}{\varrho}|\nabla\cdot v_{tt}|^2\,\mathrm{d}x  \\[10pt]\hspace{3.6cm}
- \int\limits_{\Omega} (\nabla\cdot v_{tt})v_{tt}\cdot\nabla(\frac{p}{\varrho}) \,\mathrm{d}x
\\[10pt]\hspace{3.2cm}
\lem \sqrt{\ve}\mathcal{F}[\xi,v] - \int\limits_{\Omega} \frac{p}{\varrho}|\nabla\cdot v_{tt}|^2\,\mathrm{d}x
\end{array}
\end{equation}

Then
\begin{equation}\label{Sect5_Velocity_9}
\begin{array}{ll}
\frac{\mathrm{d}}{\mathrm{d}t} \int\limits_{\Omega}|v_{tt}|^2 \,\mathrm{d}x
+ \frac{2\gamma}{a}\int\limits_{\Omega} \frac{p}{\varrho}|\nabla\cdot v_{tt}|^2\,\mathrm{d}x
\lem \sqrt{\ve}\mathcal{F}[\xi,v]
\end{array}
\end{equation}

When $\ell=3$, we estimate each term of $I_7$ separately.

The first term of $I_7$:
\begin{equation}\label{Sect5_Velocity_10}
\begin{array}{ll}
2k_1(1-\gamma)\int\limits_{\Omega}v_{ttt}\cdot\partial_{ttt}[v(\nabla\cdot v)]\,\mathrm{d}x
\\[6pt]
\lem \|v_{ttt}\|_{L^2{(\Omega)}}(\|v_{ttt}\|_{L^2(\Omega)}|\nabla\cdot v|_{\infty}
+\|v_{tt}\|_{L^4(\Omega)}\|\nabla\cdot v_t\|_{L^4(\Omega)} \\[6pt]\quad
+|v_{t}|_{\infty}\|\nabla\cdot v_{tt}\|_{L^2(\Omega)})
+ 2k_1(1-\gamma)\int\limits_{\Omega}v_{ttt}\cdot v(\nabla\cdot v_{ttt})\,\mathrm{d}x
\\[6pt]
\leq C\sqrt{\ve}\mathcal{F}[\xi,v](t)
+ \frac{\gamma}{3a}\int\limits_{\Omega}\frac{p}{\varrho}|\nabla\cdot v_{ttt}|^2\,\mathrm{d}x
+ C\int\limits_{\Omega}\frac{\varrho}{p}(v\cdot v_{ttt})^2\,\mathrm{d}x

\\[6pt]
\leq C\sqrt{\ve}\mathcal{F}[\xi,v](t)
+ \frac{\gamma}{3a}\int\limits_{\Omega}\frac{p}{\varrho}|\nabla\cdot v_{ttt}|^2\,\mathrm{d}x.
\end{array}
\end{equation}

The second term of $I_7$:
\begin{equation}\label{Sect5_Velocity_11}
\begin{array}{ll}
-2k_1\int\limits_{\Omega}v_{ttt}\cdot\partial_{ttt}[v\cdot\nabla v]\,\mathrm{d}x \\[6pt]
\lem \|v_{ttt}\|_{L^2{(\Omega})}(\|v_{ttt}\|_{L^2(\Omega)}|\nabla v|_{\infty}
+\|v_{tt}\|_{L^4(\Omega)}\|\nabla v_t\|_{L^4(\Omega)} \\[6pt]\quad
+|v_{t}|_{\infty}\|\nabla v_{tt}\|_{L^2(\Omega)})
-k_1\int\limits_{\Omega}v\cdot\nabla |v_{ttt}|^2\,\mathrm{d}x \\[6pt]

\lem \sqrt{\ve}\mathcal{F}[\xi,v](t)-k_1\int\limits_{\partial\Omega}v\cdot n |v_{ttt}|^2\,\mathrm{d}S_x
+k_1\int\limits_{\Omega}\nabla\cdot v |v_{ttt}|^2\,\mathrm{d}x
\lem \sqrt{\ve}\mathcal{F}[\xi,v](t).
\end{array}
\end{equation}

The third term of $I_7$:
\begin{equation}\label{Sect5_Velocity_12}
\begin{array}{ll}
-k_1\int\limits_{\Omega}v_{ttt}\cdot\partial_{ttt}[\nabla(|v|^2)]\,\mathrm{d}x \\[6pt]
=-k_1\int\limits_{\partial\Omega}v_{ttt}\cdot n \partial_{ttt}(|v|^2)\,\mathrm{d}S_x
+k_1\int\limits_{\Omega}\nabla\cdot v_{ttt}\partial_{ttt}(|v|^2)\,\mathrm{d}x \\[6pt]
= k_1\int\limits_{\Omega}\nabla\cdot v_{ttt}\partial_{ttt}(|v|^2)\,\mathrm{d}x

\leq \frac{\gamma}{3a}\int\limits_{\Omega}\frac{p}{\varrho}|\nabla\cdot v_{ttt}|^2\,\mathrm{d}x
+ C\int\limits_{\Omega}\frac{\varrho}{p}|\partial_{ttt}(|v|^2)|^2\,\mathrm{d}x \\[6pt]

\leq C\sqrt{\ve}\mathcal{F}[\xi,v](t)
+ \frac{\gamma}{3a}\int\limits_{\Omega}\frac{p}{\varrho}|\nabla\cdot v_{ttt}|^2\,\mathrm{d}x.
\end{array}
\end{equation}

The fourth term of $I_7$:
\begin{equation}\label{Sect5_Velocity_13}
\begin{array}{ll}
\frac{2\gamma}{a}\int\limits_{\Omega}\sum\limits_{0\leq \mu\leq 2}
[\partial_t^{3-\mu}(\frac{p}{\varrho})\partial_t^{\mu}\nabla(\nabla\cdot v)]\cdot v_{ttt} \,\mathrm{d}x
\lem |\nabla(\nabla\cdot v)|_{\infty}\|(\frac{p}{\varrho})_{ttt}\|_{L^2(\Omega)}\|v_{ttt}\|_{L^2(\Omega)}
\\[6pt]\quad
+ \|v_{ttt}\|_{L^2(\Omega)}\|(\frac{p}{\varrho})_{tt}\|_{L^4(\Omega)}\|\nabla(\nabla\cdot v_t)\|_{L^4(\Omega)}
\\[6pt]\quad
+ |(\frac{p}{\varrho})_t|_{\infty}\|\nabla(\nabla\cdot v_{tt})\|_{L^2(\Omega)}\|v_{ttt}\|_{L^2(\Omega)}
\lem  \sqrt{\ve}\mathcal{F}[\xi,v](t).
\end{array}
\end{equation}

The fifth term of $I_7$:
\begin{equation}\label{Sect5_Velocity_14}
\begin{array}{ll}
\frac{2\gamma}{a}\int\limits_{\Omega} \frac{p}{\varrho}v_{ttt}\cdot\nabla(\nabla\cdot v_{ttt})\,\mathrm{d}x
= \frac{2\gamma}{a}\int\limits_{\partial\Omega} \frac{p}{\varrho}v_{ttt}\cdot n \nabla\cdot v_{ttt}\,\mathrm{d}S_x
- \frac{2\gamma}{a}\int\limits_{\Omega} \frac{p}{\varrho}|\nabla\cdot v_{ttt}|^2\,\mathrm{d}x  \\[10pt]\quad
- \frac{2\gamma}{a}\int\limits_{\Omega} (\nabla\cdot v_{ttt})v_{ttt}\cdot\nabla(\frac{p}{\varrho}) \,\mathrm{d}x
\\[10pt]
\leq C\sqrt{\ve}\mathcal{F}[\xi,v]
- (\frac{2\gamma}{a}-\frac{\gamma}{3a})\int\limits_{\Omega} \frac{p}{\varrho}|\nabla\cdot v_{ttt}|^2\,\mathrm{d}x
+C\int\limits_{\Omega} \frac{\varrho}{p}|\nabla(\frac{p}{\varrho})|_{\infty}|v_{ttt}|^2\,\mathrm{d}x
\\[10pt]
\leq C\sqrt{\ve}\mathcal{F}[\xi,v] - (\frac{2\gamma}{a}-\frac{\gamma}{3a})\int\limits_{\Omega} \frac{p}{\varrho}|\nabla\cdot v_{ttt}|^2\,\mathrm{d}x.
\end{array}
\end{equation}

After summing the five terms of $I_7$, namely  $(\ref{Sect5_Velocity_10})+(\ref{Sect5_Velocity_11})+(\ref{Sect5_Velocity_12})+(\ref{Sect5_Velocity_13})+
(\ref{Sect5_Velocity_14})$, we get
\begin{equation}\label{Sect5_Velocity_15}
\begin{array}{ll}
\frac{\mathrm{d}}{\mathrm{d}t} \int\limits_{\Omega}|v_{ttt}|^2 \,\mathrm{d}x
+ \frac{\gamma}{a}\int\limits_{\Omega} \frac{p}{\varrho}|\nabla\cdot v_{ttt}|^2\,\mathrm{d}x
\lem \sqrt{\ve}\mathcal{F}[\xi,v]
\end{array}
\end{equation}

By $(\ref{Sect5_Velocity_5})+(\ref{Sect5_Velocity_7})+(\ref{Sect5_Velocity_9})+(\ref{Sect5_Velocity_15})$, we get
\begin{equation}\label{Sect5_Velocity_16}
\begin{array}{ll}
\frac{\mathrm{d}}{\mathrm{d}t} E[v](t)
+ \frac{\gamma}{a}\int\limits_{\Omega} \frac{p}{\varrho}\sum\limits_{\ell=0}^3|\nabla\cdot \partial_t^{\ell}v|^2\,\mathrm{d}x
\lem \sqrt{\ve}\mathcal{F}[\xi,v].
\end{array}
\end{equation}
Thus, Lemma $\ref{Sect5_Velocity_Lemma}$ is proved.
\end{proof}

The following lemma concerns a priori estimate for $\int\limits_{\Omega}\sum\limits_{\ell=1}^{3}\partial_t^{\ell}\xi \partial_t^{\ell-1}\xi \,\mathrm{d}x$, which introduces $E[\xi](t)$ to the inequality $(\ref{Sect5_XiT_toProve})$.
\begin{lemma}\label{Sect5_Energy_Estimate_Lemma1}
For any given $T\in (0,+\infty]$, if
\begin{equation*}
\sup\limits_{0\leq t\leq T} \mathcal{F}[\xi,v,\phi](t) \leq\ve,
\end{equation*}
where $0<\ve\ll 1$, then there exists $c_{13}>0$ such that for $\forall t\in [0,T]$,
\begin{equation}\label{Sect5_XiT_toProve}
- \frac{\mathrm{d}}{\mathrm{d}t}\int\limits_{\Omega}\sum\limits_{\ell=1}^{3}
\partial_t^{\ell}\xi\partial_t^{\ell-1}\xi \,\mathrm{d}x
+ E[\xi](t) \leq C\sqrt{\ve}\mathcal{F}[\xi,v](t) +c_{13} E[v](t).
\end{equation}
\end{lemma}

\begin{proof}
Similar to $(\ref{Sect3_XiT_2})$, we have
\begin{equation}\label{Sect5_XiT_1}
\begin{array}{ll}
- \frac{\mathrm{d}}{\mathrm{d}t}\int\limits_{\Omega}\xi_{t}\xi \,\mathrm{d}x
+ \int\limits_{\Omega}(\xi_t)^2 \,\mathrm{d}x
\lem \sqrt{\ve}\mathcal{F}[\xi,v](t) + \|v_t\|_{L^2(\Omega)}^2 + \|\nabla\xi\|_{L^2(\Omega)}^2 \\[6pt]
\lem \sqrt{\ve}\mathcal{F}[\xi,v](t) + \|v\|_{L^2(\Omega)}^2+ \|v_t\|_{L^2(\Omega)}^2.
\end{array}
\end{equation}

Similar to $(\ref{Sect3_XiT_4})$, we have
\begin{equation}\label{Sect5_XiT_2}
\begin{array}{ll}
- \frac{\mathrm{d}}{\mathrm{d}t}\int\limits_{\Omega}\xi_{tt}\xi_t \,\mathrm{d}x
+ \int\limits_{\Omega}(\xi_{tt})^2 \,\mathrm{d}x
\lem \sqrt{\ve}\mathcal{F}[\xi,v](t) + \|v_{tt}\|_{L^2(\Omega)}^2 + \|\nabla\xi_t\|_{L^2(\Omega)}^2 \\[6pt]
\lem \sqrt{\ve}\mathcal{F}[\xi,v](t) + \|v_{tt}\|_{L^2(\Omega)}^2 + \|-ak_1\xi_t v -ak_1\varrho v_t\|_{L^2(\Omega)}^2 \\[6pt]

\lem \sqrt{\ve}\mathcal{F}[\xi,v](t) + \|v_t\|_{L^2(\Omega)}^2+ \|v_{tt}\|_{L^2(\Omega)}^2.
\end{array}
\end{equation}

Similar to $(\ref{Sect3_XiT_6})$, we have
\begin{equation}\label{Sect5_XiT_3}
\begin{array}{ll}
- \frac{\mathrm{d}}{\mathrm{d}t}\int\limits_{\Omega}\xi_{ttt}\xi_{tt} \,\mathrm{d}x
+ \int\limits_{\Omega}(\xi_{ttt})^2 \,\mathrm{d}x

\lem \sqrt{\ve}\mathcal{F}[\xi,v](t) + \|v_{ttt}\|_{L^2(\Omega)}^2 + \|\nabla\xi_{tt}\|_{L^2(\Omega)}^2 \\[6pt]
\lem \sqrt{\ve}\mathcal{F}[\xi,v](t) + \|v_{ttt}\|_{L^2(\Omega)}^2 + \|
-ak_1\xi_{tt} v - 2ak_1\xi_t v_t -ak_1\varrho v_{tt}\|_{L^2(\Omega)}^2 \\[6pt]

\lem \sqrt{\ve}\mathcal{F}[\xi,v](t) + \|v_{tt}\|_{L^2(\Omega)}^2+ \|v_{ttt}\|_{L^2(\Omega)}^2.
\end{array}
\end{equation}

By $(\ref{Sect5_XiT_1})+(\ref{Sect5_XiT_2})+(\ref{Sect5_XiT_3})$, we get
\begin{equation}\label{Sect5_XiT_4}
- \frac{\mathrm{d}}{\mathrm{d}t}\int\limits_{\Omega}\sum\limits_{\ell=1}^{3}
\partial_t^{\ell}\xi\partial_t^{\ell-1}\xi \,\mathrm{d}x
+ \int\limits_{\Omega}\sum\limits_{\ell=1}^{3}(\partial_t^{\ell}\xi)^2 \,\mathrm{d}x \lem \sqrt{\ve}\mathcal{F}[\xi,v](t)
+\sum\limits_{\ell=0}^{3}\|v_t^{\ell}\|_{L^2(\Omega)}^2.
\end{equation}

By Lemma $\ref{Sect2_P_Infty_Lemma}$, $\bar{p}\in[\inf\limits_{x\in\Omega}p,\sup\limits_{x\in\Omega}p]$, then for any $t\geq 0$, there exists $x_t \in\Omega$ such that $\xi(x_t,t)=0$. Assume $\ell(s)$ is a curve with finite length parameter $s$ such that $\ell(0)=x_t,\ \ell(s_x)=x$, then
\begin{equation}\label{Sect5_XiT_5}
\begin{array}{ll}
\|\xi(x,t)\|_{L^2(\Omega)}^2
= \|\xi(x_t,t) + \int\limits_0^{s_x} \nabla\xi[\ell(s)]\cdot \ell(s) \,\mathrm{d}s\|_{L^2(\Omega)}^2
\\[6pt]\hspace{2.15cm}
\leq C|Diam(\Omega)|^2\|\nabla\xi\|_{L^2(\Omega)}^2
\lem \|\nabla\xi\|_{L^2(\Omega)}^2 \\[6pt]\hspace{2.15cm}
= \|-ak_1 \varrho v\|_{L^2(\Omega)}^2
\lem \|v\|_{L^2(\Omega)}^2.
\end{array}
\end{equation}

Summing $(\ref{Sect5_XiT_4})$ and $(\ref{Sect5_XiT_5})$, we get
\begin{equation}\label{Sect5_XiT_6}
\begin{array}{ll}
- \frac{\mathrm{d}}{\mathrm{d}t}\int\limits_{\Omega}\sum\limits_{\ell=1}^{3}
\partial_t^{\ell}\xi\partial_t^{\ell-1}\xi \,\mathrm{d}x
+ \int\limits_{\Omega}\sum\limits_{\ell=0}^{3}(\partial_t^{\ell}\xi)^2 \,\mathrm{d}x
\lem \sqrt{\ve}\mathcal{F}[\xi,v](t)
+ \sum\limits_{\ell=0}^{3}\|v_t^{\ell}\|_{L^2(\Omega)}^2.
\end{array}
\end{equation}

Then there exist two constants $C>0$, $c_{13}>0$ such that
\begin{equation}\label{Sect5_XiT_7}
\begin{array}{ll}
- \frac{\mathrm{d}}{\mathrm{d}t}\int\limits_{\Omega}\sum\limits_{\ell=1}^{3}
\partial_t^{\ell}\xi\partial_t^{\ell-1}\xi \,\mathrm{d}x
+ \int\limits_{\Omega}\sum\limits_{\ell=0}^{3}(\partial_t^{\ell}\xi)^2 \,\mathrm{d}x \leq C\sqrt{\ve}\mathcal{F}[\xi,v](t) +c_{13}\sum\limits_{\ell=0}^{3}\|v_t^{\ell}\|_{L^2(\Omega)}^2.
\end{array}
\end{equation}
Thus, Lemma $\ref{Sect5_Energy_Estimate_Lemma1}$ is proved.
\end{proof}

In order to prove the uniform bound of $\int\limits_0^{\infty}\sum\limits_{0\leq\ell\leq 2,\ell+|\alpha|\leq 3}
\|\partial_t^{\ell}\mathcal{D}^{\alpha}\nabla\cdot v(s)\|_{L^2{(\Omega)}}^2 \,\mathrm{d}s$, we must have
the following lemma.
\begin{lemma}\label{Sect5_Velocity_Div_Lemma}
For any given $T\in (0,+\infty]$, if
\begin{equation*}
\sup\limits_{0\leq t\leq T} \mathcal{F}[\xi,v,\phi](t) \leq\ve,
\end{equation*}
where $0<\ve\ll 1$, then for $\forall t\in [0,T]$,
\begin{equation}\label{Sect5_Velocity_Div_toProve}
\begin{array}{ll}
\sum\limits_{0\leq\ell\leq 2,\ell+|\alpha|\leq 3}
\|\partial_t^{\ell}\mathcal{D}^{\alpha}\nabla\cdot v\|_{L^2{(\Omega)}}^2 \leq C\sqrt{\ve}\mathcal{F}[\xi,v](t)
+ c_{14}(E[\xi,v](t) + \mathcal{E}_1[\omega](t)).
\end{array}
\end{equation}
\end{lemma}

\begin{proof}
By $(\ref{Sect5_Velocity_Prove14})_2 + (\ref{Sect5_Velocity_Prove15})_2 + (\ref{Sect5_Velocity_Prove16})_2$, we get
\begin{equation}\label{Sect5_Velocity_Div_1}
\begin{array}{ll}
\|\nabla\cdot v\|_{H^3{(\Omega)}}^2 + \|\nabla\cdot v_t\|_{H^2{(\Omega)}}^2
+ \|\nabla\cdot v_{tt}\|_{H^1{(\Omega)}}^2 \\[6pt]

\lem \sqrt{\ve}\mathcal{F}[\xi,v](t)+ \|v_t\|_{H^2{(\Omega)}}^2
+ \|v_{tt}\|_{H^1{(\Omega)}}^2 + \|v_{ttt}\|_{L^2{(\Omega)}}^2 \\[6pt]

\leq C\sqrt{\ve}\mathcal{F}[\xi,v](t)+ C_9c_{10}(E[\xi,v](t) + \mathcal{E}_1[\omega](t)).
\end{array}
\end{equation}

Take $c_{14}=C_9c_{10}$. Thus, Lemma $\ref{Sect5_Velocity_Div_Lemma}$ is proved.
\end{proof}

The following lemma proves not only the exponential decay of $\mathcal{F}[\xi,v](t)$ and $\mathcal{E}[\omega](t)$, but also the uniform bound of $\int\limits_{0}^{T} \mathcal{E}[\nabla\cdot v](s)\,\mathrm{d}s$.
\begin{lemma}\label{Sect5_Decay_Lemma}
For any given $T\in (0,+\infty]$, there exists $\ve_2>0$ which is independent of $(\xi_0,v_0,\phi_0)$, such that if
\begin{equation*}
\sup\limits_{0\leq t\leq T} \mathcal{F}[\xi,v,\phi](t) \leq\ve,
\end{equation*}
where $\ve\ll \min\{1,\ve_0,\ve_2\}$, then for $\forall t\in [0,T]$,
\begin{equation}\label{Sect5_Decay_Estimates_toProve}
\begin{array}{ll}
\mathcal{F}[\xi,v](t)\leq \beta_6 \|\xi_0\|_{H^4(\Omega)}^2\exp\{-\beta_7 t\}, \\[6pt]
\mathcal{E}_1[\omega](t)\leq \beta_8 \|\xi_0\|_{H^4(\Omega)}^2\exp\{-\beta_7 t\}, \\[6pt]
\int\limits_{0}^{T} \mathcal{E}[\nabla\cdot v](s)\,\mathrm{d}s \leq \beta_9 \|\xi_0\|_{H^4(\Omega)}^2,
\end{array}
\end{equation}
where $\beta_6,\beta_7,\beta_8,\beta_9$ are four positive numbers.
\end{lemma}

\begin{proof}
In view of Lemmas $\ref{Sect5_Vorticity_Lemma}$, $\ref{Sect5_Energy_Estimate_Lemma2}$, $\ref{Sect5_Velocity_Lemma}$,
$\ref{Sect5_Energy_Estimate_Lemma1}$ and $\ref{Sect5_Velocity_Div_Lemma}$, we have obtained global a priori estimates as follows:
\begin{equation}\label{Sect5_Global_A_Priori_Estimates}
\left\{\begin{array}{ll}
\frac{\mathrm{d}}{\mathrm{d}t}\int\limits_{\Omega}\mathcal{E}_1[\omega](t)\,\mathrm{d}x + 2a\mathcal{E}_1[\omega](t)
\lem C\sqrt{\ve}\mathcal{F}[\xi,v](t), \\[10pt]

\frac{\mathrm{d}}{\mathrm{d}t}E_1[\xi](t)+ 2a E[v](t) \leq C\sqrt{\ve}\mathcal{F}[\xi,v](t), \\[6pt]

\frac{\mathrm{d}}{\mathrm{d}t} E[v](t)
+ \frac{\gamma}{a}\int\limits_{\Omega} \frac{p}{\varrho}\sum\limits_{\ell=0}^3|\nabla\cdot \partial_t^{\ell}v|^2\,\mathrm{d}x \lem \sqrt{\ve}\mathcal{F}[\xi,v](t), \\[7pt]

- \frac{\mathrm{d}}{\mathrm{d}t}\int\limits_{\Omega}\sum\limits_{\ell=1}^{3}
\partial_t^{\ell}\xi\partial_t^{\ell-1}\xi \,\mathrm{d}x
+ E[\xi](t) \leq C\sqrt{\ve}\mathcal{F}[\xi,v](t) +c_{13} E[v](t), \\[10pt]

\sum\limits_{0\leq\ell\leq 2,\ell+|\alpha|\leq 3}
\|\partial_t^{\ell}\mathcal{D}^{\alpha}\nabla\cdot v\|_{L^2{(\Omega)}}^2 \leq C\sqrt{\ve}\mathcal{F}[\xi,v](t)
+ c_{14}(E[\xi,v](t) + \mathcal{E}_1[\omega](t)).
\end{array}\right.
\end{equation}

Let $\lambda_2=\max\{\frac{4}{3},\frac{c_{13}}{a}\}+1, \
\lambda_3 = \min\{\frac{1}{2c_{14}},\frac{a}{c_{14}}\}$. Define
\begin{equation}\label{Sect5_Define_E2}
\begin{array}{ll}
E_2[\xi](t) := \lambda_2 E_1[\xi](t)
-\sum\limits_{\ell=1}^{3}\int\limits_{\Omega}\partial_t^{\ell-1}\xi \partial_t^{\ell} \xi\,\mathrm{d}x.
\end{array}
\end{equation}
where $E_3>0$ by Cauchy-Schwarz inequality. Since $\lambda_2>\frac{4}{3}$, $E_2[\xi](t)\cong E[\xi](t)$, i.e., there exists $c_{15}>0$, $c_{16}>0$ such that
\begin{equation}\label{Sect5_Equivalence}
\begin{array}{ll}
c_{15} E[\xi](t) \leq E_2[\xi](t) \leq c_{16} E[\xi](t).
\end{array}
\end{equation}

By $(\ref{Sect5_Global_A_Priori_Estimates})_1+ (\ref{Sect5_Global_A_Priori_Estimates})_2\times\lambda_2
+(\ref{Sect5_Global_A_Priori_Estimates})_3 + (\ref{Sect5_Global_A_Priori_Estimates})_4 + (\ref{Sect5_Global_A_Priori_Estimates})_5\times\lambda_3$, we get
\begin{equation}\label{Sect5_Decay_1}
\begin{array}{ll}
\frac{\mathrm{d}}{\mathrm{d}t}\left(E_2[\xi](t) + E[v](t) +\mathcal{E}_1[\omega](t)\right)
+ (1-\lambda_3c_{14})E[\xi](t)+ (2a -\lambda_3c_{14})\mathcal{E}_1[\omega](t)
\\[10pt]\quad
+ (2a\lambda_2 - c_{13}-\lambda_3c_{14}) E[v](t)
+ \frac{\gamma}{a}\int\limits_{\Omega} \frac{p}{\varrho}\sum\limits_{\ell=0}^3|\nabla\cdot \partial_t^{\ell}v|^2\,\mathrm{d}x
\\[10pt]\quad
+ \lambda_3\sum\limits_{0\leq\ell\leq 2,\ell+|\alpha|\leq 3}
\|\partial_t^{\ell}\mathcal{D}^{\alpha}\nabla\cdot v\|_{L^2{(\Omega)}}^2

\leq C_{10}\sqrt{\ve}\mathcal{F}[\xi,v](t),
\end{array}
\end{equation}
for some $C_{10}>0$.

By Lemma $\ref{Sect5_Total_Energy_Lemma}$, we have
\begin{equation}\label{Sect5_Decay_2}
\begin{array}{ll}
\sqrt{\ve}\mathcal{F}[\xi,v](t) \leq c_{10}\sqrt{\ve}(E[\xi,v](t) + \mathcal{E}_1[\omega](t)).
\end{array}
\end{equation}

Let $\ve_2 =\min\{\frac{(1-\lambda_3c_{14})^2}{4C_{10}^2c_0^2}, \frac{(2a\lambda_2-c_{13}-\lambda_3c_{13})^2}{4C_{10}^2c_0^2},
\frac{(2a-\lambda_3c_{14})^2}{4C_{10}^2c_0^2}\}$,
plug $(\ref{Sect5_Decay_2})$ into $(\ref{Sect5_Decay_1})$, we get
\begin{equation}\label{Sect5_Decay_3}
\begin{array}{ll}
\frac{\mathrm{d}}{\mathrm{d}t}\left(E_2[\xi](t) + E[v](t) +\mathcal{E}_1[\omega](t)\right)
+ \frac{1-\lambda_3c_{14}}{2}E[\xi](t)+ \frac{2a -\lambda_3c_{14}}{2}\mathcal{E}_1[\omega](t)
\\[10pt]\quad
+ \frac{2a\lambda_2 - c_{13}-\lambda_3c_{14}}{2} E[v](t)
+ \frac{\gamma}{a}\int\limits_{\Omega} \frac{p}{\varrho}\sum\limits_{\ell=0}^3|\nabla\cdot \partial_t^{\ell}v|^2\,\mathrm{d}x
\\[10pt]\quad
+ \lambda_3\sum\limits_{0\leq\ell\leq 2,\ell+|\alpha|\leq 3}
\|\partial_t^{\ell}\mathcal{D}^{\alpha}\nabla\cdot v\|_{L^2{(\Omega)}}^2
\leq 0.
\end{array}
\end{equation}

Since the last two terms in $(\ref{Sect5_Decay_3})$ are positive, we have
\begin{equation}\label{Sect5_Decay_4}
\begin{array}{ll}
\frac{\mathrm{d}}{\mathrm{d}t}\left(E_2[\xi](t) + E[v](t) +\mathcal{E}_1[\omega](t)\right)
+ \frac{1-\lambda_3c_{14}}{2c_{16}}E_2[\xi](t)+ \frac{2a -\lambda_3c_{14}}{2}\mathcal{E}_1[\omega](t)
\\[10pt]\qquad
+ \frac{2a\lambda_2 - c_{13}-\lambda_3c_{14}}{2} E[v](t) \leq 0.
\end{array}
\end{equation}

Let $c_{17} = \min\{\frac{1-\lambda_3c_{14}}{2}, \frac{2a -\lambda_3c_{14}}{2},
\frac{2a\lambda_2 - c_{13}-\lambda_3c_{14}}{2}\}$, it follows from $(\ref{Sect3_Decay_2})$ that
\begin{equation}\label{Sect5_Decay_5}
\begin{array}{ll}
\frac{\mathrm{d}}{\mathrm{d}t}\left(E_2[\xi](t) + E[v](t) +\mathcal{E}_1[\omega](t)\right)
+ c_{17}(E_2[\xi](t)+ \mathcal{E}_1[\omega](t)+ E[v](t)) \leq 0.
\end{array}
\end{equation}

After integrating $(\ref{Sect5_Decay_5})$, we get
\begin{equation}\label{Sect5_Decay_6}
\begin{array}{ll}
E_2[\xi](t) + E[v](t) +\mathcal{E}_1[\omega](t)\leq (E_2[\xi](0) + E[v](0) +\mathcal{E}_1[\omega](0))
\exp\{-c_{17} t\}, \\[10pt]
\mathcal{E}_1[\omega](t)\leq (E_2[\xi](0) + E[v](0) +\mathcal{E}_1[\omega](0))
\exp\{-c_{17} t\} \\[6pt]\hspace{1.2cm}
\leq (c_{16} E[\xi](0)+ E[v](0) + \mathcal{E}[v](0))\exp\{-c_{17} t\} \\[6pt]\hspace{1.2cm}
\leq (c_{16} +1)\mathcal{F}[\xi,v](0)\exp\{-c_{17} t\} \\[6pt]\hspace{1.2cm}
\leq C_{11}(c_{16} +1)\|\xi_0\|_{H^4(\Omega)}^2\exp\{-c_{17} t\}. \\[12pt]

c_{15} E[\xi](t)\leq E_2[\xi](t)\leq (E_2[\xi](0)+E[v](0) + \mathcal{E}_1[\omega](0))
\exp\{-c_{17} t\}, \\[10pt]

\mathcal{F}[\xi,v](t)\leq c_{10}(E[\xi,v](t) + \mathcal{E}_1[\omega](t)) \\[6pt]\hspace{1.4cm}
\leq (\frac{c_{10}}{c_{15}}+ 2c_{10})(E_2[\xi](0)+E[v](0) + \mathcal{E}_1[\omega](0))
\exp\{-c_{17} t\} \\[6pt]\hspace{1.4cm}
\leq (\frac{c_{10}}{c_{15}}+ 2c_{10})(c_{16} +1)\mathcal{F}[\xi,v](0)\exp\{-c_{17} t\}
\\[6pt]\hspace{1.4cm}
\leq C_{11}(\frac{c_{10}}{c_{15}}+ 2c_{10})(c_{16}+1)\|\xi_0\|_{H^4(\Omega)}^2\exp\{-c_{17} t\}
\end{array}
\end{equation}

Take $\beta_6=C_{11}(\frac{c_{10}}{c_{15}}+ 2c_{10})(c_{16} +1)$, $\beta_7=c_{17}$, $\beta_8=C_{11}(c_{16}+1)$, the exponential decay in $(\ref{Sect5_Decay_Estimates_toProve})$ is obtained.

\vspace{0.2cm}
It follows from $(\ref{Sect5_Decay_3})$ that
\begin{equation}\label{Sect5_Decay_7}
\begin{array}{ll}
\frac{\mathrm{d}}{\mathrm{d}t}\left(E_2[\xi](t) + E[v](t) +\mathcal{E}_1[\omega](t)\right)
+ \frac{\gamma}{a}\int\limits_{\Omega} \frac{p}{\varrho}\sum\limits_{\ell=0}^3|\nabla\cdot \partial_t^{\ell}v|^2\,\mathrm{d}x \\[13pt]\quad
+ \lambda_3\sum\limits_{0\leq\ell\leq 2,\ell+|\alpha|\leq 3}
\|\partial_t^{\ell}\mathcal{D}^{\alpha}\nabla\cdot v\|_{L^2{(\Omega)}}^2
\leq 0.
\end{array}
\end{equation}

Integrate $(\ref{Sect5_Decay_7})$ from $t=0$ to $t=T$, we get
\begin{equation}\label{Sect5_Decay_8}
\begin{array}{ll}
E_2[\xi](T) + E[v](T) +\mathcal{E}_1[\omega](T)
+ \lambda_3\int\limits_{0}^{T}\sum\limits_{0\leq\ell\leq 2,\ell+|\alpha|\leq 3}
\|\partial_{\tau}^{\ell}\mathcal{D}^{\alpha}\nabla\cdot v(s)\|_{L^2{(\Omega)}}^2 \mathrm{d}s \\[6pt]\quad
+ \frac{\gamma}{a}\int\limits_{0}^{T}\int\limits_{\Omega} \frac{p}{\varrho}\sum\limits_{\ell=0}^3|\nabla\cdot \partial_{\tau}^{\ell}v(s)|^2\,\mathrm{d}x\mathrm{d}s
\leq E_2[\xi](0) + E[v](0) +\mathcal{E}_1[\omega](0).
\end{array}
\end{equation}

Then
\begin{equation}\label{Sect5_Decay_9}
\begin{array}{ll}
\int\limits_{0}^{T}\mathcal{E}[\nabla\cdot v](s)\,\mathrm{d}s \leq
C_{12}\lambda_3\int\limits_{0}^{T}\sum\limits_{0\leq\ell\leq 2,\ell+|\alpha|\leq 3}
\|\partial_{\tau}^{\ell}\mathcal{D}^{\alpha}\nabla\cdot v(s)\|_{L^2{(\Omega)}}^2\mathrm{d}s \\[10pt]\quad
+ C_{12}\frac{\gamma}{a}\int\limits_{0}^{T}\int\limits_{\Omega} \frac{p}{\varrho}\sum\limits_{\ell=0}^3|\nabla\cdot \partial_{\tau}^{\ell}v(s)|^2\,\mathrm{d}x\mathrm{d}s
\leq C_{12}(E_2[\xi](0) + E[v](0) +\mathcal{E}_1[\omega](0)) \\[10pt]
\leq C_{12}(c_{16}+1)\mathcal{F}[\xi,v](0)
\leq C_{13}C_{12}(c_{16}+1)\|\xi_0\|_{H^4(\Omega)}^2,
\end{array}
\end{equation}
where $C_{12} = \min\{\frac{1}{\lambda_3},\frac{9a\hat{\bar{\varrho}}}{4\gamma\hat{\bar{p}}}\}$.

Take $\beta_9=C_{13}C_{12}(c_{16}+1)$. Thus, Lemma $\ref{Sect5_Decay_Lemma}$ is proved.
\end{proof}

The following lemma concerns the uniform bound of $\mathcal{E}[\phi](t)$ on the condition that $v$ decays exponentially.
\begin{lemma}\label{Sect5_Entropy_Lemma}
For any given $T\in (0,+\infty]$, if
\begin{equation*}
\sup\limits_{0\leq t\leq T} \mathcal{F}[\xi,v,\phi](t) \leq\ve,
\end{equation*}
where $0<\ve\ll \min\{1,\ve_0,\ve_1,\ve_2\}$, then for $\forall t\in [0,T]$,
\begin{equation}\label{Sect5_Entropy_toProve_1}
\frac{\mathrm{d}}{\mathrm{d}t}\mathcal{E}[\phi](t)\leq \beta_4\mathcal{E}[v](t)^{\frac{1}{2}}\mathcal{E}[\phi](t).
\end{equation}

If $\mathcal{E}[v](t)\leq \beta_6 \|\xi_0\|_{H^4(\Omega)}^2\exp\{-\beta_7 t\}$, then $\mathcal{E}[\phi](t)$ has uniform bound:
\begin{equation}\label{Sect5_Entropy_toProve_2}
\mathcal{E}[\phi](t) \leq \beta_{10}\|\phi_0\|_{H^3(\Omega)}^2\left(\exp\{\|\xi_0\|_{H^4(\Omega)}\}\right)^{c_{18}},
\end{equation}
for some $c_{18}>0$.
\end{lemma}

\begin{proof}
Similar to Lemma $\ref{Sect3_Entropy_Lemma}$, we have the following a priori estimate: 
\begin{equation}\label{Sect5_Entropy_1}
\begin{array}{ll}
\frac{\mathrm{d}}{\mathrm{d}t}\mathcal{E}[\phi](t)\leq \beta_4\mathcal{E}[v](t)^{\frac{1}{2}}\mathcal{E}[\phi](t), \\[6pt]
\mathcal{E}[\phi](t) \leq \mathcal{E}[\phi](0)
\exp\{\int\limits_{0}^{t}\beta_4\mathcal{E}[v](\tau)^{\frac{1}{2}}\,\mathrm{d}\tau\}. \\[6pt]
\end{array}
\end{equation}

If $\mathcal{E}[v](t)\leq \beta_6 \|\xi_0\|_{H^4(\Omega)}^2\exp\{-\beta_7 t\}$, then
\begin{equation*}
\begin{array}{ll}
\mathcal{E}[\phi](t) \leq \mathcal{E}[\phi](0)
\exp\{\int\limits_{0}^{t}\beta_4\mathcal{E}[v](\tau)^{\frac{1}{2}}\,\mathrm{d}\tau\} \\[6pt]\hspace{1.1cm}
\leq \beta_{10}\|\phi_0\|_{H^3(\Omega)}^2
\exp\{\int\limits_{0}^{t}\beta_4\sqrt{\beta_6}\|\xi_0\|_{H^4(\Omega)} \exp\{-\beta_7 \tau\}^{\frac{1}{2}}\,\mathrm{d}\tau\}
\end{array}
\end{equation*}

\begin{equation}\label{Sect5_Entropy_2}
\begin{array}{ll}
\hspace{1.1cm}
\leq \beta_{10}\|\phi_0\|_{H^3(\Omega)}^2
\exp\{\frac{2\beta_4\sqrt{\beta_6}}{\beta_7}\|\xi_0\|_{H^4(\Omega)}(1-\exp\{-\frac{\beta_7}{2}t\})\}
\\[6pt]\hspace{1.1cm}

\leq \beta_{10}\|\phi_0\|_{H^3(\Omega)}^2
\exp\{\frac{2\beta_4\sqrt{\beta_6}}{\beta_7}\|\xi_0\|_{H^4(\Omega)}\} \\[6pt]\hspace{1.1cm}

= \beta_{10}\|\phi_0\|_{H^3(\Omega)}^2
\left(\exp\{\|\xi_0\|_{H^4(\Omega)}\}\right)^{c_{18}},
\end{array}
\end{equation}
where $c_{18}=\frac{2\beta_4\sqrt{\beta_6}}{\beta_7}$, $\beta_{10}>0$.

Therefore $\mathcal{E}[\phi](t)$ is uniformly bounded when $\mathcal{E}[v](t)$ decays exponentially.
Thus, Lemma $\ref{Sect5_Entropy_Lemma}$ is proved.
\end{proof}

The following lemma concerns the exponential decay of $\sum\limits_{\ell=1}^{3}\|\partial_t^{\ell}S\|_{H^{3-\ell}(\Omega)}^2$ on the condition that $v$ decays exponentially.
\begin{lemma}\label{Sect5_Entropy_Decay_Lemma}
For any given $T\in (0,+\infty]$, if
\begin{equation*}
\sup\limits_{0\leq t\leq T} \mathcal{E}[\xi,v,\phi](t) \leq\ve,
\end{equation*}
where $0<\ve\ll \min\{1,\ve_0,\ve_1,\ve_2\}$, then for $\forall t\in [0,T]$,
\begin{equation}\label{Sect5_Entropy_T_toProve}
\begin{array}{ll}
\sum\limits_{\ell=1}^{3}\|\partial_t^{\ell}\phi\|_{H^{3-\ell}(\Omega)}^2
\leq c_{19}\|\xi_0\|_{H^4(\Omega)}^2\|\phi_0\|_{H^3(\Omega)}^2
\left(\exp\{\|\xi_0\|_{H^4(\Omega)}\}\right)^{c_{18}}\exp\{-\beta_7 t\},
\end{array}
\end{equation}
for some $c_{19}>0$.
\end{lemma}

\begin{proof}
It follows from Lemma $\ref{Sect5_Decay_Lemma}$ that
$\mathcal{E}[v](t)\leq \beta_6 \|\xi_0\|_{H^4(\Omega)}^2\exp\{-\beta_7 t\}$.
It follows from Lemma $\ref{Sect5_Entropy_Lemma}$ that
$\mathcal{E}[\phi](t)\leq \beta_{10}\|\phi_0\|_{H^3(\Omega)}^2\left(\exp\{\|\xi_0\|_{H^4(\Omega)}\}\right)^{c_{18}}$.

Similar to Lemma $\ref{Sect3_Entropy_Decay_Lemma}$, we have the following a priori estimate:
\begin{equation}\label{Sect5_Entropy_T_Prove4}
\begin{array}{ll}
\sum\limits_{\ell=1}^{3}\|\partial_t^{\ell}S\|_{H^{3-\ell}(\Omega)}^2
= \sum\limits_{\ell=1}^{3}\|\partial_t^{\ell}\phi\|_{H^{3-\ell}(\Omega)}^2
\lem \mathcal{E}[v](t)\mathcal{E}[\phi](t) \\[6pt]\hspace{2.7cm}
\leq \beta_6\beta_{10}\|\xi_0\|_{H^4(\Omega)}^2\|\phi_0\|_{H^3(\Omega)}^2
\left(\exp\{\|\xi_0\|_{H^4(\Omega)}\}\right)^{c_{18}}\exp\{-\beta_7 t\}.
\end{array}
\end{equation}

Take $c_{19} = \beta_6\beta_{10}$. Thus, Lemma $\ref{Sect5_Entropy_Decay_Lemma}$ is proved.
\end{proof}

\begin{remark}\label{Sect5_Varrho_Remark}
Similar to the results in Lemma $\ref{Sect3_Varrho_Remark}$, we have a priori estimates for $\mathcal{E}[\varrho-\bar{\varrho}](t)$ and $\sum\limits_{\ell=1}^{3}\|\partial_t^{\ell}\varrho\|_{H^{3-\ell}(\Omega)}^2$:

When $\mathcal{F}[p-\bar{p}](t)$ and $\mathcal{E}[S-\bar{S}](t)$ are uniformly bounded, $\mathcal{E}[\varrho-\bar{\varrho}](t)$ is also uniformly bounded due to $\varrho = \frac{1}{\sqrt[\gamma]{A}}p^{\frac{1}{\gamma}}\exp\{-\frac{S}{\gamma}\}$. For any given $T\in (0,+\infty]$, if $\sup\limits_{0\leq t\leq T} \mathcal{F}[\xi,v,\phi](t) \leq\ve$,
where $0<\ve\ll \min\{1,\ve_0,\ve_1,\ve_2\}$, then $\sum\limits_{\ell=1}^{3}\|\partial_t^{\ell}\varrho\|_{H^{3-\ell}(\Omega)}^2$ also decays at an exponential rate of $C\exp\{-\beta_7 t\}$.
\end{remark}

\section{Darcy's Law and Nonlinear Diffusion of Non-Isentropic Euler Equations with Damping}
In this section, we prove the global existence of classical solutions to the diffusion equations $(\ref{Sect2_Final_Diffusion})$ under small data assumption and the nonlinear diffusion property of the non-isentropic Euler equations with damping $(\ref{Sect2_Final_Eq})$ when the time is large. For simplicity, we omit the symbol $\ \hat{}\ $ over the variables and constants in this section if there is no ambiguity, otherwise we will add the symbol $\ \hat{}\ $.

The proof of the local existence of classical solutions to IBVP for the parabolic-hyperbolic equations  $(\ref{Sect2_Parabolic_Hyperbolic})$ is standard, such as using the linearization-iteration-convergence scheme, so we give a lemma on the local existence without proof here.

\begin{lemma}\label{Sect6_LocalExistence}
$(Local\ Existence)$\\[6pt]
If $(\xi_0,\phi_0)\in H^4(\Omega)\times H^3(\Omega)$, $\inf\limits_{x\in\Omega}p_0(x)>0$ and $\partial_t^{\ell} \nabla\xi(x,0)\cdot n|_{\partial\Omega}=0$, $0\leq \ell\leq 3$, then there exists a finite time $T_{\ast}>0$, such that IBVP $(\ref{Sect2_Parabolic_Hyperbolic})$ admits a unique local classical solution $(\xi,\phi)$ satisfying
\begin{equation}\label{Sect6_Local_Regularity}
\left\{\begin{array}{ll}
(\xi,\phi)\in \underset{0\leq \ell\leq 3}{\cap}C^{\ell}([0,T_{\ast}),H^{4-\ell}(\Omega)\times H^{3-\ell}(\Omega)),
\\[6pt]
\triangle\xi \in C(\Omega\times[0,T_{\ast})).
\end{array}\right.
\end{equation}
\end{lemma}

The above lemma implies the local existence of classical solutions to IBVP $(\ref{Sect2_Final_Diffusion})$ as long as $(\xi,\phi)$ remain classical, namely, $(\xi,\phi)\in C^1([0,T^{\ast}), C^2(\Omega)\times C^1(\Omega))$. Based on the global a priori estimates for $(\xi,v,\phi)$, we obtained the global existence of classical solutions to IBVP $(\ref{Sect2_Final_Diffusion})$.

\begin{theorem}\label{Sect6_GlobalExistence_Thm}
$(Global\ Existence)$\\[6pt]
Assume $(\xi_0,\phi_0)\in H^4(\Omega)\times H^3(\Omega)$, $\inf\limits_{x\in\Omega}p_0(x)>0$ and $\partial_t^{\ell} \nabla\xi(x,0)\cdot n|_{\partial\Omega}=0$, $0\leq \ell\leq 3$.
There exists a sufficiently small number $\delta_2>0$, such that if $\|\xi_0\|_{H^4(\Omega)}+\|\phi_0\|_{H^3(\Omega)}\leq \delta_2$, then IBVP $(\ref{Sect2_Final_Diffusion})$ admits a unique global classical solution $(\xi,\phi)$ satisfying
\begin{equation}\label{Sect6_Global_Regularity}
\begin{array}{ll}
(\xi,\phi)\in \underset{0\leq \ell\leq 3}{\cap}C^{\ell}([0,+\infty),H^{4-\ell}(\Omega)\times H^{3-\ell}(\Omega)),\
\triangle\xi \in C(\Omega\times[0,+\infty)),
\end{array}
\end{equation}
moreover,
\begin{equation}\label{Sect6_Global_Regularity_Add}
\left\{\begin{array}{ll}
\varrho = \varrho(\xi,\phi)\in \underset{0\leq \ell\leq 3}{\cap}C^{\ell}([0,+\infty),H^{3-\ell}(\Omega)),\\[6pt]
v = -\frac{1}{ak_1\varrho}\nabla\xi \in \underset{0\leq \ell\leq 3}{\cap}C^{\ell}([0,+\infty),H^{3-\ell}(\Omega)),\
\nabla\cdot v \in C(\Omega\times[0,+\infty)).
\end{array}\right.
\end{equation}

$\forall t\geq 0$, $\mathcal{F}[\xi,v](t)$, $\mathcal{E}_1[\omega](t)$ and $\sum\limits_{\ell=1}^{3}\|\partial_t^{\ell}\phi\|_{H^{3-\ell}(\Omega)}^2$ decays exponentially,
$\mathcal{F}[\phi](t)$ is uniformly bounded.
\end{theorem}

\begin{proof}
In view of Lemmas $\ref{Sect5_Decay_Lemma}$ and $\ref{Sect5_Entropy_Lemma}$, we have the following global a priori estimates: for any given $T\in (0,+\infty]$, if
\begin{equation}\label{Sect6_Assumption_1}
\sup\limits_{0\leq t\leq T} \mathcal{F}[\xi,v,\phi](t) \leq\ve,
\end{equation}
where $0<\e\ll \min\{1,\ve_0,\ve_1,\ve_2\}$, then
\begin{equation}\label{Sect6_Decay}
\begin{array}{ll}
\mathcal{F}[\xi,v](t)\leq \beta_6 \|\xi_0\|_{H^4(\Omega)}^2\exp\{-\beta_7 t\}, \\[6pt]
\mathcal{E}[\phi](t) \leq \beta_{10}\|\phi_0\|_{H^3(\Omega)}^2\left(\exp\{\|\xi_0\|_{H^4(\Omega)}\}\right)^{c_{18}}.
\end{array}
\end{equation}

The constants $\ve_0,\ve_1,\ve_2$ are independent of $(\xi_0,\phi_0)$, so we can choose $\ve$ which is independent of $(\xi_0,\phi_0)$.

Take $\delta_2=\min\{\sqrt{\ve},\sqrt{\frac{\ve}{2\beta_6}},\sqrt{\frac{\ve}{2\beta_{10}}}
\left(\exp\{\sqrt{\frac{\ve}{2\beta_6}}\}\right)^{-\frac{c_{18}}{2}}\}$, then if $\|\xi_0\|_{H^4(\Omega)}+\|\phi_0\|_{H^3(\Omega)}\leq\delta_2$, we have
\begin{equation}\label{Sect6_Data_Condition}
\left\{\begin{array}{ll}
\|\xi_0\|_{H^4(\Omega)} \leq \sqrt{\frac{\ve}{2\beta_6}}, \\[6pt]
\|\phi_0\|_{H^3(\Omega)} \leq \sqrt{\frac{\ve}{2\beta_{10}}}\left(\exp\{\sqrt{\frac{\ve}{2\beta_6}}\}\right)^
{-\frac{c_{18}}{2}}.
\end{array}\right.
\end{equation}

Due to the estimates in $(\ref{Sect6_Decay})$, $(\xi,v,\phi)$ satisfy
\begin{equation}\label{Sect6_Solution_Condition}
\begin{array}{ll}
\mathcal{F}[\xi,v](t)\leq \frac{\ve}{2},\quad \mathcal{E}[\phi](t)\leq \frac{\ve}{2},\quad \forall t\in [0,T].
\end{array}
\end{equation}
This implies the a priori assumption $(\ref{Sect6_Assumption_1})$ is satisfied, the validity of the former a priori estimates is verified.

By Lemma $\ref{Sect5_Decay_Lemma}$, we have
\begin{equation}\label{Sect6_Aubin_Lions_1}
\begin{array}{ll}
\int\limits_{0}^{T} \mathcal{E}[\nabla\cdot v](s)\,\mathrm{d}s \leq \beta_9 \|\xi_0\|_{H^4(\Omega)}^2,
\end{array}
\end{equation}
which implies that for any given time $T\in (0,+\infty]$,
\begin{equation}\label{Sect6_Aubin_Lions_2}
\left\{\begin{array}{ll}
\|\nabla\cdot v\|_{L^2([0,T],H^3(\Omega))}^2 \lem \int\limits_{0}^{T} \mathcal{E}[\nabla\cdot v](s)\,\mathrm{d}s \lem \|\xi_0\|_{H^4(\Omega)}^2, \\[6pt]
\|\nabla\cdot v_t\|_{L^2([0,T],H^1(\Omega))}^2 \lem \int\limits_{0}^{T} \mathcal{E}[\nabla\cdot v](s)\,\mathrm{d}s \lem \|\xi_0\|_{H^4(\Omega)}^2.
\end{array}\right.
\end{equation}

By Aubin-Lions' Lemma, we obtain
\begin{equation}\label{Sect6_Aubin_Lions_3}
\begin{array}{ll}
\|\nabla\cdot v\|_{C([0,T],H^2(\Omega))}^2 \lem \|\xi_0\|_{H^4(\Omega)}^2,
\end{array}
\end{equation}
which implies that $\nabla\cdot v\in C(\Omega\times [0,T])$ for any $T>0$. Then
\begin{equation}\label{Sect6_Aubin_Lions_4}
\begin{array}{ll}
\triangle\xi = -a k_1\varrho\nabla\cdot v - a k_1 v\cdot\nabla\varrho \in C(\Omega\times [0,T]).
\end{array}
\end{equation}

Due to the global a priori estimates for $(\xi,v,\phi)$ and Lemma $\ref{Sect6_LocalExistence}$ on the local existence result, the classical solution $(\xi,\phi)$ can be extended to $[0,+\infty)$. $(\ref{Sect6_Aubin_Lions_4})$ holds for any given $T\in (0,+\infty]$. Thus, Theorem $\ref{Sect6_GlobalExistence_Thm}$ on the global existence of classical solutions to IBVP $(\ref{Sect2_Final_Diffusion})$ is proved.
\end{proof}

\begin{remark}\label{Sect6_Friction_Coefficient}
Our proof requires $a\geq C\sqrt{\ve}$ where $C>0$ is large enough. If $a\rto 0$,  $(p_0,\nabla p_0) \rto (\bar{p},0)$ is required.
\end{remark}

\begin{remark}\label{Sect6_DarcyLaw}
Theorem $\ref{Sect6_GlobalExistence_Thm}$ implies the global well-posedness of the diffusion equations $(\ref{Sect1_Diffusion_Eq})$ under small data assumption, thus Darcy's law is verified when the ideal gases are sufficiently mild and slow.
While the verification of Darcy's law for 1D non-isentropic p-system with damping see \cite{Hsiao_Pan_1999},
for isentropic Euler equations with damping see \cite{Pan_Zhao_2009},
for isothermal Euler equations with damping see \cite{Zhao_2010}.

\end{remark}

Since $(\xi,\phi)\in C^1(\Omega\times[0,+\infty))$ is the global classical solution to IBVP $(\ref{Sect2_Final_Diffusion})$, then $(p=\bar{p}+\xi,S=\bar{S}+\phi)$ is the global classical solution to IBVP for the diffusion equations $(\ref{Sect1_Diffusion_Eq})$.
The following theorem describes the asymptotical behavior of $(p,v,S,\varrho)$ relating to their equilibrium states $(p_{\infty},v_{\infty},S_{\infty},\varrho_{\infty})$.

\begin{theorem}\label{Sect6_Convergence_Rate_Thm}
Assume the conditions in Theorem $\ref{Sect6_GlobalExistence_Thm}$ hold. Let  $(p,S)$ be the global classical solution to IBVP $(\ref{Sect1_Diffusion_Eq})$.
$p_{\infty}=\bar{p}$, $u_{\infty}=v_{\infty}=\omega_{\infty}=0$. If $S_0\neq const$, then $S_{\infty}\neq const,\
\varrho_{\infty}(x)\neq const,\ \theta_{\infty}\neq const,\ e_{\infty}\neq const$.
As $t\rto +\infty$, $(p,u,S,\varrho)$ converge to $(\bar{p},0,S_{\infty},\varrho_{\infty})$ exponentially in $|\cdot|_{\infty}$ norm.
\end{theorem}

\begin{proof}
By Lemma $\ref{Sect5_Entropy_Decay_Lemma}$, we have
\begin{equation}\label{Sect6_S_Infty_Decay}
\begin{array}{ll}
|S_t|_{\infty}\lem
\left( \sum\limits_{\ell=1}^{3}\|\partial_t^{\ell}\phi\|_{H^{3-\ell}(\Omega)}^2 \right)^{\frac{1}{2}}
\\[6pt]\hspace{0.9cm}
\lem \|\xi_0\|_{H^4(\Omega)}\|\phi_0\|_{H^3(\Omega)}
\left(\exp\{\|\xi_0\|_{H^4(\Omega)}\}\right)^{\frac{c_{18}}{2}}\exp\{-\frac{\beta_7}{2} t\},
\end{array}
\end{equation}

So $\int\limits_{0}^{\infty}S_\tau(x,\tau) \,\mathrm{d}\tau$ converges, then
$S_{\infty}(x)=S_0(x) + \int\limits_{0}^{\infty}S_\tau(x,\tau) \,\mathrm{d}\tau$ is bounded.

Similar to the proof of Theorem $\ref{Sect4_Convergence_Rate_Thm}$, we can show $p_{\infty}=\bar{p}$, $u_{\infty}=v_{\infty}=\omega_{\infty}=0$. If $S_0\neq const$, then $S_{\infty}\neq const,\
\varrho_{\infty}(x)\neq const,\ \theta_{\infty}\neq const,\ e_{\infty}\neq const$.

The exponential decay rates of $(\xi,v,\phi_{t})$ provides exponential convergence rates of $(p,u,S,\varrho)$ to their equilibrium states as follows:
\begin{equation}\label{Sect6_Convergence_Rate}
\left\{\begin{array}{ll}
|p-p_{\infty}|_{\infty} =|p-\bar{p}|_{\infty}\lem \|\xi_0\|_{H^4(\Omega)}\exp\{-\frac{\beta_7}{2} t\}, \\[8pt]
|u-0|_{\infty}= k_1|v|_{\infty}\lem \|\xi_0\|_{H^4(\Omega)}\exp\{-\frac{\beta_7}{2} t\}, \\[8pt]

|S(x,t)-S_{\infty}(x)|_{\infty} = |-\int\limits_{t}^{\infty}S_s(x,s) \,\mathrm{d}s|_{\infty}
\leq \int\limits_{t}^{\infty} |\phi_s(x,s)|_{\infty} \,\mathrm{d}s \\[8pt]\qquad

\lem \frac{2}{\beta_7}\|\xi_0\|_{H^4(\Omega)}\|\phi_0\|_{H^3(\Omega)}
\left(\exp\{\|\xi_0\|_{H^4(\Omega)}\}\right)^{{\frac{c_{18}}{2}}}\exp\{-\frac{\beta_7}{2} t\}, \\[10pt]

|\varrho(x,t)-\varrho_{\infty}(x)|_{\infty}\lem \exp\{-\frac{\beta_7}{2} t\}.
\end{array}\right.
\end{equation}

So $(p,u,S,\varrho)\rto (\bar{p},0,S_{\infty},\varrho_{\infty})$ exponentially in $|\cdot|_{\infty}$ norm as $t\rto +\infty$.
\end{proof}

The following theorem states that the pressure and velocity of non-isentropic Euler equations with damping converge to those of the diffusion equations respectively, thus the pressure and velocity have nonlinear diffusion property.
\begin{theorem}\label{Sect6_Diffusion_Thm}
Assume $(\hat{p},\hat{u},\hat{S},\hat{\varrho})$ are variables of the diffusion equations $(\ref{Sect1_Diffusion_Eq})$ and $(p,u,S,\varrho)$ are variables of non-isentropic Euler equations with damping $(\ref{Sect1_NonIsentropic_EulerEq})$, the initial data $(p_0,u_0,S_0)$ satisfy the conditions in Theorem $\ref{Sect2_Global_Existence_Thm}$, $(\hat{p}_0,\hat{S}_0)$ satisfy the conditions in Theorem
$\ref{Sect2_GlobalExistence_Thm_DiffusionEq}$.  
If 
\begin{equation}\label{Sect6_Diffusion_1}
\begin{array}{ll}
\int\limits_{\Omega}p_0^{\frac{1}{\gamma}}\,\mathrm{d}x
=\int\limits_{\Omega}\hat{p}_0^{\frac{1}{\gamma}}\,\mathrm{d}x,
\end{array}
\end{equation}
then
\begin{equation}\label{Sect6_Diffusion_2}
\begin{array}{ll}
\|p-\hat{p}\|_{H^3(\Omega)} + \|u-\hat{u}\|_{H^3(\Omega)} \leq C_1\exp\{-C_2 t\},
\end{array}
\end{equation}
for some positive $C_1,C_2$.
\end{theorem}

\begin{proof}
The condition $(\ref{Sect6_Diffusion_1})$ implies $\bar{p}=\hat{\bar{p}}$.

Take
$C_1 =4(1+k_1^2)\max\{\beta_1\|(\xi_0,v_0)\|_{H^3(\Omega)}^2, \beta_6 \|\xi_0\|_{H^4(\Omega)}^2\}$,
$C_2 =\min\{\beta_2, \beta_7\}$. By Lemmas $\ref{Sect3_Decay_Lemma}$ and $\ref{Sect5_Decay_Lemma}$, we have
\begin{equation}\label{Sect6_Diffusion_Prove1}
\begin{array}{ll}
\|p-\hat{p}\|_{H^3(\Omega)}^2
\leq 2\|p-\bar{p}\|_{H^3(\Omega)}^2 + 2\|\hat{p}-\hat{\bar{p}}\|_{H^3(\Omega)}^2 \\[6pt]\hspace{2cm}
\leq 2\beta_1\|(\xi_0,v_0)\|_{H^3(\Omega)}^2\exp\{-\beta_2 t\} + 2\beta_6 \|\xi_0\|_{H^4(\Omega)}^2\exp\{-\beta_7 t\},
\\[6pt]\hspace{2cm}
\leq \frac{C_1}{2}\exp\{-C_2 t\}, \\[6pt]
\|u-\hat{u}\|_{H^3(\Omega)}^2 \leq 2k_1^2\|v\|_{H^3(\Omega)}^2+ 2k_1^2\|\hat{v}\|_{H^3(\Omega)}^2
 \\[6pt]\hspace{2cm}
\leq 2k_1^2\beta_1\|(\xi_0,v_0)\|_{H^3(\Omega)}^2\exp\{-\beta_2 t\}
+ 2k_1^2\beta_6 \|\xi_0\|_{H^4(\Omega)}^2\exp\{-\beta_7 t\},
\\[6pt]\hspace{2cm}
\leq \frac{C_1}{2}\exp\{-C_2 t\}.
\end{array}
\end{equation}

Thus, Theorem $\ref{Sect6_Diffusion_Thm}$ is proved.
\end{proof}

\bibliographystyle{siam}
\addcontentsline{toc}{section}{References}
\bibliography{FuzhouWu_EulerEq_Diffusion}

\end{document}